\documentclass[11pt,oneside]{amsart}
\usepackage{amsmath, ifthen, amsfonts, amssymb, srcltx, amsopn, enumerate, amsthm, comment,mathabx,adforn} 

\usepackage[dvipsnames]{xcolor}

\usepackage[colorlinks=true, 
linkcolor=Periwinkle!110, citecolor=Periwinkle!110, linktoc=none]{hyperref}

\usepackage[all]{xy}
\usepackage{pb-diagram, pb-xy}
\usepackage{overpic}
\usepackage{tikzsymbols}
\usepackage{tikz-cd}
\usepackage{glossaries}
\usepackage{fontawesome}
\usepackage{lmodern}
\usepackage{pb-diagram}

\usepackage[T1]{fontenc}
\usepackage{stmaryrd}

\newcommand{\showcommentsbox}{yes}

\newsavebox{\commentbox}
\newenvironment{com}%
{\ifthenelse{\equal{\showcommentsbox}{yes}}%
	{\footnotemark
		\begin{lrbox}{\commentbox}
			\begin{minipage}[t]{1.25in}\raggedright\sffamily\tiny
				\footnotemark[\arabic{footnote}]}
			{\begin{lrbox}{\commentbox}}}%
			{\ifthenelse{\equal{\showcommentsbox}{yes}}%
				{\end{minipage}\end{lrbox}\marginpar{\usebox{\commentbox}}}
		{\end{lrbox}}}

\def\notorth{\:\ensuremath{\reflectbox{\rotatebox[origin=c]{90}{$\nvdash$}}}}

\usepackage[overload]{keytheorems}

\newkeytheoremstyle{linkedrestate}
{
	headformat={%
		\protect\IfRestatingTF
		{%
			\NAME \space \protect\hyperlink{restated #1 #2}{\NUMBER}\NOTE}
		{%
			\protect\hypertarget{restated #1 #2}{\NAME\thmnumber{ #2}}%
			\NOTE}},
}

\newkeytheorem{thm}[name=Theorem,parent=section, style=linkedrestate]
\newkeytheorem{cor}[name=Corollary,sibling=thm, style=linkedrestate]

\newcounter{ax}
\newtheorem{lem}[thm]{Lemma}

\newtheorem{conj}[thm]{Conjecture}

\newtheorem{prop}[thm]{Proposition}

\theoremstyle{definition}
\newtheorem{defn-inner}[thm]{Definition}
\newenvironment{defn}{%
	\def\qedsymbol{\large\adfast{1}}
	\pushQED{\qed}%
	\begin{defn-inner}%
	}{%
		\popQED
	\end{defn-inner}%
}

\newtheorem{exmp-inner}[thm]{Example}

\newtheorem{notation-inner}[thm]{Notation}
\newenvironment{notation}{%
	\def\qedsymbol{\large\adfast{1}}
	\pushQED{\qed}%
	\begin{notation-inner}%
	}{%
		\popQED
	\end{notation-inner}%
}

\newtheorem{claim}{Claim}

\newtheorem{claim*}{Claim}

\newtheorem{cons}[thm]{Construction}

\newtheorem{remark-inner}[thm]{Remark}
\newenvironment{remark}{%
	\def\qedsymbol{\large\adfast{1}}
	\pushQED{\qed}%
	\begin{remark-inner}%
	}{%
		\popQED
	\end{remark-inner}%
}

\DeclareMathOperator{\dimension}{dim}
\DeclareMathOperator{\relevant}{Rel}
\DeclareMathOperator{\Fix}{Fix}
\DeclareMathOperator{\id}{id}

\DeclareMathOperator{\image}{im}
\DeclareMathOperator{\rank}{rk}

\DeclareMathOperator{\Aut}{Aut}
\DeclareMathOperator{\Out}{Out}

\DeclareMathOperator{\stabilizer}{Stab}
\DeclareMathOperator{\diam}{\textup{\textsf{diam}}}

\DeclareMathOperator{\hull}{hull}

\DeclareMathOperator{\Min}{Min}

\newcommand{\neb}{\mathcal N}
\def\MCG{\mathcal{MCG}}

\DeclareMathOperator{\Commensurator}{Comm}
\newcommand{\commensurator}[2]{\Commensurator_{#1}({#2})}

\newcommand{\factor}[2]{{\raise0.7ex\hbox{$#1$} \!\mathord{\left/ {\vphantom {#1 {#2}}}\right.\kern-\nulldelimiterspace}\!\lower0.7ex\hbox{${#2}$}}}

\newcommand{\field}[1]{\mathbb{#1}}
\newcommand{\integers}{\ensuremath{\field{Z}}}

\newcommand{\naturals}{\ensuremath{\field{N}}}
\newcommand{\reals}{\ensuremath{\field{R}}}
\newcommand{\Euclidean}{\ensuremath{\field{E}}}

\newcommand{\boundary}{{\ensuremath \partial}}

\makeatletter

\newcommand{\Rmnum}[1]{\mathbf{{\expandafter\@slowromancap\romannumeral #1@}}}

\makeatother
\DeclareMathOperator{\Isom}{Isom}

\let\oldmarginpar\marginpar
\renewcommand\marginpar[1]{\-\oldmarginpar[\raggedleft\footnotesize #1]{\raggedright\footnotesize #1}}

\newcommand{\tsh}[1]{\left\{\kern-.7ex\left\{#1\right\}\kern-.7ex\right\}}
\newcommand{\Tsh}[2]{\tsh{#2}_{#1}}
\newcommand{\ignore}[2]{\Tsh{#2}{#1}}

\newcounter{enumitemp}

\newcommand{\dist}{d}

\newcommand{\cuco}[1]{{\mathcal #1}}

\newcommand{\fontact}{{\mathcal C}}

\newcommand{\gate}{\mathfrak g}

\usepackage{mathabx}
\newcommand{\propnest}{\sqsubsetneq}

\newcommand{\nest}{\sqsubseteq}

\newcommand{\orth}{\bot}
\newcommand{\transverse}{\pitchfork}

\newcommand{\la}{\langle}
\newcommand{\ra}{\rangle}
\newcommand{\hhscomp}{\chi}
\newcommand{\bigset}{\mathrm{Big}}

\newcommand{\oldcone}[1]{\widehat{#1}}
\newcommand{\newcone}[1]{\widecheck{#1}}
\newcommand{\injhull}{\mathrm{Env}}

\newcommand{\qline}{\Sigma}
\newcommand{\BU}{C'}

\usepackage{mathtools}
\usepackage{cleveref}

\theoremstyle{definition}

\DeclareMathOperator{\support}{supp}
\newcommand{\pcollapse}{{10(EM + E)}}
\newcommand{\Morse}{\textup{\textsf{Morse}}}

\newcounter{tocsubsection}[section]
\newcommand{\newsection}[1]{\hypertarget{#1}{\section{#1}}}

\let\oldsubsection=\subsection
\renewcommand{\subsection}[1]{\stepcounter{tocsubsection} \hypertarget{\thesection.\thetocsubsection}{\oldsubsection{#1}}}

\let\oldtocsection=\tocsection
\let\oldtocsubsection=\tocsubsection
\renewcommand{\tocsection}[3]{\hyperlink{#3}{\large\oldtocsection{#1}{#2}{#3}}\vspace{2pt}}
\renewcommand{\tocsubsection}[2]{\hspace{1.9em}\hyperlink{#1#2}{\oldtocsubsection{#1}{#2}}\vspace{1pt}}

\newenvironment{proofofclaim}[1]{
	\begin{proof}[Proof of Claim~#1] \renewcommand{\qedsymbol}{$\blacksquare$}}
	{\end{proof}}

\title[Periodic quasiflats]{Periodic quasiflats in hierarchically \\ hyperbolic spaces}
\author{Pénélope Azuelos}
\address{School of Mathematics, University of Bristol, Bristol BS8 1UG, UK}
\email{penelope.azuelos@bristol.ac.uk}
\author{Mark Hagen}
\address{School of Mathematics, University of Bristol, Bristol BS8 1UG, UK}
\email{markfhagen@posteo.net}
\date{\today}

\begin{document}
	
	\maketitle 
	
	\begin{abstract}
        We prove a quasiflat closing theorem and a coarse flat torus theorem for hierarchically hyperbolic groups (HHGs). Namely, given an HHG $G$, we prove that $G$ is hyperbolic if and only if it contains no $\mathbb Z^2$ subgroups and, if $A\leq G$ is virtually $\integers^n$, then there is an $A$--invariant $n$--dimensional uniform quality quasiflat $F$ such that any two points in $F$ are joined by a uniform-quality hierarchy path lying in $F$.  The later is a consequence of a more detailed theorem describing a ``coarse minset'' for $A$ in $G$, which has various applications, including an ascending chain condition for virtually abelian subgroups, hierarchical quasiconvexity of highest abelian subgroups, and some geometric control over normalisers, centralisers, and commensurators of abelian subgroups.  We use this to rule out HHG structures for certain Coxeter groups on the basis of their affine subgroups, and to give a new proof that virtually solvable subgroups of HHGs are virtually abelian, which simplifies the original proof by avoiding Gromov's polynomial growth theorem.\end{abstract}

	\tableofcontents
	
	\newsection{Introduction}\label{sec:intro}
	A common feature among spaces with some form of (coarse) non-positive curvature is that virtually abelian subgroups of their isometry groups are often particularly well-behaved. Here are two classical illustrations of this phenomenon. First, assume $X$ is a Gromov-hyperbolic space. If $A$ is a virtually abelian group acting by semisimple isometries on $X$, then there is either an $A$--invariant, uniformly bounded subset of $X$ or an $A$-invariant uniform quality quasiline on which $A$ acts cocompactly. In particular, if the action of $A$ is proper, then $A$ is virtually cyclic with undistorted orbits. 
    Second is the Flat Torus Theorem for CAT(0) spaces, which states (in part): if $X$ is a CAT(0) space and $A$ is a virtually abelian group acting on $X$ by semisimple isometries then there is an $A$--invariant closed, convex flat $F \subseteq X$ on which $A$ acts cocompactly. At this level of generality, this was proven in \cite{BridsonHaefliger:metric}, generalising earlier results about Hadamard manifolds \cite{GromollWolf,LawsonYau}. 

    Among the CAT(0) spaces widely studied in group theory are CAT(0) cube complexes.  Chepoi famously characterised CAT(0) cube complexes as those cubical complexes whose $1$--skeletons are median graphs, and, in fact, if $X$ is a CAT(0) cube complex, then equipping each cube with the $\ell_1$ metric yields a path metric $\dist$ that is median, and extends the median metric on the $0$--skeleton given by Chepoi's theorem \cite{chepoiGraphsCATComplexes2000,mieschInjectiveMetricsCube2014}.  A subspace $Y$ of $X$ that is closed under taking medians is a \emph{median subalgebra} (the terminology reflects the correspondence between CAT(0) cube complexes and discrete median algebras \cite{Roller}).  If the median subalgebra $Y$ is a connected subcomplex, then $Y$ is itself a CAT(0) cube complex, and is \emph{path-isometric} in the sense that any two points in $Y$ are joined by a geodesic of $X$ lying in $Y$ \cite[Lem. 2.11]{HagenPetyt:BBF}.

    Haglund showed in \cite{Haglund:semisimple} that any $g\in\Isom(X,\dist)$ either fixes the barycentre of some cube, or has a $\dist$--geodesic axis $F$ --- i.e. a connected median subalgebra isometric to $\reals$ --- along which it acts as a translation\footnote{One must subdivide to arrange for there to be such an axis in the $1$--skeleton.}.  This can be viewed as a $1$--dimensional flat torus theorem for $(X,\dist)$, since the $\langle g\rangle$--action on $F$ is cocompact.
    The situation becomes more delicate for higher-rank (virtually) abelian subgroups $A\leq \Isom(X,d)$.  On one hand, in \cite{Woodhouse:axis}, Woodhouse generalised Haglund's result by showing that if $A$ is virtually $\integers^n$, and the $A$--action is proper, then (up to subdividing to remove inversions) $A$ stabilises a path-isometric subcomplex $Y\subseteq X$, where $Y$ is cubically isomorphic to $\prod_{i=1}^mC_i$ for $m\geq n$ and $C_1,\ldots,C_m$ a collection of CAT(0) cube complexes quasi-isometric to $\reals$.  Note that Woodhouse's result does not assert that $A$ acts on $Y$ cocompactly, and indeed it may be that no such $Y$ is cocompact.  It is implicit in Woodhouse's proof that $Y$ can be taken to be a median subalgebra.  In fact, Genevois elaborates on Woodhouse's result in \cite{Genevois:axis}, using the language of invariant median subalgebras.  However, any theorem providing an $A$--invariant median subalgebra must differ from the Flat Torus Theorem in one quite serious respect: one cannot in general expect $A$--cocompactness.  

    \begin{figure}
        \centering
       \scalebox{0.7}[0.8]{
				\centering
				\begin{tikzpicture}


					
					\draw[color=black!40,-stealth] (-3,-0.5) -- (3,0.5);
					\draw[color=black!40,-stealth] (0,-3) -- (0,3);
					\draw[color=black!40,-stealth] (-2.2,1.5) -- (2.2,-1.5);
					
					
					\draw[color=black!40] (-3,-0.5) -- (3,0.5);
					\draw[color=black!40] (-5.2,4) -- (0.8,5);
					\draw[color=black!40] (-5.2,1) -- (0.8,2);
					\draw[color=black!40] (-5.2,-2) -- (0.8,-1);
					\draw[color=black!40] (-3,2.5) -- (3,3.5);
					\draw[color=black!40] (-3,-3.5) -- (3,-2.5);
					\draw[color=black!40] (-0.8,1) -- (5.2,2);
					\draw[color=black!40] (-0.8,-5) -- (5.2,-4);
					\draw[color=black!40] (-0.8,-2) -- (5.2,-1);

					\draw[color=black!40] (-3,2.5) -- (-3,-3.5);
					\draw[color=black!40] (-5.2,4) -- (-5.2,-2);
					\draw[color=black!40] (-0.8,1) -- (-0.8,-5);
					\draw[color=black!40] (-2.2,4.5) -- (-2.2,-1.5);
					\draw[color=black!40] (2.2,1.5) -- (2.2,-4.5);
					\draw[color=black!40] (3,3.5) -- (3,-2.5);
					\draw[color=black!40] (0.8,5) -- (0.8,-1);
					\draw[color=black!40] (5.2,2) -- (5.2,-4);

					\draw[color=black!40] (-5.2,4) -- (-0.8,1);
					\draw[color=black!40] (-5.2,1) -- (-0.8,-2);
					\draw[color=black!40] (-5.2,-2) -- (-0.8,-5);
					\draw[color=black!40] (-2.2,4.5) -- (2.2,1.5);
					\draw[color=black!40] (-2.2,-1.5) -- (2.2,-4.5);
					\draw[color=black!40] (0.8,5) -- (5.2,2);
					\draw[color=black!40] (0.8,2) -- (5.2,-1);
					\draw[color=black!40] (0.8,-1) -- (5.2,-4);

					
					\draw[fill=MidnightBlue!70,draw=MidnightBlue!35] (-2.2,4.5) -- (-3,2.5) -- (-0.8,-2) -- (2.2,-4.5) -- (3,-2.5) -- (0.8,2) -- (-2.2,4.5);
					
					
					\draw[color=MidnightBlue!35] (-3+2.2/3,1) -- (-1.2,4.5-2.5/3);
					\draw[color=MidnightBlue!35] (-3+4.4/3,-0.5) -- (-0.2,4.5-5/3);
					\draw[color=MidnightBlue!35] (-0.8,-2) -- (0.8,2);
					\draw[color=MidnightBlue!35] (0.2,-2-2.5/3) -- (0.8+2.2/3,2-4.5/3);
					\draw[color=MidnightBlue!35] (1.2,-2-5/3) -- (0.8+4.4/3,-1);
					
					\draw[color=MidnightBlue!35] (-3+0.8/3,2.5+2/3) -- (-0.8+1,-2-2.5/3);
					\draw[color=MidnightBlue!35] (-3+1.6/3,2.5+4/3) -- (-0.8+2,-2-5/3);
					\draw[color=MidnightBlue!35] (-2.2,4.5) -- (2.2,-4.5);
					\draw[color=MidnightBlue!35] (-2.2+1,4.5-2.5/3) -- (2.2+0.8/3,-4.5+2/3);
					\draw[color=MidnightBlue!35] (-2.2+2,4.5-5/3) -- (2.2+1.6/3,-4.5+4/3);
					
					\draw[color=MidnightBlue!35] (-2.2-0.8/3,4.5-2/3) -- (0.8+2.2/3,2-4.5/3);
					\draw[color=MidnightBlue!35] (-2.2-1.6/3,4.5-4/3) -- (0.8+4.4/3,2-3);
					\draw[color=MidnightBlue!35] (-3,2.5) -- (3,-2.5);
					\draw[color=MidnightBlue!35] (-3+2.2/3,2.5-4.5/3) -- (3-0.8/3,-2.5-2/3);
					\draw[color=MidnightBlue!35] (-3+4.4/3,2.5-3) -- (3-1.6/3,-2.5-4/3);

					
					\draw[color=black!40] (0,0) -- (0,3);
					\draw[color=black!40] (0,0) -- (3,0.5);
					\draw[color=black!40,-stealth] (0,0) -- (2.2,-1.5);
					\draw[color=black!40] (2.2,1.5) -- (2.2,-4.5);
					\draw[color=black!40] (-0.8,-2) -- (5.2,-1);
					\draw[color=black!40] (-0.8,1) -- (-0.8,-5);
					\draw[color=black!40] (-3,2.5) -- (-0.8,1);
					\draw[color=black!40] (-0.8,1) -- (5.2,2);
					\draw[color=black!40] (-3,2.5) -- (0,3);
					
			\end{tikzpicture}}
	
        \caption{The $(3,3,3)$ Coxeter group acts on the standard tiling $X$ of $\Euclidean^3$ by unit cubes, preserving a $2$--dimensional flat $F$ that crosses every hyperplane.  The smallest median subalgebra of $X$ containing $F$ is the whole of $X$.}
        \label{fig:diagonal-flat}
    \end{figure}
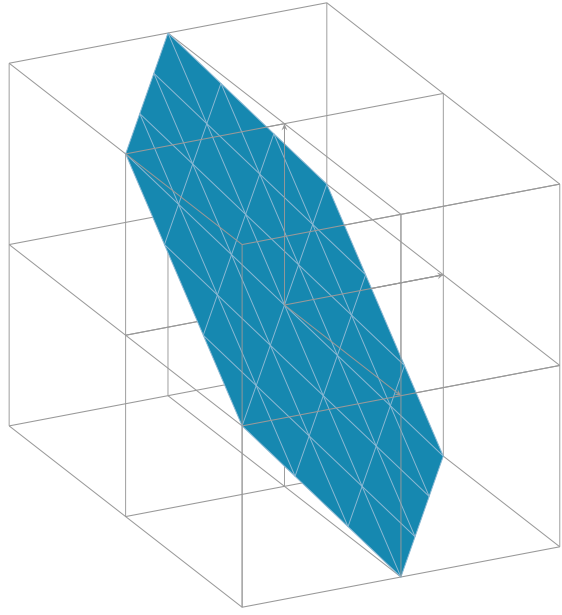

    Indeed, Figure \ref{fig:diagonal-flat} illustrates the case where $A$ is the $(3,3,3)$ Coxeter group, and $X$ is the standard tiling of $\Euclidean^3$ by unit $3$--cubes; the action of $A$ on $X$ comes from applying Sageev's construction to the set of walls in $\Euclidean^2$ provided by the $A$--translates of the three reflection lines corresponding to the generators of $A$.  Applying the CAT(0) flat torus theorem, and using the general fact that CAT(0)-convex subsets of a CAT(0) cube complex are $\dist$--isometric (though not in general $\dist$--convex) yields an $A$--cocompact $\dist$--isometric subspace $F\subseteq X$ which, equipped with the CAT(0) metric, is isometric to $\Euclidean^2$.  Moreover, by perturbing $F$ equivariantly, one can make it into an $A$--cocompact $2$--dimensional $\dist$--isometric subcomplex $Y$, homeomorphic to $\reals^2$.  However, as illustrated in Figure \ref{fig:diagonal-flat}, which shows $F$, no median subalgebra of $X$ is Hausdorff-close to $F$. 

    Nonetheless, under reasonable conditions, one can find $A$--cocompact bilipschitz flats in the median setting by relaxing the median subalgebra requirement, using instead $\dist$--isometricity.  This is a consequence of the CAT(0) result, and the proof is short enough to present here, to motivate the subsequent discussions.

    \begin{prop}\label{prop:ample-invariant-flat}
		Let $(Q,\dist,\mu)$ be a complete, connected, finite-rank median metric space and let $A\leq \Isom(Q)$ be a finitely generated virtually abelian subgroup acting properly, and let $r$ be the rank of $A$.  Then there is a proper cocompact action of $A$ on $\Euclidean^r$ and an $A$--equivariant bilipschitz embedding $f:\Euclidean^r\to Q$ such that any two points in $f(\Euclidean^r)$ are joined by a $\dist$--geodesic lying in $f(\Euclidean^r)$.
	\end{prop}
	
	\begin{proof}
		By \cite[Thm. 1.1, Thm. 8.3]{BowditchProperties16}, there is a complete CAT(0) metric $\sigma$ on $Q$ and a bilipschitz homeomorphism $h:(Q,\sigma)\to (Q,\dist)$ that sends geodesics to reparameterised geodesics (as a map of sets, $h$ is just the identity).  Moreover, $\Isom(Q,\dist)$ acts on $(Q,\sigma)$ by isometries and $h$ is $\Isom(Q,\dist)$--equivariant.  By, for instance, \cite[Cor. A]{Fioravanti:core}, infinite order elements of $A$ have positive stable translation length on $(Q,\dist)$, so since $h$ is a bilipschitz map, $A$ acts on $(Q,\sigma)$ semisimply.  Hence the Flat Torus Theorem \cite[Cor. II.7.2]{BridsonHaefliger:metric} provides a cocompact $A$--action on $\Euclidean^r$ and an $A$--equivariant isometric embedding $e:\Euclidean^r\to (Q,\sigma)$ with $\sigma$--convex image, and we conclude by taking $f=h\circ e$.  
	\end{proof} 

    \begin{remark}\label{rem:intro-proper}
    We emphasise that $Q$ need not be a locally compact space, and throughout this paper, we rarely make such an assumption.  Hence, wherever we talk about \emph{proper} maps or \emph{proper} actions, we are referring to metric properness.    
    \end{remark}

     When $Q$ is a CAT(0) cube complex, the map $f$ from Proposition \ref{prop:ample-invariant-flat} can be chosen so that $f(\Euclidean^r)$ is a $\dist$--isometric subcomplex of $Q$.  We view Proposition \ref{prop:ample-invariant-flat} as yet another reason to study such subcomplexes of CAT(0) cube complexes, which have recently received renewed attention in the very interesting paper \cite{BandeltChepoiDressKoolen} on so-called \emph{ample} subsets of $Q^{(0)}$, which correspond to $0$--skeletons of $\dist$--isometric subcomplexes.\footnote{A particularly attractive feature of the notion is that it can be characterised in non-metric terms, provided $Q$ is a finite-dimensional CAT(0) cube complex: the subcomplex $Y\subseteq Q$ is $\dist$--isometric if and only if $Y\cap h_1\cap\cdots\cap h_k$ is empty or connected whenever $h_1,\ldots,h_k$ are hyperplanes of $Q$.} In some sense, the present paper is about a ``coarse-ification'' of $\dist$--isometricity in the setting of hierarchically hyperbolic spaces. \\
	
	Hierarchically hyperbolic spaces (HHSs) were introduced in \cite{BehrstockHierarchically17,BehrstockHierarchically19} to capture the similarities in the large-scale structure of mapping class groups and compact special groups. These spaces are, in essence, highly organised examples of coarse median spaces, in the sense of Bowditch \cite{BowditchCoarse13}. The class of HHSs includes all hyperbolic spaces and is closed under taking products and under relative hyperbolicity (i.e. if $X$ is hyperbolic relative to HHS subspaces then $X$ is an HHS) \cite[Prop. 8.27, Thm. 9.3]{BehrstockHierarchically19}, but one can weaken the bounded penetration property of relative hyperbolicity to allow for the peripheral subspaces to interact in a coarse product structure, allowing the class to accommodate, for example, all right-angled Artin groups. See \cite{BehrstockHierarchically19} for the formal definition in its modern form. 
    
    The fact that HHSs are coarse median spaces is witnessed by the (strictly stronger) cubical approximation theorem \cite[Thm. 2.1]{BehrstockQuasiflats21}, which states: for any $n$ points $\{x_1, \dots, x_n\}$ in an HHS, there is a coarsely median-preserving quasi-isometric embedding, with constants depending only on the HHS structure and $n$, from a CAT(0) cube complex to the hull (see Subsection \ref{subsec: hulls}) of $\{x_1, \dots, x_n\}$, such that each $x_i$ is in the image of the 0-skeleton. In a large subclass of HHSs, this local property becomes global: Petyt showed that any colourable hierarchically hyperbolic group (HHG) admits a quasi-isometry to a CAT(0) cube complex \cite{Petyt:mapping}.
	
	Geodesics in HHSs are not in general closely related to the HHS structure.
	Instead, one usually works with hierarchy paths, which are  quasigeodesics that do respect the HHS structure (see Definition \ref{defn:hier-path}). 
	Hierarchical quasiconvexity is defined in terms of these paths: a subspace $Y$ of an HHS $X$ is hierarchically quasiconvex (HQC) if every hierarchy path between points in $Y$ is close to $Y$ (of course, this should be quantified appropriately; see Definition \ref{defn:HQC}). 
	Hierarchy paths are analogues of $\ell^1$-geodesics in CAT(0) cube complexes, so hierarchical quasiconvexity is the analogue of median convexity in CAT(0) cube complexes. In light of the preceding discussion, a flat torus theorem for HHSs will need to involve a notion of convexity which imitates CAT(0) convexity, or $\ell_1$--metric isometricity, instead. 
    
    For this purpose, we introduce the following notion: a subspace $Y$ of an HHS $X$ is \textit{existentially hierarchically quasiconvex (EHQC)} if, for all $x,y \in Y$ there is a hierarchy path in a uniform neighbourhood of $Y$ which joins $x$ to $y$ (again, this should be parametrised correctly; see Definition \ref{defn:EHQC}). 
	As in the results about CAT(0) and hyperbolic spaces, we also need a semisimplicity assumption, provided by Definition \ref{defn:hierarchically-semisimple}.

    A simplified version of our ``HHS flat torus theorem'' is the following, which answers the special cases of Questions 6.1 and 6.2 in \cite{abbottUniformUndistortionBarycentres2025} concerning HHGs:
	
	\begin{thm}[Theorem \ref{thm:HHS-semisimple-coarse-minset}.\eqref{item:main-one-cocompact-quasiflat-1}]\label{thm:intro-main}
        Let $n \in \mathbb N$ and let $A$ be a virtually $\mathbb Z^n$ group. Suppose $A$ is a subgroup of an HHG $(X, \mathfrak S)$
		--- or, more generally, that $A$ acts properly and hierarchically semisimply by HHS automorphisms on an HHS $(X,\mathfrak S)$. Then there is a proper cocompact action of $A$ on the Euclidean space $\mathbb E^n$, an $A$--invariant EHQC subspace $F \subseteq X$, and an $A$--equivariant quasi-isometry $F \rightarrow \mathbb E^n$. Moreover, the EHQC and quasi-isometry constants depend only on $X$.
	\end{thm}
	
	We prove a more detailed statement in Theorem \ref{thm:HHS-semisimple-coarse-minset}. In particular, we describe the coarse structure of the union of a natural family of $A$-invariant, uniform quality EHQC quasiflats, proving that this subspace is $A$-equivariantly quasi-isometric to a product $Y \times C \times \mathbb E^n$, where $Y$ is an injective space with a trivial $A$-action and $C$ is a CAT(0) space with a trivial $A$-action.  In the case where $X$ is the Cayley graph of an HHG $G$ and $A\leq G$, this subspace can be taken to be the normaliser of $A$. Without the hierarchical semisimplicity assumption, the theorem fails, as is already evident in the case where $X$ is hyperbolic and $A$ is generated by a single parabolic isometry. 
	
	In \cite[Thm. 3.15]{FournierMangioniSist:extensions}, Fournier-Facio, Mangioni and Sisto prove that, if $G$ is a hierarchically hyperbolic group (HHG), $K$ is finitely generated and $K \hookrightarrow G \twoheadrightarrow E$ is a central extension, then the corresponding element of $H^2(E,K)$ is represented by a bounded cocycle. We generalise this to the setting of hierarchically semisimple actions on HHSs (in particular, the following theorem applies to all subgroups of HHGs).

	\getkeytheorem{central extensions}

	Gersten showed \cite{GerstenBounded92} that if $K \hookrightarrow G \twoheadrightarrow E$ is a central extension represented by a bounded cocycle, then $G$ is naturally quasi-isometric to the direct product $K \times E$, but the converse to this false in general \cite{FrigerioSisto}, even for CAT(0) groups \cite[Thm. 1.1]{AscariMilizia}.
	
	\begin{remark}[Virtual direct products]
		The CAT(0) Flat Torus Theorem includes a stronger version of Theorem \ref{thm:B-central-extension} in the free abelian case: suppose $B$ is a finitely generated group acting by isometries on a CAT(0) space $X$ and suppose $A \unlhd B$ is a normal subgroup whose action on $X$ is proper and by semisimple isometries. Then there is a finite index subgroup $B' \leq B$ which contains $A$ as a direct factor \cite[Thm. II.7.1.(5)]{BridsonHaefliger:metric}.  

        The corresponding statement for actions on HHSs cannot hold, as we now illustrate using mapping class groups.  Let $\Sigma$ be a connected orientable surface of finite type, and suppose that $\Sigma$ is either closed with genus $\geq 3$, or has genus $\geq 2$ and at least one boundary component or two punctures.  Then the mapping class group $\MCG(\Sigma)$ contains finitely generated subgroups that split as central extensions but are not virtually direct products (see \cite[Thm. II.7.26]{BridsonHaefliger:metric}). Since mapping class groups are HHGs, and in particular act properly and hierarchically semisimply by HHS automorphisms on HHSs, this implies that Theorem \ref{thm:B-central-extension} is optimal in the sense that the conclusion cannot be improved to give virtually split extensions as in the CAT(0) case.
	\end{remark}
	
	As in the CAT(0) case, our version of the flat torus theorem has many applications. To begin with, HHGs (in fact, a more general class of groups) satisfy the ascending chain condition for virtually abelian subgroups:
	
	\getkeytheorem{cor:ascending chains}
	
	Given $n \in \mathbb N$, a virtually $\mathbb Z^n$ subgroup $A \leq G$ of a group $G$ is \textit{highest} if no finite index subgroup of $A$ is contained in a subgroup of $G$ isomorphic to $\mathbb Z^{n+1}$.

    In \cite{DurhamBoundaries17} it is stated that, if $G$ is a hierarchically hyperbolic group, then any $H\leq G$ either contains $F_2$ or is virtually abelian; the proof given in \cite{DurhamBoundaries17} had a gap which is corrected in \cite{DurhamCorrection20} assuming finite generation of $H$; Corollary \ref{cor: ascending chain} allows one to remove the finite generation hypothesis (apply the result from \cite{DurhamCorrection20} to finitely generated subgroups $H_n\leq H\leq G$ with $\bigcup_nH_n=H$).  Hence the version of the Tits alternative stated in \cite[Thm. 9.15]{DurhamBoundaries17} holds as stated. 

    Relatedly, we give a new proof that virtually solvable subgroups of HHGs are virtually abelian which, unlike that in \cite{DurhamBoundaries17,DurhamCorrection20}, avoids Gromov's polynomial growth theorem:
    
    \getkeytheorem{cor:solvable}

    Recall that, in the CAT(0) cubical setting, abelian subgroups of cubulated groups may not be convex-cocompact, but Wise--Woodhouse showed \cite{WoodhouseWise:cubical} that, if $G$ acts geometrically on a CAT(0) cube complex $X$, and $A\leq G$ is a highest virtually abelian subgroup, then there is a convex, $A$--cocompact subcomplex $Y\subseteq X$.  In exact analogy, we show:

    \getkeytheorem{cor:highest-abelian}

    In the important case where $G$ is the mapping class group of a finite-type hyperbolic surface, abelian subgroups were well-understood even before the introduction of curve graphs, subsurface projections, and the ``hierarchical'' viewpoint.  For example, Corollary \ref{cor:solvable}, and the upper bound on rank for free abelian subgroups, was established in \cite{BirmanLubotzkyMcCarthy} using Thurston's classification of mapping classes, systems of reducing curves, etc.  Although the notions of HQC and EHQC subspaces/subgroups cannot be defined without first introducing some ``hierarchical'' ideas --- enough to be able to talk about coarse medians, for instance --- the extra input of Thurston theory seems to simplify proofs like that of, say, Corollary \ref{cor:highest-abelian} in the mapping class group case.  However, as pointed out to us by Jason Behrstock, our results recover a useful information about virtually abelian subgroups in the mapping class group case without recourse to the powerful Thurston classification.\footnote{There is a limited analogue of the Thurston classification for general HHGs, which does play an important role in this paper, that goes back to the classification of HHS automorphisms/``coarse rank rigidity'' in \cite{DurhamBoundaries17} and has its most modern form in \cite[Thm. 5.1]{PetytUnbounded23}.  However, this is a statement about an action on a hierarchical structure, and doesn't rely on anything specific to surfaces, or even directly imply the Thurston classification in the MCG case.}

    Corollary \ref{cor:highest-abelian} relies on a stronger result, Corollary \ref{cor:normaliser}, showing that the normaliser $N_G(A)$ is EHQC in $G$ and that there is a finite-index subgroup $A_1\leq A$ whose centraliser $C_G(A_1)$ is hierarchically quasiconvex (which, recall, is stronger than EHQC).  Corollary \ref{cor:highest-abelian} combines with results in \cite{HaettelCoarse23,hodaCrystallographicHellyGroups2023} to limit the possible point groups of highest virtually abelian subgroups of HHGs, which we make precise in Corollary \ref{cor:abelian-subgroups-octahedral}.  In Corollary \ref{cor:coxeter-group}, we apply this to Coxeter groups; specifically, we rule out HHG structures on Coxeter groups containing certain ``poison'' affine subgroups.  We suspect that Coxeter groups without ``poison'' affine subgroups are HHGs, and that this is closely related to the question of cocompact cubulation.  It is not currently known exactly which Coxeter groups are cocompactly cubulated, although progress on a characterisation is provided by \cite{NibloReeves:coxeter,Williams:coxeter,FioravantiLevcovitzSageev,HaglundWise:coxeter}.  In view of Remark \ref{rem:crystallographic-obstruction}, Corollary \ref{cor:coxeter-group}, and existing results, we conjecture the following.

\begin{conj}\label{conj:coxeter}
	A Coxeter group is an HHG if and only if it is cocompactly cubulated.
\end{conj}

    Corollary \ref{cor:normaliser} and Lemma \ref{lem: HHS structure on HQC subspaces} combine to yield another interesting consequence:

    \begin{cor}\label{cor:intro-independence}
    Let $(G,\mathfrak S)$ be a hierarchically hyperbolic group.  Then there exists $N<\infty$, depending on the group $G$ but not the HHG structure, such that the following holds.  Let $A\leq G$ be a virtually abelian subgroup.  Then there is a free abelian subgroup $A_1\leq A$ such that $[A:A_1]\leq N$ and $C_G(A_1)$ is hierarchically quasiconvex with respect to any HHG structure on $G$, and in particular, $C_G(A_1)$ admits an HHS structure where the left multiplication action of $C_G(A_1)$ on itself is an action by HHS automorphisms.
    \end{cor}

	  Motivation for Corollary \ref{cor:intro-independence} comes from the fact that a given finitely generated subgroup $H\leq G$ may by HQC or not, depending on how we choose the HHG structure, which, in general, can be done in many ways; see \cite{Mangioni:continuum} and also \cite[Thm. 5.15, Rem. 5.16]{abbottUniformUndistortionBarycentres2025}.

      Passing to the finite-index subgroup $A_1\leq A$ in Corollary \ref{cor:intro-independence} is needed in order to rule out the possibility that $aV\orth V$ for certain $V\in\mathfrak S$, where $a\in A$.  In particular, if $(G,\mathfrak S)$ has the global property that elements of $\mathfrak S$ are not orthogonal to their $G$--translates, then we get hierarchical quasiconvexity of centralisers.  An emblematic example of such a situation is when $G$ is the fundamental group of a compact special cube complex, with the HHG structure from \cite{BehrstockHierarchically17}.

      \begin{cor}\label{cor:intro-centraliser}
    Let $(G,\mathfrak S)$ be a hierarchically hyperbolic group and suppose that $gU\notorth U$ for all $g\in G$ and $U\in\mathfrak S$.  Let $A$ be a free abelian subgroup.  Then $C_G(A)$ is hierarchically quasiconvex.
     \end{cor}

      This is proved in Section \ref{sec:applications}, as part of the proof of Corollary \ref{cor:normaliser}.  The hypothesis that elements of $\mathfrak S$ are not orthogonal to their $G$--translates holds in a finite-index subgroup whenever $G$ is \emph{colourable} in the sense of \cite{DurhamMinskySisto:stable}; while there are non-colourable HHGs \cite{hagenNoncolorableHierarchicallyHyperbolic2023}, many of the natural examples of HHGs, like mapping class groups \cite{bestvinaConstructingGroupActions2015} and compact special groups \cite{behrstockCombinatorialTakeHierarchical2024b}, are colourable.

      \begin{remark}
      The preceding corollary is related to work of Genevois \cite[Thm. 5.1, Lem. 5.2]{Genevois:axis}, which shows that, whenever $G$ acts geometrically on a cube complex $X$, then, up to subdividing $X$ once, for each infinite order $g\in G$, there is an $\ell_1$--isometric subcomplex $Min(g)$ of $X$ --- a connected median subalgebra --- stabilised cocompactly by $C_G(\langle g\rangle)$.  Now, if $G$ also acts cospecially, then, as shown in \cite{Genevois:centralisers}, $Min(g)$ coincides with the \emph{stable} minset $SMin(g)$ which, by \cite[Rem. 3.7]{Genevois:centralisers}, is convex, giving cubical convexity (and hence hierarchical quasiconvexity) of centralisers.  See also \cite[Prop. 9.6]{HuangPrytula:commensurators}.
      \end{remark}
	
	We finally discuss the structure of commensurators of virtually abelian subgroups in an HHG $(G,\mathfrak S)$, inspired by the corresponding discussion of commensurators of abelian subgroups of CAT(0) groups from \cite{HuangPrytula:commensurators}. We prove:

    \getkeytheorem{cor:commensurators}

    An observation in \cite{HuangPrytula:commensurators} about complete square complexes (CSCs) shows that this result is sharp, in that one cannot in general replace the finitely generated subgroup $K$ with the entire commensurator $\commensurator{G}{A}$.  Now, in \cite[Cor. 5.10]{HuangPrytula:commensurators}, Huang--Prytu{\l}a show that, if $G$ is compact special (in particular, both a CAT(0) group and an HHG), $\commensurator{G}{A}$ is finitely generated and coincides with $N_G(A')$ for some finite-index $A'\leq A$.  
    
    When $G$ is a mapping class group, \cite[Thm. 1.1, Rem. 1.2]{RollandJimenezSaldana} shows that, if $A\leq G$ is a virtually abelian subgroup, then $\commensurator{G}{A}$ coincides with $N_G(A')$, where $A'$ is a finite-index with $A$.  This follows from Corollary \ref{cor:commensurators} once one knows that $\commensurator{G}{A}$ is finitely generated, which is shown for mapping class groups in Corollary 4.6 in loc. cit., but which does not hold for general HHGs.

    In various recent papers, some attention is given to the possibility of a theory of ``algebraically well-behaved'' HHGs, encompassing compact special groups and mapping class groups.  Roughly speaking, such a theory should involve imposing conditions ensuring that various HQC subgroups are separable, and that coarse products of subgroups must arise from a certain amount of commutation.  For instance, the idea of requiring separability of product region subgroups, motivated by mapping class groups \cite{LeiningerMcReynolds}, is explored in \cite[Prop. 3.2]{HagenPetyt:BBF}.  The correct commutation condition may turn out to vary according to the exact applications one wants, but the notions of ``weakly commutative'' \cite[Defn. 3.4]{OhPark:embedding}, ``decomposable'' \cite{DurhamMinskySisto:farrell-jones} or ``algebraic'' \cite{Casals-RuizReal24} HHGs, or the conditions from \cite{AbbottBehrstock:conj} are all candidate conditions aimed at excluding colourable HHGs whose algebraic properties do not reflect their geometry in straightforward ways, like irreducible CSC groups.  We wonder if imposing the condition that abelian subgroups have finitely generated commensurators is a useful move in this direction.

    \subsection{Quasiflat closing}
    The famous Flat Closing Problem for CAT(0) spaces asked whether a group $G$ acting geometrically on a CAT(0) space $X$ with an isometrically embedded $\Euclidean^2$ must contain a $\integers^2$ subgroup.  This was spectacularly resolved by the counterexample provided by Martelli in the recent paper \cite{MartelliFlatClosing2025}.

    In contrast, we prove the following positive result for hierarchically hyperbolic groups:

    \getkeytheorem{higher-rank}

    A special case of this, Theorem \ref{thm: flat closing}, says that any HHG is either hyperbolic or contains a $\integers^2$ subgroup.  In Section \ref{subsec:quasiflat-closing}, we first present a proof of the latter theorem, then generalise the argument to prove Theorem \ref{thm:higher-rank}.

   For CAT(0) cube complexes, Theorem \ref{thm: flat closing} and the results in \cite{BehrstockHierarchically17} immediately yield:

    \begin{cor}
        Flat closing holds for CAT(0) cube complexes which admit a factor system.
    \end{cor}

    It is unknown whether flat closing holds for general cocompactly cubulated groups (not all of which admit factor systems \cite{shepherdCubulationNoFactor2023}), but there were some previously known positive results.  For example, it is well known that if $G$ is virtually compact special in the sense of \cite{HaglundWise:special}, then $G$ is hyperbolic or contains $\integers^2$.  Work of Sageev--Wise \cite{SageevWise:flat-closing} yields the same conclusion if $G$ acts geometrically on a CAT(0) cube complex satisfying a condition that limits intersections between certain sets of hyperplane stabilisers (and is therefore somewhat similar in spirit to the existence of a factor system).

    Theorem \ref{thm:higher-rank} follows from Corollary \ref{cor:highest-abelian} together with an argument that is elementary, in the sense that it uses only the core concepts of the theory of HHSes.  The key observation is that Corollary \ref{cor:highest-abelian} reduces the problem to that of finding some $g\in G$ such that there are at least two distinct $U\in\mathfrak S$ for which $\pi_U(\langle g\rangle)$ is unbounded.  In the case of a cube complex with a factor system, this can also be deduced from \cite{WoodhouseWise:cubical}.  The factor system case also follows by combining \cite[Cor. D]{CapraceSageev} and  \cite[Prop. 2.7]{HagenSusse}; however, it appears that this statement does not appear elsewhere in the literature.
	
    \subsection{Sketch of the proof}\label{subsec:intro-proof-discussion}
    Here is a summary of the proof of Theorem \ref{thm:intro-main}.

    As in the proof of \cite[Prop. 2.17]{HRSS:3-manifold}, we first find pairwise-orthogonal $U_1,\ldots,U_n\in\mathfrak S$ such that $\{U_1,\ldots,U_n\}$ is $A$--invariant, and any $W\in\mathfrak S$ with $\pi_W$ unbounded on $A$--orbits must satisfy $W\nest U_i$ for some $i$, using \cite[Thm. 5.1]{PetytUnbounded23}.   
    This allows us to use tools from \cite{BehrstockHierarchically19}, like the realisation theorem and distance formula, along with a modified version of the ``factored space'' construction from \cite{BehrstockAsymptotic17}, to reduce to the case where $X=H\times \newcone X$ (coarsely), where $H$ is an HHS whose index set contains the set $\mathfrak U$ of all $U$ nested in some $U_i$, and $\newcone X$ is an HHS where $A$ acts by HHS automorphisms with bounded orbits.  Moreover, letting $A'\leq A$ be a bounded-index subgroup fixing each $U_i$, we observe that $\mathcal CU_i$ contains an $A'$--invariant uniform quasiline $\gamma_{U_i}$.  In fact, the subgroup $A'$ stabilises two points $p^\pm$ in the HHS boundary from \cite{DurhamBoundaries17}, and $\gamma_{U_i}$ joins the two points $p^\pm_{U_i}\in\boundary\mathcal CU_i$ to which $p^\pm$ project.    
    
    For $U\in\mathfrak U$, either $\pi_U(X)$ is uniformly bounded, or $\gamma_{U_i}$ projects to a (bounded, but not uniformly) uniform quasigeodesic $\gamma_U$.  By passing to a suitable uniformly HQC subspace, we can assume that $\pi_U(H)$ is either uniformly bounded, or uniformly coarsely coincides with $\gamma_U$, for all $U$.  These arguments are collected in Proposition \ref{prop: almost invariant hull}.  

    In Proposition \ref{prop:bounded-orbits}, we use the fact that $\newcone X$ is coarsely dense in its injective hull, as shown in \cite{HaettelCoarse23}, together with Lang's fixed-point theorem for injective spaces \cite{LangInjective13} and the fact that $A$ has bounded orbits in $\newcone X$, to choose a point $\check x\in \newcone X$ whose $A$--orbit has uniformly bounded diameter.  This reduces the problem to finding the required $A$--invariant, uniformly EHQC, uniform quasiflat in $H$.

    This is done using Theorem \ref{thm: median hulls}, which is a variant of Durham's cubical model theorem, and in fact we borrow our strategy from \cite{DurhamCubulating23}.  The latter is itself a considerable generalisation and strengthening of the cubical approximation theorem from \cite{BehrstockQuasiflats21} (which was reproved by Bowditch by different means in \cite{Bowditchconvex18}), and strengthens Durham--Zalloum's cubical model theorem for HHS boundary points \cite{DurhamGeometry22}.  (See \cite{Zalloum:survey} for a survey of the various cubical modelling tools available in HHSs, and \cite{DurhamMinskySisto:stable} for another strengthened form of cubical approximation for colourable HHGs.)  

    Given the inputs to Theorem \ref{thm: median hulls}, Durham's theorem would provide (depending in principle on the choice $z_0\in H$ of basepoint) a uniform quasi-isometry from a CAT(0) cube complex $Q$ to the HHS $H$, preserving (coarse) medians up to uniformly bounded error.  Durham also obtains a form of equivariance, but it is not quite what we need.  We restrict ourselves to a more limited setting than in Durham's work, because we are concerned only with the case where $H$ is the hull of a pair of boundary points (Durham allows arbitrary finite sets), which lets us dispense with basepoints and obtain a complete, connected, finite-rank median space $Q$ and a uniformly quasi-median quasi-isometry $H\to Q$, as in Durham's theorem, but with the added feature that $A$ acts on $Q$ by isometries and the map to $Q$ is $A$--equivariant.  From there, a result of Bowditch \cite{BowditchProperties16} yields a compatible CAT(0) metric on $Q$ to which we can apply the classical flat torus theorem to obtain the desired EHQC quasiflat in $H$.

    To prove the more detailed version of our theorem (Theorem \ref{thm:HHS-semisimple-coarse-minset}), we consider the set of all $A$--invariant quasi-flats obtained in this way. The resulting quasi-product structure $Y \times C \times \mathbb E^n$ comes from the fact that the fixed point set of $A$ in the injective hull of $\newcone X$ is itself injective, and that the CAT(0) flat torus theorem provides an $A$--invariant subspace $C \times \mathbb E^n \subseteq Q$ such that $C$ is closed and convex (and thus CAT(0)).

    \subsection{Outline of the paper}\label{subsec:outline}
    Section \ref{sec:background} contains background on hierarchically hyperbolic spaces, as well as a few miscellaneous facts needed later in the paper.  Section \ref{sec:median-model} is devoted to the proof of Theorem \ref{thm: median hulls}, on median models of hulls.  Section \ref{sec:bounded-orbits} deals with bounded actions on HHSs, yielding Proposition \ref{prop:bounded-orbits}.  In Section \ref{subsec:initial-quasiflat}, we prove Proposition \ref{prop: almost invariant hull}, providing the necessary canonical $A$--invariant HQC and EHQC subspaces of $X$.  In Section \ref{sec:main-theorem}, we combine these ingredients to prove our main theorems, Theorem \ref{thm:HHS-semisimple-coarse-minset} and Theorem \ref{thm:B-central-extension}, and we apply these to obtain the various corollaries in Section \ref{sec:applications}.  
    
    There is also Appendix \ref{app:constants}, where we restate several existing results about hierarchical hyperbolicity, making explicit how constants produced by these statements depend on the inputs.  In no such case do we present a new proof; we just extract from the proofs in the literature the required information about constants.  Typically, we require some output constant to depend only on the parameters of the HHS, and obtaining this is a matter of tracing through an existing proof to get a quantitative version of the statement we need.

    \subsection{Acknowledgments}\label{subsec:ack}
    We thank Giorgio Mangioni and Ervin Hadziosmanovic for a helpful discussion. We also thank Jason Behrstock, Giorgio Mangioni, Harry Petyt, and Abdul Zalloum for several useful and interesting comments on a previous version.  MH is grateful to Carolyn Abbott, Harry Petyt, and Abdul Zalloum for discussions during the writing of \cite{abbottUniformUndistortionBarycentres2025} leading to some of the questions addressed in this paper.  Finally, we are particularly grateful to Jason Behrstock for some useful perspective on the mapping class group case, and for asking us a question that prompted us to work out the quasiflat closing theorem.
	
	\newsection{Background on hierarchically hyperbolic spaces}\label{sec:background}
	Fix a hierarchically hyperbolic space (HHS) $(X,\mathfrak S)$, as defined in \cite[Defn. 1.1]{BehrstockHierarchically19}.  We adopt the notation from there, some of which we recall here.
	
	For each $U\in\mathfrak S$, let $\pi_U:X\to\mathcal CU$ be the associated projection to the hyperbolic space $\mathcal CU$.  Whenever $U,V\in\mathcal CU$ satisfy $U\propnest V$ or $U\transverse V$, the uniformly bounded set in $\mathcal CV$ associated to $U$ is denoted $\rho^U_V$, and similarly, if $U\propnest V$, there is a map $\rho^V_U:\mathcal CV\to\mathcal CU$.  Let $\dist$ be the metric on $X$, which is assumed to be a quasigeodesic space.
	For each $V\in\mathfrak S$, let $\dist_V$ be the (hyperbolic geodesic) metric on $\mathcal CU$.  Following the usual convention, given $x,y\in X$, we write $\dist_V(x,y)$ to mean $\dist_V(\pi_V(x),\pi_V(y))$.
	
	\subsection{HHS boundary and boundary point projections}\label{subsec:HHS-boundary}
	We recall the definition of underlying set of the \emph{HHS boundary} of $(X,\mathfrak S)$ from \cite{DurhamBoundaries17}. We will not need the topology.
	
	\begin{defn} \label{defn: boundary points}
		The \emph{HHS boundary} of an HHS $(X, \mathfrak S)$ is a set $\partial X$ of points $\lambda \in \partial X$, each of which is determined by the following data:
		\begin{enumerate}
			\item a pairwise orthogonal set of domains $\support(\lambda) \subseteq \mathfrak S$, called the \textit{support} of $\lambda$;
			\item a point in the Gromov boundary $\lambda^U \in \partial \mathcal CU$ for each $U \in \support(\lambda)$;
			\item a collection of positive real numbers $\{a_U : U \in \support(\lambda)\}$ such that $\sum_{U \in \support(\lambda)} a_u = 1$.
		\end{enumerate}
		Set $\support^\perp(\lambda) \coloneqq \{V \in \mathfrak S : V \perp U$ for all $U \in \support(\lambda)\}$.
	\end{defn}
	
	\begin{notation}\label{notation:morse}
		Given $\delta>0$, denote by $\Morse_\delta$ the function provided by the Morse lemma for $\delta$--hyperbolic geodesic spaces, i.e. for any $\epsilon_1,\epsilon_2$, any $(\epsilon_1,\epsilon_2)$--quasigeodesic in such a space is at Hausdorff distance at most $\Morse_\delta(\epsilon_1,\epsilon_2)$ from the geodesic joining its endpoints.
	\end{notation}
	
	The projection maps $\pi_U$ can be extended to the boundary as follows. This process is canonical unless $U \in \support^\perp$, in which case it requires a choice of basepoint in $X$.
	
	Let $\delta \geq 0$ be the hyperbolicity constant (so $\mathcal CU$ is $\delta$-hyperbolic for all $U \in \mathfrak S$) and let $\mathcal E$ be the constant from the bounded geodesic image axiom (\cite[Defn. 1.1.(7)]{BehrstockHierarchically19}).  Let $\lambda \in \partial X$, let $U \in \support(\lambda)$, and let $V \in \mathfrak U$ be such that $V \propnest U$. Denote by $\lambda^V_U \subseteq \mathcal CU$ the union of all $(1,20\delta)$-quasi-geodesic rays $\gamma$ in $\mathcal CU$ in the equivalence class corresponding to the boundary point $\lambda^U$ such that $\gamma \cap \mathcal N^{\mathcal CU}_{\mathcal E + \Morse_\delta(1,20\delta)}(\rho^V_U) = \emptyset$.
	
	\begin{defn}[Projections of boundary points]\label{defn:boundary-projection}
		Let $\lambda \in \partial X$ and let $x_0 \in X$. For each $U \in \mathfrak S$,  define the \emph{projection based at $x_0$}, denoted $\pi_{U,x_0}(\lambda) \in 2^{\mathcal CU} \cup \partial \mathcal CU$, as follows:
		\begin{itemize}
			\item if $U \in \support^\perp(\lambda)$ then let $\pi_{U,x_0}(\lambda) \coloneqq \pi_U(x_0)$;
			\item if $U \in \support(\lambda)$ then let $\pi_{U,x_0}(\lambda) \coloneqq \lambda^U$;
			\item if $U \notin \support(\lambda)$ and $U \propnest V$ for some $V \in \support(\lambda)$ then let $\pi_{U,x_0}(\lambda) \coloneqq \rho^V_U(\lambda^U_V)$. Also set $\rho^V_U(\pi_V(\lambda, x_0)) \coloneqq \pi_U(\lambda, x_0)$.
			\item If none of the above hold, then the set $\mathcal V \coloneqq \{V \in \support(\lambda) : V \propnest U$ or $V \transverse U\}$ is non-empty and we define $\pi_{U,x_0}(\lambda) \coloneqq \bigcup_{V \in \mathcal V} \rho^V_U$.
		\end{itemize}
		If $V \in \mathfrak S - \support^\perp(\lambda)$, then $\pi_{V,x_0}(\lambda)$ is independent of $x_0$, so we can write $\pi_V(\lambda) \coloneqq \pi_{V,x_0}(\lambda)$.
	\end{defn}
	
	\subsection{HHS parameters}\label{subsec:parameters}
	Our main result asserts that certain ``output'' quantities depend only on the quantities in the definition of an HHS (\cite[Defn. 1.1]{BehrstockHierarchically19}), namely:
	\begin{enumerate}
		\item $q$ is the quasigeodesic constant for $X$, i.e. any two points are joined by a $(q,q)$--quasigeodesic.
		\item $\delta$ is the hyperbolicity constant, i.e. $\mathcal CU$ is $\delta$--hyperbolic for all $U\in\mathfrak S$.
		\item Constants $K$ and $\xi$: for all $U\in\mathfrak S$, the coarse map $\pi_U$ is $(K,K)$--coarsely lipschitz and has $K$--quasiconvex image, while $\pi_U(x)$ has diameter at most $\xi$ for all $x\in X$.  Also, $\diam(\rho^U_V)\leq \xi$ whenever $U\propnest V$ or $U\transverse V$.
		
		\item There is a constant $\kappa_0$ such that all of the inequalities in \cite[Defn. 1.1.(4)]{BehrstockHierarchically19} (the transversality/consistency axiom) use the upper bound $\kappa_0$.
		
		\item The \emph{complexity} $\hhscomp\in\naturals$ is the maximum length of a $\propnest$--chain in $\mathfrak S$.
		
		\item The large link axiom \cite[Defn. 1.1.(6)]{BehrstockHierarchically19}, bounded geodesic image axiom \cite[Defn. 1.1.(7)]{BehrstockHierarchically19} and partial realisation axiom (\cite[Defn. 1.1.(8)]{BehrstockHierarchically19}) each introduce one additional constant, denoted $\lambda, \mathcal E$ and $\alpha$ respectively.

		\item The uniqueness axiom (\cite[Defn. 1.1.(9)]{BehrstockHierarchically19}) provides a function $\theta_u:\reals_{\geq 0}\to\reals_{\geq 0}$ such that $d(x,y)\leq \theta_u(\sup_{V\in\mathfrak S} \dist_V(x,y))$ for all $x,y\in X$.
	\end{enumerate}
	
	Using the definition of boundary projections in Section \ref{subsec:HHS-boundary}, it is straightforward to prove:
	
	\begin{lem}\label{lem:boundary-projection-bound}
		If $\lambda \in \partial X$, $x_0 \in X$ and $U \in \mathfrak S - \support(\lambda)$, then $\diam(\pi_{U,x_0}(\lambda)) \leq \xi'$, where $\xi'$ is a function of $\xi$, $\alpha$ and $\mathcal E$. 
	\end{lem}
	
	\begin{proof}
		Let $U\in\mathfrak S-\support(\lambda)$.  There are three cases.
		
		First suppose that $U\orth V$ for all $V\in\support(\lambda)$.  Then $\diam(\pi_{U,x_0}(\lambda))=\diam(\pi_U(x_0))\leq\xi$.
		
		Second, suppose that $U\transverse V$ or $V\propnest U$ for some $V\in\support(\lambda)$, i.e. the set $\mathcal V$ from Definition \ref{defn:boundary-projection} is nonempty.  By definition, $\pi_{U,x_0}(\lambda)=\bigcup_{V\in\mathcal V}\rho^V_U$.  Since $V\orth V'$ for all distinct $V,V'\in \mathcal V$, we have $\dist_U(\rho^V_U,\rho^{V'}_U)\leq 2\alpha$, so $\diam(\pi_{U,x_0}(\lambda))\leq 2(\alpha+\xi)$.
		
		Third, suppose that $U\propnest V$ for some $V\in\support(\lambda)$.  Recall that $\lambda^U_V\subseteq \mathcal CV$ is the set of $(1,20\delta)$--quasigeodesic rays $\gamma$ representing the boundary point $\lambda^V$ and avoiding the neighbourhood of $\rho^U_V$ of radius $\mathcal E_0:=\mathcal E+\Morse_\delta(1,20\delta)$.  Recall that in this case, $\pi_{U,x_0}(\lambda)$ is defined to be $\rho^V_U(\lambda^U_V)$, whose diameter we need to bound.  Let $\gamma$ be a ray in $\mathcal CV$ as above.  Let $a,b\in\gamma$.  Let $\beta$ be a geodesic in $\mathcal CV$ from $a$ to $b$.  Since $\beta$ lies in the $\Morse_\delta(1,20\delta)$--neighbourhood of $\gamma$, it cannot come $\mathcal E$--close to $\rho^U_V$, so by the bounded geodesic image axiom (\cite[Defn. 1.1.(7)]{BehrstockHierarchically19}), $\dist_U(\rho^V_U(a),\rho^V_U(b))\leq \mathcal E$.  Since this works for arbitrary $a,b$, we have $\diam(\pi_U(\gamma))\leq \mathcal E$.  Now let $\gamma'$ be some other $(1,20E)$--quasigeodesic ray in $\mathcal CV$ avoiding $\neb_{\Morse_\delta(1,20\delta)+\mathcal E}(\rho^U_V)$ and representing $\lambda^V$.  Since $\rho^U_V$ is bounded and $\gamma,\gamma'$ are at finite Hausdorff distance, there exist $a\in\gamma,a'\in\gamma'$ such that every geodesic in $\mathcal CV$ from $a$ to $a'$ avoids $\neb_{\mathcal E}(\rho^U_V)$.  Hence $\dist(\rho^V_U(\gamma),\rho^V_U(\gamma'))\leq \mathcal E$, by another application of bounded geodesic image.  Thus $\diam(\rho^V_U(\lambda^U_V))\leq 3\mathcal E$.
		
		Combining the three cases gives $\diam(\pi_{U,x_0}(\lambda))\leq 2(\alpha+\xi)+3\mathcal E$.
	\end{proof}
	
	\begin{defn}[HHS constant]\label{defn:HHS-constant}
		Fix $E \geq \max\{q,K,\delta,\xi, \xi',\kappa_0,\alpha, \mathcal E\}$; $E$ is the \emph{HHS constant}.  
	\end{defn}
	
	\begin{remark}\label{rem:HHS-constant-accounts-for-boundary}
		Definition \ref{defn:HHS-constant} is almost identical to what is usually referred to as the HHS constant (see e.g \cite[Remark~1.6]{BehrstockHierarchically19}) except for the condition that $E \geq \xi'$ which accounts for the boundary projections.  As usual, since all of the statements in \cite[Defn. 1.1]{BehrstockHierarchically19} involving the quantities used to define $E$ remain true when the quantity in question is replaced by $E$, we henceforth work with $E$ instead of $\kappa_0,\alpha$, etc.
	\end{remark}

	\begin{defn}[Complexity]\label{defn:complexity}
		We denote the complexity of $(X,\mathfrak S)$ by $\hhscomp$. Note that it depends only on the poset $(\mathfrak S,\nest)$. 
	\end{defn}
	
	\begin{defn}[HHS parameters]\label{defn:HHS-parameters}
		The \emph{HHS parameters} of $(X,\mathfrak S)$ are: the HHS constant $E$, the complexity $\hhscomp$, and the uniqueness function $\theta_u$.
	\end{defn}
	
		When we say that some quantity $C$ depends only on the HHS parameters, we mean that $C$ is a function of $E,\hhscomp,$ and $\theta_u$.  For brevity, when working in a fixed HHS, we say that a quantity is \emph{uniform} to mean that it depends only on the HHS parameters.

	\begin{defn}[Adjusting the boundary projections]
		In the interest of only using the constant $E$ from now on, we make a small change to the definition of the boundary projections. If $\lambda \in \partial X$ and $V \propnest U$ for some $U \in \support(\lambda)$ then let ${\lambda'}^V_U \coloneqq \lambda^V_U - \mathcal N_{10E}(\rho^V_U)$. Define $\pi_U(\lambda)' \coloneqq \rho^U_V({\lambda'}^V_U)$. Then $\pi_U(\lambda)' \subseteq \pi_U(\lambda)$ so its diameter is still bounded by $E$. We replace $\lambda^V_U$ by ${\lambda'}^V_U$ and $\pi_U(\lambda)$ by $\pi_U(\lambda)'$.
	\end{defn}
	
	\subsection{HHS automorphisms and HHGs}\label{subsec:autos}
	We work with the following notion of \emph{HHS automorphisms} introduced in \cite[Sec. 2]{PetytUnbounded23}, which is simpler to work with than the original formulation (and no less general, by \cite[Sec. 2.1]{DurhamCorrection20}).
	
	\begin{defn}[HHS automorphism]\label{defn:HHS-automorphism}
		Let $(X,\mathfrak S)$ be an HHS.  A group $H$ acts on $(X,\mathfrak S)$ by \emph{HHS automorphisms} if there is an action $H\to Sym(X)$, an action $H\to Sym(\mathfrak S)$ preserving the relations $\nest,\orth,\transverse$, and, for each $U\in\mathfrak S$ and each $h\in H$, an isometry $h:\mathcal CU\to\mathcal C hU$ such that all of the following hold:
		\begin{enumerate}
			\item If $U\in\mathfrak S$ and $g,h\in H$, then the isometry $gh:\mathcal CU\to\mathcal C gh U$ is the composition of $g:\mathcal C hU\to\mathcal C ghU$ with $h:\mathcal CU\to\mathcal C hU$.
			\item $\pi_{gU}(gx)=g(\pi_U(x))$ for all $x\in X, U\in\mathfrak S, g\in H$.
			\item $\rho^{gU}_{gV}=g(\rho^U_V)$ whenever $U\propnest V$ or $U\transverse V$.
            \qedhere
		\end{enumerate}
	\end{defn}
	
	If $H$ acts by HHS automorphisms, the distance formula (Theorem \ref{thm:distance-formula}) implies that each $h\in H$ acts on $X$ as a quasi-isometry with constants depending on the HHS parameters only, but in this paper, it will always suffice to work with groups $H$ of HHS automorphisms with the additional property that the $H$--action on $(X,\dist)$ is isometric.  A special case is:
	
	\begin{defn}[HHG]\label{defn:HHG}
		Let $G$ be a group and let $\dist$ be a finitely generated word metric on $G$.  Suppose that $(G,\dist)$ admits an HHS structure $(G,\mathfrak S)$ on which $G$ acts by HHS automorphisms in such a way that the associated $G$--action on the space $(G,\dist)$ is the left-multiplication action, and the $G$--action on $\mathfrak S$ is cofinite.  Then $(G,\mathfrak S)$ is a \emph{hierarchically hyperbolic group (HHG) structure} for $G$, and $G$ is a \emph{hierarchically hyperbolic group (HHG)} if it has an HHG structure.  We say that the pair $(G,\mathfrak S)$ is \emph{an HHG} in this situation. 
	\end{defn}
	
	Any HHS automorphism $g$ extends to the HHS boundary: if $\lambda \in \partial X$ is as in Definition~\ref{defn: boundary points} then $g \lambda \in \partial X$ is determined by the support $g \support(\lambda)$, the boundary points $g \lambda^U \in \boundary \mathcal C gU$ for all $gU \in g\support(\lambda)$ and the constants $a_{gU} \coloneqq a_U$ for all $gU \in g\support(\lambda)$.

	\begin{lem}[Equivariant downward $\rho$ maps]\label{lem:invariant-down-rho}
		Up to uniformly modifying the HHS parameters, we can assume that
		\begin{itemize}
			\item $\pi_U(x) \subseteq \rho^V_U(\pi_V(x))$ whenever $x \in X$ and $U \propnest V$.
			\item if $g$ is an HHS automorphism of $X$ and $U,V \in \mathfrak S$ satisfy $U \propnest V$, then 
			\[
			\rho^{gV}_{gU}(gz) = g \rho^{V}_U(z)
			\]
			for all $z \in \mathcal CV - \mathcal N_{10E}(\rho^V_U)$.
		\end{itemize}
	\end{lem}
	\begin{proof}
		Let $x \in X$ and let $U, V \in \mathfrak S$ satisfy $U \propnest V$. Let $z\in \mathcal CV$, and suppose that $\dist_{V}(z,\rho^U_{V})\geq10E$. Let $x_z \coloneqq \{y\in X:\dist_{V}(y,z)\leq E\}$, which is non-empty by the coarse surjectivity of $\pi_{V}$. Let $\varrho^{V}_U(z)=\bigcup_{x\in x_z}\pi_U(x)$. The consistency and bounded geodesic image axioms (\cite[Defn. 1.1.(4),(6)]{BehrstockHierarchically19}) imply that $\varrho^{V}_U(z)$ has diameter bounded in terms of $E$, and lies uniformly close (in terms of $E$) to $\rho^{V}_U(z)$.  In particular, we can (and do) replace $\rho^{V}_U$ by $\varrho^{V}_U$ in the HHS structure $(X,\mathfrak S)$ and still have an HHS whose parameters are bounded in terms of the original ones.  (For $z\in \mathcal CV$ with $\dist_{V}(z,\rho^U_{V}) < 10E$, let $\varrho^{V}_U(z)=\rho^{V}_U(z)$.)
		
		Let $g$ be an HHS automorphism and let $z \in \mathcal CV$ satisfy $d_V(z, \rho^U_V) \geq 10E$.  Definition~\ref{defn:HHS-automorphism} gives $d_{gV}(gz, \rho^{gU}_{gV}) \geq 10E$ and $gx_z = x_{gz}$, so $\rho^{gV}_{gU}(gz) = \pi_{gU}(gx_z) = g \pi_U(x_z) = g \rho^V_U(z)$.
	\end{proof}
	
	We will henceforth assume that the downwards $\rho$ maps are defined as in the proof of Lemma~\ref{lem:invariant-down-rho} and thus satisfy its conclusions. As a result, the boundary projections are also equivariant with respect to HHS automorphisms:
	
	\begin{lem} \label{lem: equivariant boundary projections}
		Let $\lambda \in \partial X$ be a point in the HHS boundary of $X$, let $g$ be an HHS automorphism and let $U \in \mathfrak S$. Then $g \pi_{U,x_0}(\lambda) = \pi_{gU, gx_0}(g\lambda)$ for all $x_0 \in X$.
	\end{lem}
	\begin{proof}
		The only case which is not covered by Definition~\ref{defn:HHS-automorphism} is if $U \propnest V$ for some (necessarily unique) $V \in \support(U)$. In this case $\pi_U(\lambda) = \rho^V_U(\lambda^U_V)$, $\pi_{gU} = \rho^{gV}_{gU}(\lambda^{gU}_{gV})$ and $\lambda^U_V \subseteq \mathcal CV - \mathcal N_{10E}(\rho^U_V)$, $\lambda^{gU}_{gV} \subseteq \mathcal CgV - \mathcal N_{10E}(\rho^{gU}_{gV})$. By Lemma~\ref{lem:invariant-down-rho} it follows that $g\pi_U(\lambda) = \pi_{gU}(g\lambda)$.
	\end{proof}
	
	\subsection{Realisation}\label{subsec:realisation}
	We recall from \cite{BehrstockHierarchically19} the definition of a \emph{consistent tuple}, and the \emph{realisation theorem} that, roughly speaking, characterises the image of $\prod_U\pi_U:X\to \prod_U\mathcal CU$ as the ``solution set'' of the consistency inequalities.  The following is \cite[Defn. 1.17]{BehrstockHierarchically19}:
	
	\begin{defn}[Consistent tuple]\label{defn:consistent-tuple}
		Let $\kappa\geq 0$.  Let $\vec b=(b_U)_{U\in\mathfrak S}\in\prod_{U\in\mathfrak S}2^{\mathcal CU}$.  Then $\vec b$ is \emph{$\kappa$--admissible} if $\diam(b_U)\leq \kappa$ and $\dist_U(b_U,\pi_U(X))\leq\kappa$ for all $U\in\mathfrak S$.  If $\vec b$ is $\kappa$--admissible, then it is \emph{$\kappa$--consistent} if 
		\begin{itemize}
			\item $\min\{\dist_U(\rho^V_U,b_U),\dist_V(\rho^U_V,b_V)\}\leq \kappa$ whenever $U\transverse V$, and
			\item $\min\{\dist_V(\rho^U_V,b_V),\diam(\rho^V_U(b_V)\cup b_U)\}\leq \kappa$ whenever $U\propnest V$.\qedhere
		\end{itemize}
	\end{defn}
	
	Next is the realisation theorem, which is Theorem 3.1 in \cite{BehrstockHierarchically19}.  Following \cite[Thm. 12.5]{Casals-RuizReal24}, we have restated it to make explicit what the quantities involved depend on; as noted in \cite{Casals-RuizReal24}, the constants in question are easy to extract from the proof in \cite{BehrstockHierarchically19}.
	
	\begin{thm}[Realisation]\label{thm:hhs_realisation}
		Let $(\cuco X,\mathfrak S)$ be a hierarchically hyperbolic space.  Then there exists $r_0$, depending only on the HHS parameters, such that the following holds.  Let $\kappa\ge1$ and let $(b_V)_{V\in\mathfrak S}\in\prod_{V\in\mathfrak S}\fontact V$ be a  $\kappa$--consistent, $\kappa$--admissible tuple.  Then there exists $x\in\cuco X$ such that $\dist_V(x,b_V)\leq r_0\kappa$ for all $V\in\mathfrak S$.
	\end{thm}
	
	\subsection{Coarse median}\label{sec:coarse-median}
	We use the notion of a \emph{coarse median space} from \cite[Sec. 8]{BowditchCoarse13}, although we almost always work with special cases --- median spaces or HHSes --- rather than the general definition.  See \cite{NibloWrightZhang:intrinsic} for more background on \emph{coarse median algebras/spaces}.
	
	\begin{defn}[Coarse median space]\label{defn:coarse-median}
		A metric space $(\Omega,d)$ is \emph{coarse median} with \emph{parameter} $\kappa:\integers_{\ge0}\to [0,\infty)$ if there is a map $\mu:\Omega^3\to\Omega$ (the \emph{coarse median}) such that:
		\begin{enumerate}
			\item [(C1)] We have $d(\mu(x,y,z),\mu(x',y',z'))\leq \kappa(0)\left(d(x,x')+d(y,y')+d(z,z')\right)+\kappa(0)$ for all $x,y,z,x',y',z'\in \Omega$.
			
			\item [(C2)] For all $n\in\naturals$ and all $A\subseteq \Omega$ with $|A|\leq n$, there is a finite median algebra $(M,\eta)$ and maps $f:M\to \Omega$ and $\bar f:A\to M$ such that $d(\mu(f(x),f(y),f(z)),f(\eta(x,y,z)))\leq \kappa(n)$ for all $x,y,z\in M$, and $d(a,f\bar f(a))\leq \kappa(n)$ for all $a\in A$.
            \qedhere
		\end{enumerate}
	\end{defn}
	
	Next is the coarse median on an HHS $X$, from \cite[Sec. 7]{BehrstockHierarchically19} (or \cite[Sec. 14]{Casals-RuizReal24}).
	
	\begin{cons}[HHS coarse median]\label{cons:coarse-median}
		Let $x,y,z\in X$.  For each $U\in\mathfrak S$, let $m_U$ be the set of $p\in\mathcal CU$ such that $\dist_U(p,\gamma)\leq 10E$ for any geodesic $\gamma$ of $\mathcal CU$ with endpoints in $\pi_U(x)\cup\pi_U(y)\cup\pi_U(z)$.  Let $(b_U)_{U\in\mathfrak S}$ be a tuple with $b_U\in m_U$ for all $U$.  By \cite[Lem. 2.6]{BehrstockHierarchically19}, there exists $\kappa$ such that any such $(b_U)_{U\in\mathfrak S}$ is $\kappa$--consistent, and the proof of the same lemma shows we can take $\kappa=100E$.  Hence, the realisation theorem (Theorem \ref{thm:hhs_realisation}) provides a point $m\in X$ such that $\dist_U(m,b_U)\leq 100Er_0$, where $r_0$ is the constant, depending only on the HHS parameters, from Theorem \ref{thm:hhs_realisation}.
		
		Let $\mu(x,y,z)$ be the set of all $m\in X$ such that $\dist(m,m_U)\leq 100Er_0$ for all $U\in\mathfrak S$.  The above discussion shows that $\mu(x,y,z)\neq\emptyset$.  Now, if $m,m'\in\mu(x,y,z)$, then for all $U$, we have $\dist_U(m,m')\leq 200E(r_0+1)$, so $d(m,m')\leq \theta_u(200E(r_0+1))$, where $\theta_u$ is the uniqueness function (an HHS parameter).  Thus $\diam(\mu(x,y,z))$ is bounded, independently of $x,y,z$, in terms of the HHS parameters only.  Hence $\mu$ is a well-defined map sending points in $X^3$ to uniformly bounded sets in $X$. 
	\end{cons}
	
	In summary, Construction \ref{cons:coarse-median} provides constants $C_0,C_1$, depending only on the HHS parameters, and a map $\mu:X^3\to 2^X$ satisfying both of the following conditions:
	\begin{itemize}
		\item $\diam(\mu(x,y,z))\leq C_0$ for all $x,y,z\in X$.
		\item For all $U\in\mathfrak S$ and $x,y,z\in X$, the following holds.  Let $a,b\in\{x,y,z\}$ be distinct.  Then $\dist_U(\mu(x,y,z),\gamma)\leq C_1$, where $C_1$ is any $\mathcal CU$--geodesic joining a point in $\pi_U(a)$ to a point in $\pi_U(b)$.
	\end{itemize}
	If a group $G$ acts on $(X,\mathfrak S)$ by HHS automorphisms, then the coarse map $\mu$ is $G$--equivariant (for the diagonal action on $X^3$ and the action on $2^X$ induced by the $G$--action on $X$).  
	
	In \cite[Sec. 14]{Casals-RuizReal24}, it is observed that if $(G,\mathfrak S)$ is an HHG then we can, moreover, assume that $\mu$ is a map from $X^3\to X$ (rather than a coarse map) by equivariantly choosing points in the bounded sets $\mu(x,y,z)$ defined above.  This can be done, more generally, as long as $G$ acts on $(X,\mathfrak S)$ by HHS automorphisms in such a way that the $G$--action on $X$ is free.  Later, we will typically reduce to the case of free actions, and in such situations we will assume that $\mu$ is an equivariant \emph{map}, not just an equivariant coarse map.  
	
	\begin{remark}\label{rem:coarse-median-is-coarse-median}
		That $\mu$ satisfies Definition \ref{defn:coarse-median} follows from the cubical approximation theorem explained in Section \ref{appsubsec:cubical-approx}; see Corollary \ref{cor:coarse-median}.
	\end{remark}

	\begin{defn}[Quasimedian map]\label{defn:quasi-median}
		Let $\Lambda,\Omega$ be coarse median spaces with coarse medians $\lambda,\omega$ respectively; let $d$ be the metric on $\Omega$. Let $\kappa\geq 0$.  A map $f:\Lambda\to \Omega$ is \emph{$\kappa$--quasimedian} if $d(\omega(f(a),f(b),f(c)),f(\lambda(a,b,c)))\leq \kappa$ for all $a,b,c\in\Lambda$.
	\end{defn}
	
	The next two definitions help to interpret HHS notions in terms of the coarse median.
	
	\begin{defn}[Median quasiconvexity]\label{defn:coarse-median-quasiconvex}
		Let $\Omega$ be a coarse median space with coarse median $\omega$ and metric $\dist$.  Given $\kappa\geq 0$, a subset $Y\subseteq \Omega$ is \emph{$\kappa$--median-quasiconvex} if $\dist(\omega(y,y',x),Y)\leq \kappa$ for all $y,y'\in Y$ and $x\in \Omega$.
	\end{defn}
	
	\begin{defn}[Coarse subalgebra]\label{defn:coarse-median-subalgebra}
		Let $\Omega$ be a coarse median space with coarse median $\omega$ and metric $\dist$. Given $\kappa\geq 0$, a \emph{$\kappa$--coarse median subalgebra} of $\Omega$ is a subset $Y\subseteq \Omega$ such that $\dist(\omega(y,y',y''),Y)\leq \kappa$ for all $y,y',y''\in Y$.
	\end{defn}

	\subsection{Hierarchical quasiconvexity and related notions}\label{subsec: (E)HQC}
	We recall the notion of hierarchical quasiconvexity from \cite[Sec. 5]{BehrstockHierarchically19}:
	
	\begin{defn}\label{defn:HQC}
		Given a function $\kappa:[0,\infty)\to[0,\infty)$, a subset $Y\subseteq X$ is \emph{$\kappa$--hierarchically quasiconvex (HQC)} if both of the following hold:
		\begin{itemize}
			\item $\pi_U(Y)$ is $\kappa(0)$--quasiconvex in the $E$--hyperbolic space $\mathcal CU$ for all $U\in\mathfrak S$;
			\item for all $r\geq 0$, if $x\in X$ satisfies $\dist_U(x,Y)\leq r$ for all $U\in\mathfrak S$, then $\dist(x,Y)\leq \kappa(r)$.\qedhere
		\end{itemize}
	\end{defn}
	
	We next have \cite[Prop. 5.11]{RussellConvexity23}, relating hierarchical and coarse median quasiconvexity:
	
	\begin{prop}\label{prop:RST}
		Let $(X,\mathfrak S)$ be an HHS and let $Y\subseteq X$.  Then both of the following hold:
		\begin{enumerate}
			\item For all $\kappa:[0,\infty)\to[0,\infty)$, there exists $C\geq 0$, depending only on $\kappa$ and the HHS parameters, such that, if $Y$ is $\kappa$--hierarchically quasiconvex, then $Y$ is $C$--median quasiconvex.
			\item For all $C\geq 0$, there exists $\kappa:[0,\infty)\to[0,\infty)$, depending only on $C$ and the HHS parameters, such that, if $Y$ is $C$--median quasiconvex, then it is $\kappa$--hierarchically quasiconvex.
		\end{enumerate}
	\end{prop}
	
	Proposition \ref{prop:RST} is proved using \emph{hierarchy paths}, as defined in \cite{BehrstockHierarchically19}, which generalise the mapping class group notion from \cite{MasurMinsky:II}.  We will also use hierarchy paths, so we recall the definition (\cite[Defn. 4.2]{BehrstockHierarchically19}):
	
	\begin{defn}[Hierarchy path]\label{defn:hier-path}
		Given $D\geq 0$, a $D$--\emph{hierarchy path} in $X$ is a $(D,D)$--quasi-isometric embedding $\gamma:[0,L]\to X$, for some $L\geq 0$, with the property that $\pi_W\circ\gamma$ is an unparametrised $(D,D)$--quasigeodesic in $\mathcal CW$ for all $W\in\mathfrak S$.
	\end{defn}
	
	Combining \cite[Lem. 1.37]{BehrstockQuasiflats21} with \cite[Thm. 4.4]{BehrstockHierarchically19} yields:
	
	\begin{prop}\label{prop:hierarchy-path-char}
		There exists $D$, depending only on the HHS parameters, such that any two points in $X$ are joined by a $D$--hierarchy path.  
		
		Moreover, for any $D'\geq 0$, there exists $C\geq 0$, depending only on $D'$ and the HHS parameters, such that a $(D',D')$--quasigeodesic $\gamma:[0,L]\to X$ is a $D'$-hierarchy path provided $\gamma$ is $C$--quasimedian.  Conversely, for all $C$, there exists $D''$ depending only on $C$ and the HHS parameters such that any $C$--quasimedian $(C,C)$--quasiisometric embedding $\gamma:[0,L]\to X$ is a $D''$--hierarchy path.
	\end{prop}
	
	Hierarchical quasiconvexity (equivalently, coarse median quasiconvexity) generalises quasiconvexity in hyperbolic spaces and ``coarsifies'' the notion of median convexity in median metric spaces.  Accordingly, the coarse closest point projection to a quasiconvex subset of a hyperbolic space generalises to HQC subsets of HHSes in a way that mimics gate maps in median spaces, using the following construction from \cite[Sec. 5]{BehrstockHierarchically19}.
	
	\begin{cons}[HQC gate maps]\label{cons:gate}
		Let $(X,\mathfrak S)$ be an HHS and let $Y\subseteq X$ be a $\kappa$--HQC subset.  Define  $\gate_Y:X\to 2^Y$ as follows.  For each $U\in\mathfrak S$, let $p^Y_U:\mathcal CU\to 2^{\pi_U(Y)}$ be the coarse closest-point projection, which takes points to sets of diameter bounded in terms of $E$ and $\kappa(0)$.  Fix $x\in X$.  Lemma 5.3 of \cite{BehrstockHierarchically19} provides a constant $\kappa'$, depending on the HHS parameters and the function $\kappa$ but not on $x$, such that any tuple $(q_U)_{U\in\mathfrak S}$ with $q_U\in p^Y_U(\pi_U(x))$ for all $U\in\mathfrak S$ is $\kappa'$--consistent.  Hence there exists $\kappa''$, depending on the HHS parameters and $\kappa$ only, such that there exists at least one $\bar x\in Y$ such that $\dist_U(\bar x,p^Y_U(\pi_U(x)))\leq \kappa''$ for all $U$.  Let $\gate_Y(x)$ be the set of all such $\bar x$.  As noted in \cite[Defn. 5.4, Lem. 5.5]{BehrstockHierarchically19}, up to uniformly enlarging $\kappa''$, we have $\diam(\gate_Y(x))\leq \kappa''$ for all $x\in X$, and the coarse map $\gate_Y$ is $(\kappa'',\kappa'')$--coarsely lipschitz, $\kappa''$-quasimedian, and $\dist_Y(\gate_Y(y),y)\leq \kappa''$ for all $y\in Y$.
	\end{cons}

    One can uniformly perturb a gate map $\mathfrak g_Y$ so that it is a coarsely lipschitz retraction onto $Y$.  Thus HQC subsets are \emph{coarse lipschitz retracts}.
	
	\begin{lem}[Equivariant gate maps]\label{lem:equivariant-gate}
		Let $(X,\mathfrak S)$ be an HHS.  Let $G$ be a group equipped with an action on $(X,\mathfrak S)$ by HHS automorphisms.  Suppose that $Y\subseteq X$ is $\kappa$--HQC and $G$--invariant.  Then the gate map $\gate_Y:X\to 2^Y$ from Construction \ref{cons:gate} is $G$--equivariant.  Moreover, if the $G$--action on $X$ is free, then there is a $G$--equivariant map $\gate'_Y:X\to Y$ such that $\sup_{x\in X}\diam(\gate_Y(x)\cup \gate'_Y(x))$ is bounded above in terms of $\kappa$ and the HHS parameters.
	\end{lem}
	
	By Lemma \ref{lem:equivariant-gate}, in any situation where the $G$--action on $X$ is free and preserves $Y$, we will assume that $\gate_Y$ is a $G$--equivariant map $X\to Y$ (not just a coarse map) with the properties from Construction \ref{cons:gate}, with $\kappa''$ depending on the HHS parameters and $\kappa$ only.
	
	\begin{proof}[Proof of Lemma \ref{lem:equivariant-gate}]
		From Definition \ref{defn:HHS-automorphism} and the assumption that $GY=Y$, for all $U\in\mathfrak S$ and $g\in G$, we have $\pi_{gU}(Y)=g(\pi_U(Y))$. From the definition of coarse closest-point projection in a hyperbolic space, it follows that $p^Y_{gU}(\pi_{gU}(gx))=p^Y_{gU}(g(\pi_U(x))=g(p^Y_U(x))$ for all $x\in X$, so Construction \ref{cons:gate} implies that $\gate_Y(gx)=g\gate_Y(x)$ for all $g\in G,x\in X$. Now suppose that the $G$--action on $X$ is free.  Let $\{x_i\}_{i\in I}\subseteq X$ contain exactly one point of $X$ in each $G$--orbit, and for each $i\in I$, choose $\gate'_Y(x_i)\in \gate_Y(x)$ arbitrarily.  Given $x\in X$, there is a unique $g\in G$ and $i\in I$ with $x=gx_i$, so set $\gate'_Y(x)=g\gate_Y'(x_i)$.  This gives an equivariant map $\gate'_Y:X\to Y$, and $\diam(\gate_Y(x)\cup \gate'_Y(x))=\diam(\gate_Y(x))\leq \kappa''$ for all $x\in X$, as claimed.
	\end{proof}
	
	An important feature of hierarchically quasiconvex subsets is that they can be equipped with an induced HHS structure. The following holds by \cite[Lem.~5.6, Rem.~5.7]{BehrstockHierarchically19}.  The fact that the action of $G$ on $Y$ is by HHS automorphisms follows from the same considerations as in the equivariance part of the proof of Lemma \ref{lem:equivariant-gate}.
	
	\begin{lem}[HHS structure inherited by HQC set]\label{lem: HHS structure on HQC subspaces}
		Let $(X, \mathfrak S)$ be an HHS, let $\kappa: [0,\infty) \rightarrow [0,\infty)$ be a function, and let $Y \subseteq X$ be a $\kappa$-hierarchically quasiconvex subspace. Let $G$ be a group acting on $X$ by HHS automorphisms such that $GY = Y$. Then there is an HHS structure $\mathfrak S_Y$ on $Y$ such that the action of $G$ on $Y$ is by HHS automorphisms and the $\pi_U$--projection maps are coarsely surjective. The HHS parameters of $(Y, \mathfrak S_Y)$ depend only on the HHS parameters of $(X, \mathfrak S)$ and $\kappa$.  
	\end{lem}
	
	\begin{remark}\label{rem:HQC-HHS}
		The structure $(Y,\mathfrak S_Y)$ from Lemma \ref{lem: HHS structure on HQC subspaces} can be constructed in various ways.  If we do not require coarse surjectivity of the projections, then it has the following particularly simple description.  Let $\mathfrak S_Y=\mathfrak S$ (as a set with relations $\nest,\orth,\transverse$), for each $U\in\mathfrak S_Y$, let $\mathcal C_YU=\mathcal CU$, and let $\pi_U^Y:Y\to\mathcal CU$ be $\pi^Y_U=\pi_U|_{Y}$.  Recalling that the definition of an HHS does not require coarse surjectivity of projections, but only quasiconvexity of their images, $(Y,\mathfrak S_Y)$ is an HHS and in particular the distance formula (Theorem \ref{thm:distance-formula}) applies.  This description is convenient for some distance estimates in the proof of Proposition \ref{prop: almost invariant hull}, and we do not need it elsewhere.
		
		However, if one wishes to stay within the subclass of HHSes with uniformly coarsely surjective projections, then one can simply normalise, as in \cite[Rem. 1.4]{BehrstockHierarchically19}.  In this procedure, if $V\propnest U$ or $V\transverse U$, then replace $\rho^V_U$ with its image under the coarse closest-point projection $\mathcal CU\to 2^{\pi_U(Y)}$ (using that $\pi_U(Y)$ is uniformly quasiconvex in $\mathcal CU$).
	\end{remark}
	
	\subsection{Existential hierarchical quasiconvexity}\label{subsec:EHQC}
	Here is the weaker version of hierarchical quasiconvexity used in our main result:
	
	\begin{defn}[EHQC]\label{defn:EHQC}
		Let $k \geq 1, r \geq 0$. A subspace $Y \subseteq X$ is $(k,r)$\textit{--existentially hierarchically quasiconvex (EHQC)} if, for all $y_1, y_2 \in Y$, there exists a $k$-hierarchy path $\gamma$ from $y_1$ to $y_2$ with $\gamma \subseteq \mathcal{N}_r(Y)$.
	\end{defn}

    \begin{remark}\label{rem:0-EHQC}
    Observe that for any $k\geq 1,r\geq 0$, there exists $k'$ such that $Y$ is $(k,r)$--EHQC if and only if it is $(k',0)$--EHQC, since one can perturb a hierarchy path to lie in $Y$ at the expense of a controlled change in the hierarchy path constant.
    \end{remark}
	
	The following is straightforward to check using Proposition \ref{prop:hierarchy-path-char} and Proposition \ref{prop:RST}; it motivates Definition \ref{defn:EHQC}, but we will not use it.  Compare \cite[Lem. 2.11, Prop. 2.8]{HagenPetyt:BBF}.
	
	\begin{cor}\label{cor:subalgebra-EQHC}
		For all $C$, there exist $k,r$, depending on $C$ and the HHS parameters, such that any $C$--coarsely connected $C$--coarse median subalgebra of $X$ is $(k,r)$--EHQC.
	\end{cor}
	
	The converse to the preceding corollary is not true (consider codimension--$1$ affine subspaces of $\reals^3$, and give $\reals^3$ the $\ell_1$ metric), so we have to work with the more general notion of EHQC, and the language of paths, instead of coarse medians and coarse subalgebras.
	
	\subsection{Hulls in $X$ and $\boundary X$} \label{subsec: hulls}
	We recall \emph{hierarchically quasiconvex hulls} from \cite[Sec. 6]{BehrstockHierarchically19}:
	
	\begin{defn}[Hulls of interior points]
		Let $A \subseteq X$. For all $U\in\mathfrak S$, let $\hull_U(A)$ be the union of all $(1,20E)$-quasi-geodesics in $\mathcal CU$ with endpoints in $\pi_U(A)$.  Observe that $\hull_U(A)$ is a $10(E+\Morse_E(1,20E))$--quasiconvex subset of $\mathcal CU$.  For any $M \geq 0$, let $H_{M}(A)$ be the set of all $x\in X$ such that $\dist_U(\pi_U(x),\hull_U(A))\leq M$ for all $U\in\mathfrak S$.
	\end{defn}
	
	Hulls can be defined for sets that include points in the HHS boundary in a number of ways (see e.g. \cite{DurhamCubulating23,DurhamGeometry22}). We only need hulls of pairs of points in  $X\cup \boundary X$:
	
	\begin{defn}[Hulls of pairs]\label{defn:pair-hull}
		Fix $z_0\in X$ and let $x,y \in X \cup \partial X$. Define the \emph{hull} $\hull_{U,z_0}(x,y)$ of $\{x,y\}$ as follows, for each $U\in\mathfrak S$:
		\begin{itemize}
			\item If $x,y \in X$ then $\hull_{U,z_0}(x,y) \coloneqq \hull_U(\{x,y\})$.
			\item If $x \in X$ and $y \in \partial X$ then $\hull_{U,z_0}(\{x,y\}) = \hull_U(x,y)$ is the union of all $(1,20E)$-quasi-geodesics from $\pi_U(x)$ to $\pi_{U,x}(y)$.
			\item If $x,y \in \partial X$ then $\hull_{U,z_0}(\{x,y\})$ is the union of all $(1,20E)$-quasi-geodesics from $\pi_{U,z_0}(x)$ to $\pi_{U,z_0}(y)$.
		\end{itemize}
		Given $M \geq 0$, let $H_{M,z_0}(x,y)$ be the set of points $p \in X$ such that $\pi_U(p)$ is in the $M$-neighbourhood of $\hull_{U,z_0}(x,y)$ for all $U \in \mathfrak S$.
	\end{defn}
	
	\begin{remark}\label{rem:hull-is-hull}
		If $x,y \in X$ then $H_{M,z_0}(x,y) = H_M(\{x,y\})$.
	\end{remark}
	
	Lemma 6.2 in \cite{BehrstockHierarchically19} provides a constant $\theta'$ such that, for all $M \geq \theta'$ there is a function $h_M$ such that $H_M(A)$ is $h_M$--hierarchically quasiconvex for any $A\subseteq X$, and, as noted in \cite[Sec. 15]{Casals-RuizReal24}, examining the proof of \cite[Lem. 6.2]{BehrstockHierarchically19} shows that $\theta'$ and $h_M$ depend only on the HHS parameters (using the fact that the number $r_0$ from the realisation theorem depends only on the HHS parameters). The proof of the same lemma also reveals:
	
	\begin{lem} \label{lem: hulls are quasiconvex}
		There exists a constant $\theta \geq 0$, depending only on the HHS parameters, such that the following holds. For all $M \geq \theta$ and $z_0 \in X$, there exists a function $h_M: [0,\infty) \rightarrow [0,\infty)$, depending only on $M$ and the HHS parameters, such that, if $H_{M,z_0}(x,y) \neq \emptyset$ then it is $h_M$-hierarchically quasiconvex. 
	\end{lem}
	
	\subsection{Standard product regions}\label{subsec:product-regions}
	Recall from \cite[Sec. 17]{Casals-RuizReal24} and \cite[Sec. 5]{BehrstockHierarchically19} the \emph{standard product regions} in $X$: for each $U\in\mathfrak S$, the product region $P_U$ is the set of $x\in X$ such that $\dist_V(x,\rho^U_V)\leq E$ for all $V\in\mathfrak S$ for which $U\propnest V$ or $U\transverse V$.  By the definition of an HHS automorphism, we have $P_{gU}=gP_U$ for all $g\in G$, whenever $G$ is a group acting on $(X,\mathfrak S)$ by HHS automorphisms.
	
	Next, recall the coarse product structure of $P_U$.  Fix $q\in P_U$, and for each $V\in\mathfrak S$ let $q_V\in\pi_V(p)$.  Then $F_U\times \{q\}$ is the following subset of $P_U$.  Consider all $(p_V)_{V\nest U}\in\prod_{V\nest U}\mathcal CU$ that satisfy all of the $E$--consistency inequalities for pairs of nested or transverse pairs $V,V'\nest U$.  For each such $(p_V)_{V\nest U}$, extend to a tuple in $\prod_{V\in\mathfrak S}\mathcal CV$ by setting $p_V=q_V$ for $V\orth U$, and $p_V\in\rho^U_V$ for $U\propnest V$ or $U\transverse V$.  Then $(p_V)_{V\in\mathfrak S}$ is $E$--consistent, and the realisation theorem provides a point $p\in P_U$ such that $\dist_V(p,p_V)\leq r_0E$ for all $V\in\mathfrak S$.  By the distance formula (see Theorem \ref{thm:distance-formula} below), $p$ is coarsely unique (constants depending only on the HHS parameters) but not necessarily unique.  Let $F_U\times\{q\}$ be the set of all such $p$, as the tuple $(p_V)_{V\nest U}$ varies.  When $q$ is not important, we will denote $F_U\times\{q\}$ by $F_U$ and refer to it as a \emph{coarse parallel copy of $F_U$ in $P_U$}.
	
	It is shown in, for instance, \cite[Sec. 17]{Casals-RuizReal24}, that the subspaces $P_U$ and $F_U\times\{q\}$ and $\{p\}\times E_U$ are hierarchically quasiconvex, with hierarchical quasiconvexity functions depending only on the HHS parameters.  In particular, there is no loss  of generality in assuming that each of these subspaces is $E$--quasimedian quasiconvex, and we will do this to save notation.
	
	\subsection{Canonical HHS cone-off}\label{subsec:cone-off}
	The procedure for coning off certain hierarchically quasiconvex subsets of $(X,\mathfrak S)$ to obtain new HHS structures was introduced in \cite[Section~2]{BehrstockHierarchically19} and reproved in a slightly stronger form in \cite[Section~19]{Casals-RuizReal24}.  Here we adapt the cone-off construction to suit our purposes and use the statement from \cite{Casals-RuizReal24} to show that the adapted construction has the properties we need.
	
	\begin{defn}[Cone-off data]\label{defn:cone-off-data}
		Let $(X,\mathfrak S)$ be an HHS, let $\dist$ be the metric on $X$, let $A\to\Isom(X)$ be a group action by HHS automorphisms on $(X,\mathfrak S)$, and let $\mathfrak U\subseteq \mathfrak S$ be an $A$--invariant subset that is downward-closed in the partial order $\nest$.  The tuple $((X,\mathfrak S),A,\mathfrak U)$ is called \emph{cone-off data}.
	\end{defn}
	
	In both cone-off constructions, one creates a new HHS, with an $A$--action by HHS automorphisms, where the hierarchically quasiconvex subsets $F_U\subset X$ for $U\in\mathfrak U$ have become bounded, and the new HHS index set is $\mathfrak S-\mathfrak U$.  The new HHS, as a set, coincides with $X$, and the cone-off construction amounts to replacing $\dist$ with a new metric.  The difference between the construction from \cite{BehrstockAsymptotic17,Casals-RuizReal24} and our adapted version is a technical difference in how the new metric is defined.  With the original definition, as noted in \cite{Casals-RuizReal24}, the group $A$ does act on $X$, but the action is by uniform quasi-isometries, not isometries.  The purpose of the adapted construction is to keep the $A$--action isometric.
	
	We first recall the cone-off metric from \cite{BehrstockAsymptotic17,Casals-RuizReal24}:
	
	\begin{defn}[Non-canonical $\mathfrak U$--cone-off]\label{defn:old-cone-off}
		Given $x,y\in X$, let $D(x,y) \coloneqq \dist(x,y)$ unless there exists $U\in\mathfrak U$ and $e\in E_U$ such that $\{x,y\}\subset F_U\times \{e\}$, in which case $D(x,y)\coloneqq\min\{1,\dist(x,y)\}$.  The \emph{non-canonical $\mathfrak U$--cone-off} is the metric space $\oldcone X=(X,\widehat{\dist})$, where $\widehat{\dist}$ is the length metric on $X$ induced by $D$.
	\end{defn}
	
	\begin{remark}[Non-canonical]\label{rem:non-canonical}
		The phrase \emph{non-canonical} refers to the fact that the subspaces $F_U$ of $X$ are generally only coarsely well-defined.  Because of the arbitrary choices involved in constructing the subspaces $F_U\times\{q\}$, the condition on pairs $x,y$ ensuring that $D(x,y)\leq 1$ is not always $A$--invariant, so the $A$--action on $\oldcone X$ may not be isometric.    
	\end{remark}
	
	\begin{defn}[Canonical $\mathfrak U$--cone-off]\label{defn:modern-cone-off}
		Given $x,y\in X$, let $\widecheck{D}(x,y) \coloneqq \dist(x,y)$ unless both of the following hold:
		\begin{itemize}
			\item There exists $U\in\mathfrak U$ such that $x,y\in P_U$, and 
			\item $\dist_V(x,y)\leq 10E$ for all $V\in\mathfrak S$ such that $V\orth U$.
		\end{itemize}
		If $x,y$ satisfy both conditions, let $\widecheck{D}(x,y)\coloneqq\min\{1,\dist(x,y)\}$.  Now let $\widecheck{\dist}(x,y)$ be the length metric on $X$ induced by $\widecheck{D}$, and let $\newcone X=(X,\widecheck{\dist})$, which is called \emph{canonical $\mathfrak U$--cone-off}.
	\end{defn}
	
	\begin{lem}\label{lem:cone-off-action}
		The $A$--action on $X$ induces an action of $A$ on $\newcone X$ by isometries.
	\end{lem}
	
	\begin{proof}
		This follows from the fact that the $A$--action on $(X,\mathfrak S)$ is by HHS automorphisms and by isometries on $(X,\dist)$.  In particular, $gP_U=P_{gU}$, and $\dist_{gV}(gx,gy)=\dist_V(x,y)$, so $\widecheck{D}$ is $A$--invariant, and hence $\widecheck{\dist}$ is also. 
	\end{proof}
	
	\begin{lem}\label{lem:cone-offs-are-QI}
		There exists $C$, depending only on the HHS parameters of $(X,\mathfrak S)$, such that the set-theoretic identity map $\newcone X\to\oldcone X$ is a $(C,C)$--quasi-isometry.
	\end{lem}
	
	\begin{proof}
		Fix $x,y\in X$ and suppose that $D(x,y)\leq 1$.  If $\dist(x,y)\leq 1$, then $\widecheck{D}(x,y)\leq 1$.  The other possibility is that there exists $U\in\mathfrak U$ and $q\in P_U$ such that $x,y\in F_U\times \{q\}$.  By definition, this means that $x,y\in P_U$, and for all $V$ such that $V\orth U$, we have $\dist_V(x,q_V)\leq E$ and $\dist_V(y,q_V)\leq E$, so $\dist_V(x,y)\leq 2E$, and hence $\widecheck{D}(x,y)\leq 1$.  
		
		Conversely, suppose that $\widecheck{D}(x,y)\leq 1$.  If $\dist(x,y)\leq 1$, then $D(x,y)\leq 1$.  Otherwise, there exists $U\in\mathfrak U$ such that $x,y\in P_U$ and $\dist_V(x,y)\leq 10E$ whenever $V\orth U$.  Consider the subspace $F_U\times \{x\}$.  As noted just after Lemma 17.1 in \cite{Casals-RuizReal24} (or see \cite[Sec. 5]{BehrstockHierarchically19}), $F_U\times\{x\}$ is uniformly hierarchically quasiconvex in $(X,\mathfrak S)$, and admits a gate map $\gate:X\to F_U\times \{x\}$ that is uniformly coarsely lipschitz and has the following properties, for all $z\in X$:
		\begin{itemize}
			\item If $V\in\mathfrak S$ and $V\nest U$, then $\dist_V(z,\gate(z))\leq r_0E$.
			\item If $V\in\mathfrak S$ and $V\orth U$, then $\dist_V(\gate(z),x)\leq r_0E$.
			\item If $V\in\mathfrak S$ and $U\propnest V$ or $U\transverse V$, then $\dist_V(\gate(z),\rho^U_V)\leq r_0E$.
		\end{itemize}
		In particular, $\dist_V(x,\gate(x))\leq r_0E$ for all $V\in\mathfrak S$, so $\dist(x,\gate(x))\leq \theta_u(r_0E)$.  Similarly, $\dist_V(y,\gate(y))\leq 10E+r_0E$ for all $V$, so $\dist(y,\gate(y))\leq\theta_u(10E+r_0E)$.  Now, $D(\gate(x),\gate(y))\leq 1$ by the definition of $D$, so $D(x,y)\leq \theta_u(r_0E)+\theta_u(10E+r_0E)+1$, which depends only on the HHS parameters.  
		
		We have shown that the identity map $(X,\widehat{\dist})\to (X,\widecheck{\dist})$ and its inverse are both uniformly coarsely lipschitz, so they are quasi-isometries, as required.
	\end{proof}
	
	\begin{cor}\label{cor:uniform-cone-off}
		Let $(X,\mathfrak S)$, $A$, and $\mathfrak U$ be as in Definition \ref{defn:cone-off-data}.  Then there is an HHS $(\newcone X,\mathfrak S-\mathfrak U)$, with HHS parameters depending only on the HHS parameters of $(X,\mathfrak S)$, such that $A$ acts on $(\newcone X,\mathfrak S-\mathfrak U)$ by HHS automorphisms and on $\newcone X$ is by isometries, and there is an $A$--equivariant quasi-isometry $\newcone X\to \oldcone X$ with constants depending only on the HHS parameters.
		
		Finally, the identity map $c:X\to\newcone X$ is coarsely lipschitz and quasimedian, with the constants for both properties depending only on the HHS parameters of $(X,\mathfrak S)$.
	\end{cor}

	\begin{proof}
		By Proposition \ref{prop:cone-off-QI}, we have a constant $C$, depending only on the HHS parameters of $(X,\mathfrak S)$, such that $A$ acts by $(C,C)$--quasi-isometries on the non-canonical cone-off $\oldcone X$, and this action is by HHS automorphisms on $(\oldcone X,\mathfrak S-\mathfrak U)$.

		Lemma \ref{lem:cone-offs-are-QI} implies that the canonical cone-off $(\newcone X,\mathfrak S-\mathfrak U)$ is an HHS with uniform HHS parameters, defining projections $\pi'_W:\newcone{X}\to \mathcal CW,\ W\in\mathfrak S-\mathfrak U$ by composing the projections $\pi_W:\oldcone X\to \mathcal CW$ with the identity $\newcone X\to\oldcone X$ (see also \cite[Prop. 1.10]{BehrstockHierarchically19}).  The $A$--action on $\newcone X$ is isometric by Lemma \ref{lem:cone-off-action}. The action is by HHS automorphisms since that holds for $(\oldcone X,\mathfrak S-\mathfrak U)$, and passing to $\newcone X$ did not change the index set or the hyperbolic spaces, and did not change the projection maps at the set-theoretic level. 
		
		Finally, the map identity map $c$ is coarsely lipschitz with uniform constants as an immediate consequence of Definition \ref{defn:modern-cone-off} and the fact that $X$ is a uniform quasigeodesic space.  The fact that $c$ is quasimedian follows from the characterisation of the coarse median in an HHS (Construction \ref{cons:coarse-median}) and the fact that, for $U\in\mathfrak S-\mathfrak U$, the projections to $\mathcal CU$ coincide for the HHS structures on $X$ and $\newcone X$.  This completes the proof.
	\end{proof}
	
	\begin{remark}\label{rem:assume-graph}
		In our applications, the input HHS $X$ will be the vertex set of a graph, and $\dist$ will be the usual graph metric obtained by assigning edges length $1$.  In this setting, $\newcone X$ is obtained by adding unit length edges joining any two vertices $x,y$ with $\widecheck{D}(x,y)\leq 1$, and taking the resulting graph metric on the vertex set.
	\end{remark}
	
	\subsection{Coarse injectivity and consequences}\label{subsec:coarse-injectivity}
	Fix an HHS $(X,\mathfrak S)$ and let $\sigma$ be the metric on $X$ given by Theorem \ref{thm:coarse-injective}, which is $M_1$--coarsely injective; coarse injectivity originates in \cite{ChepoiPacking07} and we refer to \cite[Sec. 1.1]{HaettelCoarse23} for an explicit definition, which we do not need here.  The same theorem says that $\sigma$ makes the identity $(X,\dist)\to (X,\sigma)$ an $M_2$--quasi-isometry, where $M_1,M_2$ depend only on the HHS parameters, and that any subgroup of $\Isom(X,\dist)$ acting on $(X,\mathfrak S)$ by HHS automorphisms also acts by isometries on $(X,\sigma)$.
	
	Let $(\injhull(X),d_\infty)$ be the injective hull of $(X,\sigma)$ and recall that the canonical isometric embedding $i: (X,\sigma) \hookrightarrow (\injhull(X),d_\infty)$ is $\Isom(X,\sigma)$-equivariant (see \cite{LangInjective13}).  Hence $G$ acts by isometries on $\injhull(X)$ and $i$ is $G$--equivariant.  The next corollary follows from the proof of Proposition~3.12 in \cite{ChalopinHelly25}.
	
	\begin{cor} \label{cor: coarsely dense in injective hull}
		The image $i(X)$ is $M_1$-coarsely dense in $\injhull(X)$.
	\end{cor}
	
	The main consequence of coarse injectivity that we shall apply later is:
	
	\begin{cor}[Uniformly bounded orbits] \label{cor: uniformly bounded orbits}
		Let $(X, \mathfrak{S})$ be an HHS and assume $X$ is a graph. There exists a constant $B$ depending only on the HHS parameters such that, if $H \leq \Isom(X,\dist)$ is a group of HHS automorphisms of $(X,\mathfrak S)$ acting on $X$ with bounded orbits, then there exists $x_0 \in {X}$ such that $\diam(H \cdot x_0) \leq B$.
	\end{cor}
	
	\begin{proof}
		By \cite[Proposition~1.2]{LangInjective13}, there is a point $f \in \injhull(X)$ which is fixed by $G$. Let $x \in X$ be such that $d_\infty(i(x), f) \leq M_1$. Then 
		\[
		\diam_d(G \cdot x) \leq M_2\diam_\sigma(G \cdot x) + M_2 \leq 2 M_2 M_1 + M_2 \eqcolon B,
		\]
		which depends only on the HHS parameters since the same is true of $M_1,M_2$. 
	\end{proof}
	
	\subsection{Simplifying actions}\label{subsec:simplifying-actions}
	Here are two more or less standard lemmas that will be convenient in places where we want to perturb maps to make them equivariant.
	
	\begin{lem}\label{lem:assume-graph}
		For every $q\geq 1$, there exists $r\geq 1$, depending only on $q$, such that the following holds.  Let $(X,\dist)$ be a $q$--quasigeodesic space on which the group $G$ acts by isometries.  Then there is a graph $\Omega$ such that $G$ acts on $\Omega$ by graph automorphisms and there is a $G$--equivariant map $f:X\to \Omega$ that is an $r$--coarsely surjective $(r,r)$--quasi-isometry, where $\Omega$ is equipped with the usual path-metric in which all edges have unit length.  Moreover, $\Omega$ has vertex set $X$ and we may take $f$ to be the inclusion $X\to \Omega$.
	\end{lem}
	
	\begin{proof}
		This follows from \cite[Lem. 3.B.6]{Cornulier:metric}.
	\end{proof}
	
	\begin{remark}\label{rem:HHS-are-graphs}
		Let $(X,\mathfrak S)$ be an HHS on which $G$ acts isometrically by HHS automorphisms.  Then Lemma \ref{lem:assume-graph} allows us to remetrise $X$ as the vertex set of a graph, without changing the projections $\pi_U:X\to \mathcal CU$; by \cite[Prop. 1.10]{BehrstockHierarchically19}, $(X,\mathfrak S)$ is still an HHS, it is immediate from Definition \ref{defn:HHS-automorphism} that the action is still by isometric HHS automorphisms, and the new HHS parameters just depend on the old ones, via the quasi-isometry constant $r$ from Lemma \ref{lem:assume-graph}, which just depends on the old HHS parameter $q$ from Subsection \ref{subsec:parameters}.   In a few places, it will be convenient to assume that $X$ is a graph, and this observation enables it.
	\end{remark}
	
	\begin{lem}[Freeing actions]\label{lem:graph-free}
		Let $\Omega$ be a simplicial graph and let $G$ act by graph automorphisms on $\Omega$ without inverting edges.  Then there is a graph $\widehat\Omega$, a free $G$--action $G\to\Aut(\widehat\Omega)$, and a $G$--equivariant surjective map $f:\widehat\Omega\to\Omega$ that is a $(1,10)$--quasi-isometry taking vertices to vertices and edges to edges or vertices.
	\end{lem}
	
	\begin{proof}
		Let $V$ be the vertex set of $\Omega$ and let $E$ be the set of edges.  Let $\Delta(G)=\mathrm{Cay}(G,G)$ be a complete graph with a free, vertex-transitive $G$--action.  The $G$--action on $G$ by left multiplication extends to an action on $\Delta(G)$, and we are given a $G$--action on $\Omega$.  These two actions yield a diagonal action on the complex $\Delta(G)\times \Omega$ by combinatorial automorphisms.  
       		
		Let $\{v_i\}_{i\in I}\subseteq V$ contain exactly one element of $V$ in each $G$--orbit.  Define $\widehat V\subseteq G\times V$ as follows.  Given $g\in G$ and $v\in V$, let $i\in I$ be the unique element such that $v=hv_i$ for some $h\in G$.  Then $(g,v)\in\widehat V$ if and only if $h^{-1}g\in \stabilizer_G(v_i)$.  Note that if $(g,hv_i)\in\widehat V$ and $a\in G$, then $(ag,ahv_i)\in\widehat V$, so the $G$--action preserves $\widehat V$.  The set $\widehat V$, with this $G$--action, will be the vertex set of $\widehat \Omega$.  The $G$--action on $\widehat V$ is free because the action on the first coordinate is free.
		
		Declare $(g,v),(h,w)\in \widehat V$ to be adjacent in  $\widehat\Omega$ if and only if $v=w$ or $\{v,w\}\in E$.  The $G$--action on $\widehat V$ preserves these relations and thus extends to a free $G$--action on $\widehat\Omega$.  
		
		The natural projection $\Delta(G)\times \Omega\to \Omega$ restricts to a $G$--equivariant map $f:\widehat\Omega\to\Omega$ taking vertices to vertices.  Each edge of $\widehat \Omega$ joining vertices with the same second coordinate gets collapsed to a vertex, and the remaining edges are sent isometrically to edges.  Hence $f$ is $1$--lipschitz.  On the other hand, the map $\bar f:V\to\widehat V$ given by $\bar f(hv_i)=(h,hv_i)$ (the map is determined by choices of coset representatives for the $\stabilizer_G(v_i)$) defines a section of $f$ that is 1--lipschitz and satisfies $\dist_{\widehat\Omega}(x,\bar ff(x))\leq 1$ for all $x$.  This completes the proof. 
	\end{proof}

    \subsection{Virtually abelian and crystallographic groups}\label{subsec:virtually-abelian-crystallographic}
    We recall some standard facts about virtually abelian and crystallographic groups. Fix $n\ge0$ and let $A$ be a virtually $\integers^n$ group, so that there is a short exact sequence
    $A_0\hookrightarrow A\twoheadrightarrow F$, where $F$ is finite and $A_0\cong\integers^n$.

    \begin{defn}[Point group of $A$]\label{defn:point-group}
     The conjugation action of $A$ on $A_0$ gives a homomorphism $A\to GL_n(\integers)$ whose image we denote by $\bar F$ and call the \emph{point group} of $A$.
    \end{defn}
    
     Note that $A\to GL_n(\integers)$ factors through the quotient $A\to F$, so $\bar F$ is finite.  Moreover, if $A_0'\leq A_0$ is another finite-index free abelian subgroup that is normal in $A$, then the conjugation actions of $A$ on $A_0$ and $A_0'$ make the map $\iota:A_0'\hookrightarrow A_0$ an $A$--equivariant map, where $A$ acts on $A_0'$ via the point group $\bar F'=\image(A\to Aut(A_0'))$ and on $A_0$ via $\bar F$.  But $\iota$ extends to an isomorphism $\reals^n\to\reals^n$, which therefore conjugates $\bar F'$ to $\bar F$.  This implies that the point group $\bar F\leq GL_n(\reals)$ is, up to conjugacy, independent of the choice of $A_0$.

     Virtually $\integers^n$ groups are related to $n$--dimensional crystallographic groups by the following classical fact \cite{Zassenhaus}:

     \begin{lem}\label{lem:virtually-abelian-crystallographic}
     If $A$ is a virtually $\integers^n$ group, then there is a finite normal subgroup $K\leq A$ and a monomorphism $\phi:A/K\to \Isom(\Euclidean^n)$ whose image is $n$--dimensional crystallographic, i.e. $\phi(A/K)$ acts faithfully on $\Euclidean^n$ and this action is proper and cocompact.
     \end{lem}

    In Section \ref{sec:applications}, we will use these facts via the following lemma:

    \begin{lem}\label{lem:finitely-many-crystallographic}
    Let $n\geq 0$ and let $K$ be a finite group.  Then there are finitely many abstract isomorphism classes of groups $A$ such that $K\triangleleft A$ and $A/K$ is isomorphic to an $n$--dimensional crystallographic group.
    \end{lem}

    \begin{proof}
    Fix $n$.  Let $A$ be as in the statement and let $C=A/K$.  By Bieberbach's theorems, $C$ belongs to one of finitely many isomorphism classes.  Since $K$ is finite and $A$ is finitely generated, there are finitely many homomorphisms $C\to \Out(K)$.  The centre $Z(K)$ of $K$ is finite, so for each way of regarding $Z(K)$ as a $C$--module, $H^2(C,Z(K))$ is finite, and hence yields at most finitely many possibilities for $A$.
    \end{proof}

 In \cite{PetytUnbounded23}, Petyt and Spriano characterised crystallographic HHGs, using the following notion, which was also used in \cite{Hagen:crystallographic} and \cite{hodaCrystallographicHellyGroups2023}, respectively, to characterise crystallographic groups that are cocompactly cubulated and act geometrically on injective spaces.
    
   \begin{defn}[Hyperoctahedral group]\label{defn:hyperoctahedral}
	Let $A$ be a virtually $\integers^n$ group and let $\bar F$ be the point group of $A$.  Then $A$ is \emph{hyperoctahedral} if $\bar F$ is conjugate in $GL_n(\reals)$ into $O_n(\integers)$.
\end{defn}

We will use this notion in Section \ref{sec:applications}, when we study crystallographic \emph{subgroups} of HHGs.
	
	\newsection{An equivariant median model}\label{sec:median-model}
	Our aim in this section is to prove the following equivariant version of the cubical approximation theorem.  By \cite[Rem. 1.4]{BehrstockHierarchically19}, in this section, we can and shall assume that the maps $\pi_U:X\to \mathcal CU,\ U\in\mathfrak S$ are $E$--coarsely surjective. 
	
	\begin{thm}[store=median approx., note=Equivariant median model]\label{thm: median hulls}
		Let $(X, \mathfrak S)$ be an HHS and suppose there exists $M \geq 0$ and $p^-,p^+ \in X \cup \partial X$ such that $X = H_{M,z_0}(p^-,p^+)$ for any $z_0 \in X$. Let $G$ be a group acting on $X$ by HHS automorphisms and, if $p^-, p^+ \in \partial X$, suppose $G$ is virtually abelian. Then there is a complete, connected, finite rank, proper median space $Q$, equipped with an action of $G$, and a $G$-equivariant $L'$-quasimedian $(L',L')$-quasi-isometry $\Psi: X \rightarrow Q$, where the constant $L' \geq 1$ depends only on the HHS parameters and $M$.
	\end{thm}
	
	To support the proof of Theorem \ref{thm:HHS-semisimple-coarse-minset}.\eqref{it: M is coarsely X(A)} we also prove the following in Section \ref{subsec:cocompact}.
	
	\begin{prop}\label{prop:short-mountains}
		Let $(X,\mathfrak S),M,p^\pm,$ and $G$ be as in Theorem \ref{thm: median hulls}. Suppose that $p^\pm\in\boundary X$, and $\support(p^-)=\support(p^+)$, and $\pi_U(G\cdot z_0)$ is unbounded for all $U\in\support(p^\pm)$ and some (equivalently any) $z_0 \in X$.  Let $Q$ be the $G$--median space provided by Theorem \ref{thm: median hulls}.  Then there exist median spaces $Q_1,\ldots,Q_r$ such that:
		\begin{itemize}
			\item each $Q_i$ is quasi-isometric to $\reals$, and
			\item $G$ acts by isometries on $\prod_{i=1}^rQ_i$, preserving the product structure, and there is a $G$--equivariant isometry $Q\to\prod_{i=1}^rQ_i$,
			\item if $G' \leq G$ is the finite index subgroup which acts trivially on the set of factors, then the $G'$-action on each $Q_i$ is cocompact.
		\end{itemize}
		Moreover, there exists $B<\infty$ such that $\diam(\mathcal CV)<B$ for all $V\in\mathfrak S-\support(p^\pm)$.
	\end{prop}
	
	\begin{remark}
		The quasi-isometry constants and the constant $B$ in Proposition \ref{prop:short-mountains} are not uniform; they depend on more than the HHS constants and $M$.
	\end{remark}

	\subsection{Quasi-linear domains}\label{subsec:quasilinear}
	
	We need to fix quasi-isometries from the domains of $\mathfrak S$ to lines, rays and intervals in an equivariant way. We start with some lemmas.
	
	\begin{lem} \label{lem: orthogonal supports are bounded}
		Let $(X,\mathfrak S)$ be an HHS and let $p^\pm,z_0,G$ and $M$ be as in Theorem \ref{thm: median hulls}. 
		If $p^+ \in \partial X$ (resp. $p^- \in \partial X$) and $U \in \support^\perp(p^+)$ (resp. $U \in \support^\perp(p^-)$) then $\diam(\mathcal CU) \leq \ddot M$, where $\ddot M\geq M$ depends on $M$ and the HHS parameters only.
	\end{lem}
	
	\begin{proof}
		Since $\pi_U(X)=\pi_U(H_{M,z_0}(p^-,p^+))$ for all $z_0\in X$ and $\pi_U$ is $E$--coarsely surjective, $\mathcal CU$ is $(M+E)$--Hausdorff close to $\hull_{U,z_0}(p^-,p^+)$ for all $U\in\mathfrak S$, by Definition \ref{defn:pair-hull}.  
		
		Suppose $p^+\in\partial X$ and $U\in \support^\orth(p^+)$.  Fix $z_0\in X$.  By Definition \ref{defn:boundary-projection}, $\pi_{U,z_0}(p^+)=\pi_U(z_0)$.  By Definition \ref{defn:pair-hull}, $\hull_{U,z_0}(p^-,p^+)$ is therefore the union of all $(1,20E)$--quasigeodesics from $\pi_U(z_0)$ to $q:=\pi_{U,z_0}(p^-)$.  If $p^-\in\boundary X$ and $U\in\support^\orth(p^-)$, then $q=\pi_U(z_0)$, so $\hull_{U,z_0}(p^-,p^+)\subseteq \mathcal N_{100E}(\pi_U(z_0))$, which has uniformly bounded diameter.
		
		Otherwise, $q$ is independent of the choice of $z_0$ (see Definition \ref{defn:boundary-projection}).  In this case, since $X=H_{M,z_0}(p^-,p^+)$ for every choice of $z_0$, the diameter of $\hull_{U,z_0}(p^-,p^+)$, and hence $\mathcal CU$, is bounded in terms of $M$ and $E$.   
	\end{proof}
	
	\begin{lem}\label{lem:quasi-line-action}
		For each $k_1 \geq 1$, there exists $k \geq 1$ such that the following holds. Let $A$  be a group and let $\qline$ be a $k_1$-quasi-line with an isometric $A$--action.  Suppose that either the action has bounded orbits or $A$ is virtually abelian. Then $A$ acts isometrically on $\mathbb{R}$ and there is an $A$-equivariant $k$-quasi-isometry $f: \qline \rightarrow \mathbb{R}$.
	\end{lem}

	\begin{proof}
		Fix $k_1 \geq 1$ and let $A, \qline$ be as in the statement. Assume momentarily that $A$ fixes the ends of $\qline$. By \cite[Cor. 4.12]{KerrTree23}, there exists $k_2 \geq 1$, depending only on $k_1$, and a $(1,k_2)$-quasi-isometry $f_1: \qline \rightarrow \mathbb{R}$. Without loss of generality, $0 = f_1(x_0)$ for some $x_0 \in \qline$. For all $a \in A$, define $\ell(a) \coloneqq \lim_{n \in \mathbb{N}} \frac{1}{n} f_1(a^n x_0)$. 
		Since $f_1$ is a rough isometry, $|\ell(a)| = \lim_{n \in \mathbb{N}} \frac{1}{n} \dist_\qline(x_0, a^n x_0) < \infty$ for all $a \in A$.
		
		\begin{claim}\label{claim:homomorphism}
			The map $\ell: A \rightarrow \mathbb{R}$ is a homomorphism.
		\end{claim}
		\begin{proof}
			\renewcommand{\qedsymbol}{$\blacksquare$}
			Let $a,b \in A$. Then, for all $n \in \mathbb{N}$,
			\[
			f_1(a^n x_0) + f_1(b^n x_0) - 3k_2 \leq f_1(a^n b^n x_0) \leq f_1(a^n x_0) + f_1(b^n x_0) + 3k_2,
			\]
			so the claim follows from \cite[Prop. 2.65]{Calegari:scl} in the virtually abelian case.  Otherwise, our assumption that $A$ has bounded orbits and fixes the endpoints implies that $\ell$ is the trivial map, which is a homomorphism.
			
		\end{proof}

		There exists a constant $k_3$ depending only on $k_1$ such that, for any ball $B_\qline(x,k_3) \subseteq \qline$, the complement $\qline - B_\qline(x,k_3)$ has precisely two unbounded components.
		
		\begin{claim}\label{claim:quasi-R}
			For all $a \in A$ and $x \in \qline$ we have $f_1(ax) - f_1(x) - 2k_2 - k_3 \leq \ell(a) \leq f_1(ax) - f_1(x) + 2k_2$.
		\end{claim}
		\begin{proof}
			\renewcommand{\qedsymbol}{$\blacksquare$}
			Firstly, 
			\[
			|\ell(a)| = \lim_{n \rightarrow \infty} \frac{d_\qline(x_0, a^nx_0)}{n} \leq \lim_{n \rightarrow \infty} \frac{2d_\qline(x_0, x) + \dist_\qline(x, a^nx)}{n} = \lim_{n \rightarrow \infty} \frac{d_\qline(x, a^nx)}{n}.
			\]
			A symmetric argument shows that $\lim_{n \rightarrow \infty} \frac{1}{n} \dist_\qline(x, a^nx) \leq |\ell(a)|$. It follows that $\ell(a) = \lim_{n \rightarrow \infty} \frac{1}{n} (f_1(a^n x) - f_1(x))$. Now, using the triangle inequality and the fact that $f_1$ is a rough isometry,
			\[
			f_1(a^n x) - f_1(x) \leq \sum_{i=0}^{n-1} f_1(a^{i+1}x) - f_1(a^ix) \leq nf_1(ax_0) + 2nk_2
			\]
			for all $n \in \mathbb{N}$, which proves the second inequality. 
			
			To prove the first inequality, observe that, if $d_\qline(x, ax) > k_3$ then either $f_1(a^nx) < f_1(a^{n+1}x)$ for all $n \in \mathbb{Z}$ or $f_1(a^nx) > f_1(a^{n+1}x)$ for all $n \in \mathbb{Z}$. It follows that 
			\[
			f_1(a^n x) - f_1(x) = \sum_{i=0}^{n-1} f_1(a^{i+1}x) - f_1(a^ix) \geq nf(ax_0) - 2nk_2. \qedhere
			\]
		\end{proof}
		
		Let $\{x_i\}_{i\in I}\subseteq \qline$ contain exactly one point in each $A$--orbit.  Define $f:\qline\to \reals$ by $f(ax_i)=\ell(a)+f_1(x_i)$ for $i\in I$ and $a\in A$.  Observe that if $ax_i=x_i$, then $a^nx_i=x_i$ for all $n$, and hence $f_1(a^nx_0)$ is bounded, so $\ell(a)=0$.  Thus, if $a,b\in A$ satisfy $ax_i=bx_i$, then $\ell(a)=\ell(b)$, and therefore $f$ is a well-defined map.  Claim \ref{claim:quasi-R} implies that $f$ is a quasi-isometry with constants depending only on $k_1$.  Finally, the map $A\times\reals \ni (a,t)\mapsto \ell(a)+t$ defines an isometric action of $A$ on $\reals$, by Claim \ref{claim:homomorphism}, and $f$ is equivariant with respect to this action, by construction.  
		
		If $A$ permutes the ends of $\qline$, let $A'\leq A$ be the index--$2$ subgroup fixing the ends and let $f:\qline\to \reals$ and $\ell:A'\to\reals$ be the homomorphism obtained by applying the preceding construction to $A'$.  Then the standard induced representation construction gives the desired action $A\to \reals\rtimes\integers/2$ making the map $\qline\to \reals$ equivariant.
	\end{proof}
	
	\begin{prop}\label{prop:quasi-linear-domains}
		Let $G$ and $(X, \mathfrak S)$ be as in Theorem~\ref{thm: median hulls}. There exists a constant $L \geq 1$ depending only on the HHS parameters and $M$ such that the following holds. There exists an interval, line or ray $T_U$ for each $U \in \mathfrak S$ and an $L$--coarsely surjective $(L,L)$--quasi-isometry $f_U: \mathcal CU \rightarrow T_U$ such that:
		\begin{itemize}
			\item for all $g \in G$ and $U \in \mathfrak S$ there is an isometry $g: T_U \rightarrow T_{gU}$;
			\item for all $g,h \in G$ and $U \in \mathfrak S$ the isometry $gh:T_U \rightarrow T_{ghU}$ is the composition $g \circ h$. 
			\item for all $U \in \mathfrak S$ the map $f_U$ is $\stabilizer_G(U)$--equivariant. 
            \item $T_U$ is unbounded if and only if $U\in\support(q)$ for some $q\in\{p^-,p^+\}\cap\boundary X$.
		\end{itemize}
		Moreover, letting $p^\pm\in X\cup\boundary X$ be as in Theorem \ref{thm: median hulls}, $G$ has a finite index subgroup fixing  $\{p^-,p^+\}\cap\boundary X$ pointwise.
	\end{prop}
	
	\begin{proof}
        In this proof, \emph{uniform} quantities are those bounded in terms of the HHS parameters and $M$. Recall that $X=H_{M,z_0}(p^-,p^+)$ for all $z_0\in X$, by hypothesis.  For each $U\in\mathfrak S$, our standing assumption ensures that $\mathcal CU=\neb_{E}(\pi_U(H_{M,z_0}(p^-,p^+)))$.  By Definition \ref{defn:pair-hull}, it follows that $\mathcal CU=\neb_{E+M}(\hull_{U,z_0}(p^-,p^+))$. 
        
         We first deduce from this that each $\mathcal CU$ is uniformly quasi-isometric to an interval, line, or ray.  By the preceding observation, it suffices to show that each $\hull_{U,z_0}(p^-,p^+)$ has this property.  
		If $p^-,p^+\in X$, then by Remark \ref{rem:hull-is-hull}, $\mathcal CU$ is at uniformly bounded Hausdorff distance from some geodesic joining $\pi_U(p^-)$ to $\pi_U(p^+)$, as required.  
		
		Therefore, suppose that $p^+\in\boundary X$. If $p^+\in\boundary X$ but $p^-\in X$, then for all $U$, $\mathcal CU$ is, by the second point in Definition \ref{defn:pair-hull}, at uniformly bounded Hausdorff distance from a geodesic interval or ray, as required.  If $p^-,p^+\in\boundary X$, then the third point in Definition \ref{defn:pair-hull} yields the same conclusion.  (If $U\in\support(p^-)\cap\support(p^+)$, we implicitly used that $\pi_U(p^-) \neq \pi_U(p^+)$, but this follows from Definition \ref{defn:pair-hull} since $X\neq\emptyset$ and $X=H_{M,z_0}(p^-,p^+)$.)
		
		We now construct $T_U$ and $f_U$ for each $U\in \mathfrak S$.  First, we have just shown that there is a uniform constant $k_1\ge1$ such that for all $U\in\mathfrak S$, $\mathcal CU$ is $(k_1,k_1)$--quasi-isometric to an interval, ray, or line.  Second, we will use the following observation: if $p^-,p^+\in X$, then, as noted above, $\mathcal CU$ is bounded for all $U$.  Hence, if any $\mathcal CU$ is unbounded, then one of $p^\pm\in\boundary X$ and therefore $G$ is virtually abelian, by the hypothesis in Theorem \ref{thm: median hulls}.
		
		Fix $U\in\mathfrak S$.  We define $T_U$ and $f_U$ as follows.  

        If $\mathcal CU$ is two-ended, then Lemma \ref{lem:quasi-line-action} provides a homomorphism $\phi_U:\stabilizer_G(U)\to \Isom(\reals)$ and a quasi-isometry $f_U:\mathcal CU\to \reals=:T_U$ such that $f_U(gc)=\phi_U(g)f_U(c)$ for all $g\in\stabilizer_G(U)$ and $c\in \reals$.

        If $\mathcal CU$ is either one-ended or bounded, then at least one of $\pi_U(p^\pm)$ is a uniformly bounded subset of $\mathcal CU$, and the set $\pi_U(p^-)\cup \pi_U(p^+)$ is $\stabilizer_G(U)$--invariant and each $g\in \stabilizer_G(U)$ either stabilises $\pi_U(p^\pm)$ or sends it bijectively to $\pi_U(p^\mp)$.  Form a new space $\mathcal C'U$ as follows: if $\pi_U(p^+)$ (respectively, $\pi_U(p^-)$) is contained in $X$, then add a vertex $v^+$ (respectively $v^-$) joined by an edge to each element of $\pi_U(p^+)$ (respectively, $\pi_U(p^-)$), and then attach a copy of $[0,\infty)$ by identifying $0$ with $v^+$; if $v^-$ has been defined, also attach a differente copy of $[0,\infty)$ there.  Then $\mathcal C'U$ is a uniform quasiline in which $\mathcal CU$ is isometrically embedded as a $G$--invariant subset, and the $\stabilizer_G(U)$ action extends naturally over $\mathcal C'U$.  Apply Lemma \ref{lem:quasi-line-action} to get a $\stabilizer_G(U)$--action on $\reals$ and an equivariant map $\mathcal C'U\to\reals$, which we then restrict to $\mathcal CU$.

		Applying the above construction to $G$--orbit representatives in $\mathfrak S$ and extending equivariantly in the usual way yields isometries $g:T_U\to T_{gU}$ with the equivariance properties required by the statement. 
        
         Finally, we characterise the $U \in \mathfrak S$ for which $\mathcal CU$ (equivalently $T_U$) is unbounded.  If $p^\pm\in\boundary X$ and $U\in\support(p^\pm)$, then $\mathcal CU$ is unbounded since $\pi_{U,z_0}(p^\pm)\in\boundary \mathcal CU$ in this case.  Conversely, suppose that $\mathcal CU$ is unbounded.  As noted above, this implies that we can assume $p^+\in\boundary X$.  Lemma \ref{lem: orthogonal supports are bounded} implies that $U\not\in\support^\orth(p^\pm)$.  Hence it suffices to consider the case where $U\not \in\support(p^\pm)\cup\support^\orth(p^\pm)$, but in this case, Definition \ref{defn:boundary-projection} and Definition \ref{defn:pair-hull} imply that $\mathcal CU$ is bounded.
		
		Since $G$ acts by HHS automorphisms, the set of $U\in\mathfrak S$ for which $\mathcal CU$ is unbounded is $G$--invariant, which implies that $\support(p^-)\cup\support(p^+)$ is $G$--invariant (here we take $\support(p^\pm)=\emptyset$ for $p^\pm\in X$).  Since $\support(p^\pm)$ consists of pairwise-orthogonal elements, $|\support(p^-)\cup\support(p^+)|$ is finite (in fact bounded in terms of the complexity of $(X,\mathfrak S)$), and hence $G'=\bigcap_{U\in\support(p^-)\cup\support(p^+)}\stabilizer_G(U)$ has finite (in fact, uniformly bounded) index in $G$.  For each $U\in\support(p^-)\cup\support(p^+)$, the $G'$--action on $\mathcal CU$ extends, as usual, to an action on $\boundary\mathcal CU$, which consists of at most two points.  Hence passing to the intersections of the kernels of the $G'$--actions on these $\boundary\mathcal CU$ yields the desired finite-index subgroup. 
	\end{proof}
	
	\subsection{The product of trees $\mathcal Y$}\label{subsec:product-of-trees}
	Let $L$ and $\{f_U:\mathcal CU\to T_U:U\in\mathfrak S\}$ be as in Proposition \ref{prop:quasi-linear-domains}. Let $U\in\mathfrak S$.  If $U\not\in\support^\orth(p^\pm)$, then by Definition \ref{defn:boundary-projection}, $\pi_U(p^\pm)$ is defined independently of $z_0$, and $\pi_{U,z_0}(p^\pm)=\pi_U(z_0)$ otherwise.  Up to enlarging $L$ by a uniformly bounded amount, we can (and do) assume that, if $\pi_U(p^\pm) \subseteq \mathcal CU$, then $f_U(\mathcal N_{\pcollapse}(\pi_U(p^\pm)))$ is an endpoint of $T_U$. Denote this point by $p_U^\pm$.
	
    If $\pi_U(p^\pm) \in \partial \mathcal CU$ then let $(x_n)_{n \in \mathbb N} \subseteq \mathcal CU$ be a sequence which converges to $\pi_U(p_U^\pm)$ and let $p_U^\pm \in \partial T_U$ be the limit of $(f_U(x_n))_{n \in \mathbb N}$ in $\boundary T_U$, which is independent of the choice of $(x_n)_{n \in \mathbb N}$. 
	
	Enlarging $L$ be a uniform amount again, we can fix, for each $U \in \mathfrak S$, a quasi-inverse $f_U^{-1}: T_U \rightarrow \mathcal CU$ for $f_U$ and assume it is an $(L,L)$-quasi-isometry. In general it will not be possible to make $f_U^{-1}$ $\stabilizer_G(U)$--equivariant.
	
	\begin{defn}\label{defn:support-nested}
		Let $C_0,C_1$ be as in Section \ref{sec:coarse-median}. Fix constants $E' \coloneqq L(5E + L + 1)$ and $K \coloneqq 100L(\pcollapse + E) + 100E' + L^2  + L(C_0 + C_1)+\ddot M$. Let $\mathfrak U \coloneqq \{U \in \mathfrak S : \diam(T_U) \geq K\}$. We say that $U \in \mathfrak U$ is $\nest_{\mathfrak U}$\emph{-minimal} if it is $\nest$-minimal in $\mathfrak U$.
	\end{defn}
	
	\begin{remark}\label{rem:support-nested-prooperties}
		\begin{itemize}
			\item[]
			\item The set $\mathfrak U$ is $G$--invariant, by Proposition \ref{prop:quasi-linear-domains}.
			\item If $U \in \mathfrak U$ and $p^\pm \in \partial X$ then it follows from Lemma~\ref{lem: orthogonal supports are bounded} that $U \notin \support^\perp(p^\pm)$. Therefore $\pi_U(p^\pm)$ is well-defined (independently of a choice of basepoint) for any $p^\pm\in X\cup\boundary X$ and $U\in\mathfrak U$.\qedhere
		\end{itemize}
	\end{remark}
	
	Let $\mathcal Y \coloneqq \prod_{U \in \mathfrak U} T_U$ and define an action of $G$ on $\mathcal Y$ as follows. Given $(x_U)_{U \in \mathfrak U} \in \mathcal Y$ and $g \in G$, let $y_U \coloneqq g x_{g^{-1} U}$ for each $U \in \mathfrak U$ and let $g (x_U)_{U \in \mathfrak U} \coloneqq (y_U)_{U \in \mathfrak U}$, where $g:T_{g^{-1}U}\to T_U$ is the isometry from Proposition \ref{prop:quasi-linear-domains}.
	
	\begin{defn}[Interval projections]
		Given distinct $U, V \in \mathfrak S$ that are not orthogonal, define the \textit{interval projections} $\delta^U_V, \delta^V_U$ as follows:
		\begin{itemize}
			\item If $U \transverse V$ or $U \propnest V$ then define $\delta^U_V \coloneqq \hull(f_V(\rho^U_V))$.
			\item If $U \propnest V$ then define $\delta^V_U: T_V \rightarrow 2^{T_U}$ by $\delta^V_U \coloneqq f_U \circ \rho^V_U \circ f_V^{-1}$.
            \qedhere
		\end{itemize}
	\end{defn}

	\begin{lem}[Consistent projections]
		\label{lem: pseudo HHS}
		For all $U,V, W \in \mathfrak U$ we have:
		\begin{enumerate}
			\item \label{it: transverse consistency in Y} If $U \transverse V$ then either $\delta^U_V = \{p_V^-\}$ and $\delta^V_U = \{p_U^+\}$, or $\delta^U_V = \{p_V^+\}$ and $\delta^V_U = \{p_U^-\}$.
			\item \label{it: bounded projections in Y} If $U \propnest V$ then $\diam(\delta^U_V) \leq E'$.
			\item \label{it: non-transverse nested domains in Y} If $V,W \propnest U$ are not transverse then $\diam_{T_U}(\delta^V_U \cup \delta^W_U) \leq E'$.
			\item \label{it: consistency in Y, transverse + nested} If $U \transverse V$ and $W \propnest V$ and $U \not\orth W$ then $\diam_{T_U}(\delta^V_U \cup \delta^W_U) \leq E'$.
            \item \label{it: consistency in Y: transverse + orthogonal} If $U \orth V$ and $U,V \not\orth W$ and $W \not\propnest U,V$, then $\diam_{T_W}(\delta^U_W \cup \delta^V_W) \leq E'$.
			\item \label{it: BGI in Y} If $U \propnest V$ and $C$ is a connected component of $T_V - \mathcal N_{E'}(\delta^U_V)$, then $C$ contains $p_V^+$ or $p_V^-$. In the former case $\delta^V_U(C) = \{p_U^+\}$ and in the latter $\delta^V_U(C) = \{p_U^-\}$. 
			\item \label{it: BGI special case in Y} If $V,W \propnest U$ and $V \transverse W$ and $d_{T_U}(\delta^V_U, \delta^W_U) > 10E'$ then $\delta^U_V(\delta^W_U) = \delta^W_V$ and $\delta^U_V(\delta^W_U) = \{p_V^+\}$ or $\{p_V^-\}$.
		\end{enumerate}
	\end{lem}
	\begin{proof}
		Suppose $U \transverse V$.  By definition, there are sequences $(x_n^\pm)_n$ in $X$ such that $\pi_U(x_n^\pm)$ converges to $\pi_U(p^\pm)$ (or, if $\pi_U(p^\pm)\subseteq \mathcal CU$, to a uniformly bounded set that uniformly coarsely coincides with $\pi_U(p^\pm)$.  For all sufficiently large $n$,  the consistency axiom implies that either
		\[ 
		\dist_U(\rho^V_U, x_n^+), \dist_V(\rho^U_V, x_n^-) < E \quad \text{or} \quad \dist_U(\rho^V_U, x_n^-), \dist_V(\rho^U_V, x_n^+) < E,
		\]
		so 
		\[ 
		\dist_U(\rho^V_U, \pi_U(p^+)), \dist_V(\rho^U_V, \pi_V(p^-)) < E \quad \text{or} \quad \dist_U(\rho^V_U, \pi_U(p^-)), \dist_V(\rho^U_V, \pi_V(p^+)) < E.
		\]
		We assumed that $f_W(\mathcal N_\pcollapse(\pi_W(p^\pm))) = p_W^\pm$ for $W \in \{U,V\}$, so this proves Item \eqref{it: transverse consistency in Y}.
		Item \eqref{it: bounded projections in Y} holds because $f_V$ is an $(L,L)$-quasi-isometry and $\diam_U(\rho^U_V) \leq E$. Item \eqref{it: non-transverse nested domains in Y} holds similarly because $\diam_U(\rho^V_U \cup \rho^W_U) \leq 2E$ by \cite[Lemma~1.5]{DurhamBoundaries17}. Item \eqref{it: consistency in Y, transverse + nested} is an immediate consequence of \cite[Defn. 1.1.(4)]{BehrstockHierarchically19} and Item \eqref{it: consistency in Y: transverse + orthogonal} follows from \cite[Lem. 1.5]{DurhamBoundaries17}. 
		Item \eqref{it: BGI in Y} follows from the bounded geodesic image axiom and the fact that $f_U$ collapses the $\pcollapse$-neighbourhoods of $\pi_U(p^\pm)$. Item \eqref{it: BGI special case in Y} follows from Item \eqref{it: BGI in Y} and the partial realisation axiom \cite[Defn. 1.1.(8)]{BehrstockHierarchically19}.
	\end{proof}
	
	\subsection{The product of reduced trees $\widehat{\mathcal Y}$}
	
	We now fix some constants, all depending only on the HHS parameters, to be used throughout the remainder of this section:
	\begin{itemize}
		\item $R \coloneqq 50E'$;
		\item $\alpha_1 \coloneqq 10R$;
		\item $\kappa \coloneqq r_0(L \alpha_1 + LK + 2L)$, where $r_0$ is the constant from Theorem~\ref{thm:hhs_realisation};
		\item $r \coloneqq L\kappa + L + 1$.
	\end{itemize}
	
	\begin{defn}[Clusters, edges and collapsed intervals]
		Let $U \in \mathfrak U$. If $U$ is $\nest_{\mathfrak U}$-minimal then let $T_U^c \coloneqq \emptyset$. Otherwise, we say a \emph{basic cluster} in $T_U$ is one of the following:
		\begin{itemize}
			\item the closed $R$--neighbourhood of some $\delta^V_U$ for which $V \propnest U$;
			\item one of the sets $\{p_U^-\}$ or $\{p_U^+\}$.
		\end{itemize}
		
		Declare two points $t,t'\in T_U$ to be equivalent if there are basic clusters $c,c'$ such that $\dist_{T_U}(c,c')<r$ and $\{t,t'\}$ is contained in the convex hull in $T_U$ of $c\cup c'$.  Extend this to an equivalence relation by taking the transitive closure.  A \emph{cluster} in $T_U$ is an equivalence class, and we define $T_U^c$ to be the union of the clusters.

		The components of $T_U^e \coloneqq T_U - T_U^c$ are called \textit{edge components}. Let $q_U: T_U \rightarrow \widehat T_U$ be the quotient map which collapses each cluster to a point. Let $\widehat T_U^c \coloneqq q_U(T_U^c)$ and $\widehat T_U^e \coloneqq q_U(T_U^e)$.
	\end{defn}
	
	\begin{remark}\label{rem:clusters-are-connected}
		\begin{itemize}
			\item[]
			\item By construction, clusters and edge components are subintervals of $T_U$.
			\item The quotient space $\widehat T_U$ is an interval, a line or a ray, and $q_U$ restricts to an isometry on the edge components.
			\qedhere
		\end{itemize}
	\end{remark}
	
	Given $U \in \mathfrak U$, let $\widehat p_U^\pm \coloneqq \lim_{n \rightarrow \infty} q_U(t_n)$ for some sequence $(t_n)_{n \in \mathbb N} \subseteq T_U$ which converges to $p_U^\pm$. Observe that this is independent of the choice of sequence, and if $p_U^\pm \in T_U$ then $\widehat p_U^\pm = q_U(p_U^\pm)$.
	
	\begin{remark}
		\label{rem: collapsed trees}
		\begin{enumerate}
			\item[]
			\item \label{it: nest minimal isoms} If $U \in \mathfrak U$ is $\nest_{\mathfrak U}$-minimal then $q_U$ is the identity map, and $\widehat T_U=T_U=T_U^e=\widehat T_U^e$.
			\item \label{it: equivariance of shadows} If $U \transverse V$ or $U \propnest V$ and $g \in G$, then $g\delta^U_V = \delta^{gU}_{gV}$.
            \item \label{it: action on collapsed trees}	The above point implies that, for all $g \in G$ and $U \in \mathfrak U$, the isometry $g: T_U \rightarrow T_{gU}$ induces an isometry $g: \widehat T_U \rightarrow \widehat T_{gU}$ and we have $g \circ q_U 
            = q_{gU} \circ g$.\qedhere
		\end{enumerate}
	\end{remark}
	
	Let $\widehat{\mathcal Y} \coloneqq \prod_{U \in \mathfrak U} \widehat T_U$ and define:
	\begin{itemize}
		\item the canonical projections $\widehat \pi_U: \widehat{\mathcal Y} \rightarrow \widehat T_U$ for all $U \in \mathfrak U$;
		\item if $U \transverse V$ or $U \propnest V$, then $\widehat \delta^U_V \coloneqq q_V(\delta^U_V)$;
		\item if $U \propnest V$ then $\widehat \delta^V_U \coloneqq q_U \circ \delta^V_U$.
	\end{itemize}
    By Remark \ref{rem: collapsed trees}.\eqref{it: action on collapsed trees}, the action of $G$ on $\mathcal Y$ induces an action of $G$ on $\widehat{\mathcal Y}$.
	
	\begin{notation}
		If $U \transverse V$ or $U \propnest V$ then $\widehat \delta^U_V$ is a singleton $\{x\} \subseteq \widehat T_V$ (by definition, in the nested case, and by the definition plus Lemma \ref{lem: pseudo HHS}.\eqref{it: transverse consistency in Y} in the transverse case).
		We will abuse notation and refer to $x$ itself as $\widehat \delta^U_V$.
	\end{notation}
	
	The following is a consequence of Lemma~\ref{lem: pseudo HHS} and the choice of constant $R$. 
	
	\begin{lem}
		\label{lem: collapsed projections}
		For all $U,V,W \in \mathfrak U$ the following hold:
		\begin{enumerate}
			\item \label{it: transverse consistency in hatY} If $U \transverse V$ then either $\widehat \delta^U_V = \widehat p_V^+$ and $\widehat \delta^V_U = \widehat p_U^-$, or $\widehat \delta^U_V = \widehat p_V^-$ and $\widehat \delta^V_U = \widehat p_U^+$.
			\item \label{it: non-transverse nested domains in hatY} If $V,W \propnest U$ are not transverse then $\widehat \delta^V_U = \widehat \delta^W_U$. 
			\item \label{it: consistency in hatY: transverse + nested} If $U \transverse V$ and $W \propnest V$ and $U \notorth W$ then $\widehat \delta^V_U = \widehat \delta^W_U$.
			\item \label{it: BGI in hatY} If $V \propnest U$ and $C^-$ and $C^+$ are the connected components of $\widehat T_U - \widehat \delta^V_U$ containing $\widehat p_U^-$ and $p_U^+$ respectively, then $\widehat \delta^U_V(C^\pm) = \widehat p_V^\pm$.
		\end{enumerate}
	\end{lem}
	
	\subsection{The median space $Q$}\label{subsec:construction-of-Q}
	We now construct the median space $Q$ whose existence is asserted in Theorem \ref{thm: median hulls}.
	
	\begin{defn}[Consistency, the space $Q$]\label{defn:Q}
		A tuple $(x_U)_{U \in \mathfrak U} \in \widehat{\mathcal Y}$ is \textit{consistent} if, for all $U,V \in \mathfrak U$,
		\begin{itemize}
			\item if $U \transverse V$, then $x_U = \widehat \delta^V_U$ or $x_V = \widehat \delta^U_V$ and
			\item if $U \propnest V$ then either $x_V = \widehat \delta^U_V$ or $x_U \in \widehat \delta^V_U(x_V)$.
		\end{itemize}
		Let $Q \subseteq \widehat{\mathcal Y}$ be the set of consistent tuples and define $d_Q(x,y) \coloneqq \sum_{U \in \mathfrak U} \dist_{\widehat T_U}(\widehat \pi_U(x), \widehat \pi_U(y))$ for all $x,y \in Q$.
	\end{defn}

    Remark \ref{rem: collapsed trees}.\eqref{it: equivariance of shadows} and Lemma \ref{lem: collapsed projections}.\eqref{it: BGI in hatY} imply:

    \begin{lem} \label{lem: Q is G-invariant}
        The subspace $Q \subseteq \widehat Y$ is $G$--invariant.
    \end{lem}
	
	\begin{defn}[Relevant domains]\label{defn:relevant-Q}
		Given $x,y \in Q$, let $$\relevant_Q(x,y) \coloneqq \{U \in \mathfrak U : \widehat \pi_U(x) \neq \widehat \pi_U(y)\}.\qedhere$$
	\end{defn}
	
	We now establish that $d_Q$ defines a median metric on $Q$. Recall that $\hhscomp$ denotes the complexity of $(X, \mathfrak S)$.
	
	\begin{lem}
		\label{lem: Q is a metric space}
		\begin{enumerate}
			\item[]
			\item There exists $\hhscomp'$, depending only on $\hhscomp$, such that, if $x \in Q$ then for all but at most $\hhscomp'$ domains $U \in \mathfrak U$ we have $x_U \in \{\widehat p_U^-,  \widehat p_U^+\}$.\label{item:domains-with-x-interior}
			\item  For all $\widehat x, \widehat y \in Q$ the set $\relevant_Q(\widehat x, \widehat y)$ is finite.\label{item:rel-finite}
			\item \label{it: median metric and rank of Q} $(Q, \dist_Q)$ is a median metric space with rank at most $\hhscomp$.
		\end{enumerate}
	\end{lem}

	\begin{proof}
		\begin{enumerate}
			\item[]
			\item Let $x \in Q$, and let $\mathfrak U_x \coloneqq \{V \in \mathfrak U : x_V \notin \{p_V^\pm\}\}$. If $U,V \in \mathfrak U_x$ then $U \orth V$ or $U$ and $V$ are $\nest$-related, by consistency of $x$ and Lemma \ref{lem: collapsed projections}. Let $\Gamma$ be the complete graph with vertex set $\mathfrak U_x$. Given $U,V \in \mathfrak U_x$, colour the edge between them red if $U \orth V$ and blue if they are $\nest$-related. Blue cliques have at most $\hhscomp$ vertices by the definition of $\hhscomp$ and red cliques have at most $\hhscomp$ vertices by \cite[Lemma~2.1]{BehrstockHierarchically19}. Thus Ramsey's theorem provides the required $\hhscomp'$.
			\item If $p^+, p^- \in X$ then $\relevant_Q(\widehat x, \widehat y) \subseteq \mathfrak U \subseteq \relevant_{K/L - L}(p^+, p^-)$, which is finite by \cite[Lemma~13.5]{Casals-RuizReal24} and the choice of $K$.
			
			Otherwise, suppose, without loss of generality, that $p^+ \in \partial X$. For each $U \in \mathfrak U$, let $x_U \in \mathcal CU$ be such that $d_{T_U}(f_U(x_U), \widehat x_U) \leq L$, and define $y_U$ analogously. For each $V \in \support(p^+)$, let $W_V \coloneqq \{U \in \relevant_Q(\widehat x, \widehat y) : U \propnest V\}$. Then there is a uniformly bounded neighbourhood of $\hull_{\mathcal CV}(x_V, y_V)$ which contains $\rho^U_V$ for all $U \in W_V$. Therefore there exists a point $z_V^+ \in \mathcal CV$ such that $z_V^+$ is in the same connected component of $\mathcal CV - \mathcal N_{2E + \Morse_E(1,20E)}(\rho^U_V)$ as $\pi_V(p^+)$ for all $U \in W_V$. Using partial realisation (Definition \cite[Defn. 1.1.(8)]{BehrstockHierarchically19}), let $z^+ \in X$ be a point such that $d_V(z^+, z_V^+) \leq E$ for all $V \in \support(p^+)$. If $p^- \in X$ then $\relevant_Q(\widehat x, \widehat y) \subseteq \relevant_{K/L + L}(z^+, p^-)$. If not, then we can define a point $z^- \in X$ analogously to $z^+$ so that $\relevant_Q(\widehat x, \widehat y) \subseteq \relevant_{K/L + L}(z^+, z^-)$. In either case, \cite[Lemma~13.5]{Casals-RuizReal24} implies that $\relevant_Q(\widehat x, \widehat y)$ is finite.
			\item The fact that $d_Q$ is a metric follows easily from Item (2). For each $U \in \mathfrak U$, let $m_U$ denote the median map on $\widehat T_U$.  
			It is straightforward to check that $m = (m_U)_{U \in \mathfrak U}$ is consistent and that it is the metric median map for $(Q,d_Q)$. If $C = \{\sigma(\varepsilon_1, \dots, \varepsilon_k) : \varepsilon_i \in \{0,1\}\} \subseteq Q$ is a median embedded $k$-cube, then consistency implies the following: let $x,y,z\in C$ be such that $x$ differs from each of $y$ and $z$ in exactly one coordinate, and $y,z$ differ in two coordinates.  Then $U\orth V$ whenever $U\in\relevant_Q(x,y)$ and $V\in\relevant_Q(x,z)$.  Hence $\mathfrak U$ contains a set of $k$ pairwise orthogonal elements, so $k\leq \hhscomp$. (The consistency argument for the rank bound is virtually identical to the proof of \cite[Lem. 6.13]{Casals-RuizReal24}, for instance, so we omit the details.)
			\qedhere
		\end{enumerate}
	\end{proof}
	
	\begin{lem}
		\label{lem: Q is complete}
		$Q$ is complete.
	\end{lem}
	\begin{proof}
		Let $(x^n)_{n\in\mathbb N}$ be a Cauchy sequence in $Q$.  For each $U \in \mathfrak U$, the sequence $(x^n_U)_{n \in \mathbb N} \subseteq \widehat T_U$ is Cauchy. Let $x_U \coloneqq \lim_{n \rightarrow \infty} x^n_U$ and let $x \coloneqq (x_U)_{U \in \mathfrak U}$. It follows from the consistency of the $x^n$'s that $x$ is consistent.  To conclude that $\lim_{n\to\infty}x^n=x$, it then suffices to show that $|\cup_{n \in \mathbb N} \relevant_Q(x^1, x^n)| < \infty$. If not, another application of Ramsey's theorem will produce, for all $N \in \mathbb N$, a set $\mathfrak V$ of pairwise transverse domains such that $|\mathfrak V| \geq N$ and $V \subseteq \relevant_Q(x^1,x^n)$ for all sufficiently large $n$. By Lemma \ref{lem: Q is a metric space}.\eqref{item:domains-with-x-interior} we can remove $\leq \chi'$ elements from $\mathfrak V$ and assume that $x^1_U = p_U^-$ and $x^n_U = p_U^+$ for all $U \in \mathfrak V$ (or vice versa). 
        By Lemma \ref{lem: pseudo HHS}.\eqref{it: BGI in Y}, we can also assume that each $U \in \mathfrak V$ is $\nest_\mathfrak U$--minimal, which implies that $d_{\widehat T_U}(p^-_U,p^+_U) \geq K$ for all $U \in \mathfrak U$. But then $d_Q(x,x^n) \rightarrow \infty$. 
	\end{proof}
	
	\begin{lem}
		\label{lem: Q is connected}
		$Q$ is path-connected.
	\end{lem}
	\begin{proof}
		Let $x,y \in Q$ be distinct elements. We will argue by induction on the cardinality of $\relevant_Q(x,y)$ that there is a path in $Q$ connecting $x$ and $y$.
		
		Suppose $\relevant_Q(x,y) = \{U\}$ and let $\gamma_U:[0,1] \rightarrow \widehat T_U$ be the unique geodesic from $\widehat \pi_U(x)$ to $\widehat \pi_U(y)$. For each $t \in [0,1]$, let $z_t = (z_V)_{V \in \mathfrak U}$, where $z_V \coloneqq x_V = y_V$ for all $V \neq U$ and $z_U \coloneqq \gamma_U(t)$. It follows from the fact that both $x$ and $y$ are consistent that $z_t$ is consistent. Thus we can define $\gamma: [0,1] \rightarrow Q$ by $\gamma(t) \coloneqq z_t$ for all $t \in [0,1]$. Continuity of $\gamma$ follows immediately from continuity of $\gamma_U$ and the definition of $d_Q$, so $\gamma$ is a path.
		
		Let $k \in \mathbb N$, suppose that $|\relevant_Q(x,y)| = k+1$ and suppose $x'$ and $y'$ are connected by a path for all $x',y' \in Q$ such that $|\relevant_Q(x',y')| \leq k$. We consider two cases: \\
		
		\noindent \textbf{Case 1:} Suppose there exist $U, V \in \relevant_Q(x,y)$ such that $V \propnest U$. We can assume without loss of generality that $U$ is $\nest$-maximal in $\relevant_Q(x,y)$. Define $z^1 = (z_W^1)_ {W \in \mathfrak U}, z_W^2 = (z^2_W)_{W \in \mathfrak U}$ as follows. For each $W \in \mathfrak U$ such that $W \nest V$ or $W \orth V$, let $z^1_W \coloneqq x_W$ and let $z^2_W \coloneqq y_W$. For each $W \in \mathfrak U$ such that $W \transverse V$ or $V \propnest W$, let $z^1_W \coloneqq z^2_W \coloneqq \widehat \delta^V_W$.
		
		Let us check that $z^1$ and $z^2$ are consistent. Let $W_1, W_2 \in \mathfrak U$ be transverse. If $W_1, W_2$ are either nested in or orthogonal to $V$, then the required condition follows from the consistency of $x$ and $y$. If $W_1 \nest V$ or $W_1 \orth V$, and $W_2 \not\nest V$ and $W_2 \not\orth V$, then $\widehat \delta^{W_1}_{W_2} = \widehat \delta^V_{W_2} = z^1_{W_2} = z^2_{W_2}$. Suppose $W_1$ and $W_2$ are neither nested in nor orthogonal to $V$ and $\widehat \delta^V_{W_1} \neq \widehat \delta^{W_2}_{W_1}$. Then $d_{W_1}(\rho^V_{W_1}, \rho^{W_2}_{W_1}) > r/L - L > K_2 + 10E$. Since $\diam(\mathcal CV) > 100E$ and $\pi_V$ is $E$--coarsely surjective, there exists $p \in X$ such that $d_V(p, \pi_V(p^\pm)) > 10E$, which implies by the consistency and bounded geodesic image axioms that $d_{W_1}(p, \rho^V_{W_1}) < E + K_2$. Then $d_{W_1}(p, \rho^{W_2}_{W_1}) > E$ so $d_{W_2}(p, \rho^{W_1}_{W_2}), \dist_{W_2}(p, \rho^{V}_{W_2}) \leq E$. Thus $d_{W_2}(\rho^{W_1}_{W_2}, \rho^V_{W_2}) \leq 2E$, which implies that $\widehat \delta^{W_1}_{W_2} = \widehat \delta^V_{W_2}$.
		Next suppose that $W_1 \propnest W_2$. Again, if $W_2 \propnest V$ or $W_2 \orth V$, then the required condition follows from the consistency of $x$ and $y$. Suppose $V \propnest W_2$. If $\widehat \delta^{W_1}_{W_2} \neq \widehat \delta^V_{W_2}$ then Lemma~\ref{lem: collapsed projections}.(\ref{it: non-transverse nested domains in hatY},\ref{it: BGI in hatY}), gives $W_1 \transverse V$ and $\widehat \delta^V_{W_1} = \widehat \delta^{W_2}_{W_1} (\widehat \delta^V_{W_2})$. Lastly, suppose $V \transverse W_2$ and without loss of generality, suppose that $\widehat \delta^V_{W_2} = \widehat p_{W_2}^+$. Then either $\widehat \delta^{W_1}_{W_2} = \widehat \delta^V_{W_2}$ or $V \transverse W_1$ and $\widehat \delta^{W_2}_{W_1}(\widehat p_{W_2}^+) = \widehat p_{W_1}^+ = \widehat \delta^V_{W_1}$ by Lemma~\ref{lem: collapsed projections}.\eqref{it: BGI in hatY}. 
		
		Observe that $|\relevant_Q(x, z^1)|, |\relevant_Q(z^1, z^2)|, |\relevant_Q(z^2, y)| \leq k$ so there exist paths in $Q$ connecting $x$ to $z^1$, $z^1$ to $z^2$ and $z^2$ to $y$. Their concatenation is a path in $Q$ from $x$ to $y$. \\
		
		\noindent \textbf{Case 2:} If Case 1 fails then for all $U,V$ in $\relevant_Q(x,y)$, either $U \transverse V$ or $U \orth V$. Then, up to swapping $x$ and $y$, we can write $\relevant_Q(x,y) = \{U_0, \dots, U_k\}$, where the labels are chosen so that, if $0 < j$, then $U_0 \orth U_j$ or $\widehat \delta^{U_j}_{U_0} = \widehat p_{U_0}^+ = y_{U_0}$ and $\widehat \delta^{U_0}_{U_j} = \widehat p_{U_j}^- = x_{U_j}$; this is a consequence of consistency (compare \cite[Prop. 2.8]{BehrstockHierarchically19}). Define $z = (z_W)_{W \in \mathfrak U}$ where $z_W \coloneqq x_W$ for all $W \neq U_0$ and $z_{U_0} \coloneqq y_{U_0}$. It is straightforward to check that $z$ is consistent. Moreover $|\relevant_Q(x,z)| = 1$ and $|\relevant_Q(z,y)| = k$ so the induction hypothesis implies that there are paths in $Q$ from $x$ and $z$ and from $z$ to $y$, whose concatenation is the required path in $Q$.
	\end{proof}
	
	\begin{lem}\label{lem:Q-proper}
		The space $(Q,\dist_Q)$ isometrically embeds in $(\reals^k,\ell_1)$ for some $k$, and hence the space $(Q,\dist_Q)$ is proper.
	\end{lem}
	
	\begin{proof}
		Given $U,V\in\mathfrak U$, write $U<V$ if $U\transverse V$ and $\widehat\delta^U_V=\widehat p_V^-$ (and hence $\widehat\delta^V_U=\widehat p_U^+$, by Lemma \ref{lem: collapsed projections}).  Note that $<$ is a partial order on $\mathfrak U$, and $U,V\in\mathfrak U$ are incomparable in this order if and only if they are either orthogonal or $\propnest$--related.  Hence, by Ramsey's theorem, there is a bound (depending only on the complexity of $\mathfrak S$) on the cardinalities of $<$--antichains in $\mathfrak U$.  Therefore, by Dilworth's theorem \cite{Dilworth:poset}, there exists a finite partition $\mathfrak U=\bigsqcup_{i=1}^k\mathfrak U^i$ where each $\mathfrak U^i$ is a $<$--chain.  For each $U\in\mathfrak U^i$, choose $x(U)\in Q$ such that $x_U\not\in\{\hat p_U^\pm\}$ $U\in \mathfrak U_{x(U)}$, using Lemma \ref{lem: Q is connected}.
		
		By Lemma \ref{lem: Q is a metric space}.\eqref{item:rel-finite}, $\relevant_Q(x(U),x(V))$ is finite for all $U,V\in\mathfrak U^i$, so the totally ordered set $(\mathfrak U^i,<)$ has finite intervals, so there is an interval $I\subseteq \reals$ such that $\mathfrak U^i=\{U^n:n\in\integers\cap I\}$, where the $U^n$ are labelled so that $n<m$ if and only if $U^n<U^m$.  Note that, if $U^n$ has a successor (predecessor), then the interval $\widehat T_{U^n}$ has a maximum (minimum).  Hence there is a space $L^i$, isometric to $\reals,[0,\infty),$ or a bounded interval, equipped with isometric embeddings $\iota^n:\widehat T_{U^n}\to L^i$ such that $\bigcup_n\iota^n(\widehat T_{U^n})=L^i$, and $\iota^m,\iota^n$ have disjoint images unless $|m-n|\leq 1$, and $\iota^n(\widehat T_{U^n})\cap \iota^{n+1}(\widehat T_{U^{n+1}})=\iota^n(p_{U^n}^+)=\iota^{n+1}(p_{U^{n+1}}^-)$.  Define $\pi^i:Q\to L^i$ as follows.  Given $x\in Q$, let $n$ be maximal such that $x_{U^n}=p_{U^n}^+$.  If $n+1\not \in I$, then $L^i$ has a maximum, which we take as $\pi^i(x)$.  If $n+1$ exists, we let $\pi^i(x)=\iota^{n+1}(x_{U^{n+1}})$.  If there is no such $n$, then there is a unique minimal $n\in I\cap\integers$, and we take $\pi^i(x)=\iota^n(x_{U^n})$.  Hence $\prod_{i=1}^k\pi^i:Q\to \prod_{i=1}^kL_i$ is an isometric embedding (where the codomain has the $\ell_1$ metric), so $Q$ is locally compact and hence, since it is complete by Lemma \ref{lem: Q is complete}, proper.
	\end{proof}
	
	\subsection{Defining the map {$\widehat \Psi: X \rightarrow Q$}}\label{subsec:map-X-to-Q}
	For each $U \in \mathfrak U$, let $\psi_U \coloneqq f_U \circ \pi_U$. Then define $\Psi: X \rightarrow \mathcal Y$ by $\Psi(x) \coloneqq (\psi_U(x))_{U \in \mathfrak U}$.  By Definition \ref{defn:HHS-automorphism} and Proposition \ref{prop:quasi-linear-domains}, the map $\Psi$ is $G$-equivariant. 
	
	\begin{defn}
		Let $\alpha \geq 0$. An element $x=(x_U)_{U\in\mathfrak U} \in \mathcal Y$ is $\alpha$\textit{-consistent} if, for all $U, V \in \mathfrak U$,
		\begin{itemize}
			\item if $U \transverse V$ then $\min\{d_{T_U}(x_U, \delta^V_U), \dist_{T_V}(x_V, \delta^U_V)\} \leq \alpha$;
			\item if $U \propnest V$ then $\min\{d_{T_V}(x_V, \delta^U_V), \diam(\{x_U\} \cup \delta^V_U(x_V))\} \leq \alpha$.
		\end{itemize}
		Let $\mathcal Z_\alpha$ be the set of $\alpha$-consistent tuples in $\mathcal Y$. 
	\end{defn}

	Next is an elementary consequence of Lemma~\ref{lem: pseudo HHS} and the HHS consistency axioms:
	\begin{lem}
		\label{lem: consistent image 1}
		We have $\Psi(x) \in \mathcal Z_{\alpha}$ for all $x \in X$ and $\alpha \geq E'$. 
	\end{lem}
	
	Define $\Phi: \mathcal Y \rightarrow \widehat{\mathcal Y}$ by $(x_U)_U\mapsto (q_U(x_U))_U$, and let $\widehat \Psi \coloneqq \Phi \circ \Psi$. For each $U \in \mathfrak U$, also define $\widehat \psi_U: X \rightarrow \widehat T_U$ by $\widehat \psi_U \coloneqq q_U \circ \psi_U$.  By construction, $\widehat{\Psi}$ is $G$--equivariant.
	
	\begin{lem} \label{lem: consistent image 2}
		The image of $\widehat\Psi$ is consistent in the sense of Definition \ref{defn:Q}, i.e. $\widehat \Psi(X) \subseteq Q$.
	\end{lem}
	\begin{proof}
		Let $x \in X$ and $U,V \in \mathfrak U$. If $U \transverse V$ then $\min\{d_{T_U}(\psi_U(x), \delta^V_U), \dist_{T_V}(\psi_V(x), \delta^U_V)\} \leq E' < R$ so either $\widehat \psi_U(x) = \widehat \delta^V_U$ or $\widehat \psi_V(x) = \widehat \delta^U_V$. Suppose $U\propnest V$. If $\widehat \psi_V(x) \neq \widehat \delta^U_V$, then $d_{T_V}(\delta^U_V, \psi_V(x)) > R > 50E'$ so $d_V(\rho^U_V, x) > 50E$. Therefore $\diam_{\mathcal CU}(\pi_U(x) \cup \rho^V_U(\pi_V(x))) \leq E$ by the consistency axiom for HHSes (\cite[Defn. 1.1.(4)]{BehrstockHierarchically19}). It follows from the bounded geodesic image axiom for HHSes (\cite[Defn. 1.1.(7)]{BehrstockHierarchically19}) that $$\min\{\diam_{\mathcal CU}(\rho^V_U(\pi_V(x)) \cup \pi_U(p^-)), \diam_{\mathcal CU}(\rho^V_U(\pi_V(x)) \cup \pi_U(p^+))\} \leq 5E,$$ so $\psi_U(x) = p_U^+$ or $p_U^-$ and $\psi_U(x) \in \delta^V_U(\psi_V(x))$. Thus $\widehat \psi_U(x) \in \widehat \delta^V_U(\widehat \psi_V(x))$, as required.
	\end{proof}
	
	\subsection{Coarse surjectivity}\label{subsec:coarse-surjectivity-consistent}
	Next we check that $X$ coarsely surjects to $Q$, by proving:
	
	\begin{prop}
		\label{prop: coarsely surjective}
		There exists $M_1$, depending only on the HHS parameters, such that $\widehat \Psi$ is $M_1$-coarsely surjective.
	\end{prop}
	
	We first need the following lemma; recall that $\alpha_1=10R$.
	
	\begin{lem}
		\label{lem: coarsely surjective part 2}
		The map $\Phi|_{\mathcal Z_{\alpha_1}}: \mathcal Z_{\alpha_1} \rightarrow Q$ is surjective.
	\end{lem}
	\begin{proof}
		Let $\widehat x = (\widehat x_U)_{U \in \mathfrak U} \in Q$. Let us construct an $\alpha_1$--consistent tuple $x = (x_U)_{U \in \mathfrak U} \in \mathcal Y$ such that $\Phi(x) = \widehat x$. 
		Let $U \in \mathfrak U$. 
		If $\hat x_U\in \widehat T^e_U$, then $q_U^{-1}(\hat x_U)$ contains a single point, which we denote $x_U$.  Otherwise, let $C \subseteq T_U$ be the cluster such that $q_U(C) = \widehat x_U$. Define
		\[
		\mathfrak U_C \coloneqq \{V \in \mathfrak U : V \propnest U, \; \delta^V_U \subseteq C \text{ and } \widehat x_V\in \widehat T_V^e\}.
		\]
		Note that $\mathfrak U_C$ contains at least all $\nest_{\mathfrak U}$-minimal $V \nest U$ such that $\delta^V_U \subseteq C$, by Remark~\ref{rem: collapsed trees}. For each $V\in\mathfrak U_C$, we have a unique point $x_V\in q_V^{-1}(\widehat x_v)$.  Given $V \in \mathfrak U_C$, define $C_V \subseteq T_U$ as follows. If $x_V \notin \{p_V^-, p_V^+\}$ then define $C_V \coloneqq \mathcal N_{R}(\delta^V_U)$. If $x_V = p_V^-$ then let $C_V \coloneqq C \cap \hull_{T_U}(\{p_U^-\} \cup \mathcal N_R(\delta^V_U))$ and if $x_V = p_V^+$ then let $C_V \coloneqq C \cap \hull_{T_U}(\mathcal N_R(\delta^V_U) \cup \{p_U^+\})$.
		
		\setcounter{claim}{0}
		\begin{claim}\label{claim:pairwise-intersecting-intervals}
			If $V,W \in \mathfrak U_C$ then $C_V \cap C_W \neq \emptyset$. 
		\end{claim}
		\begin{proof}
			\renewcommand{\qedsymbol}{$\blacksquare$}
			Suppose $V \neq W$. If $V$ and $W$ are not transverse then Lemma~\ref{lem: pseudo HHS}.\eqref{it: non-transverse nested domains in Y} implies that $d_{T_U}(\delta^V_U, \delta^W_U) < R$, so $C_V \cap C_W \neq \emptyset$ by construction. If $V \transverse W$, then consistency of $\widehat x$ implies that either $x_V \in \delta^W_V \in \{p_V^\pm\}$ or $x_W \in \delta^V_W \in \{p_W^\pm\}$. Suppose $x_V = p_V^+$; the other cases are similar. If $d_{T_U}(\rho^V_U, \rho^W_V) > R > 10E'$ then Lemma~\ref{lem: pseudo HHS}.(\ref{it: BGI in Y},\ref{it: BGI special case in Y}) imply that $\delta^W_U$ is in the same connected component of $T_U - \mathcal N_R(\delta^V_U)$ as $p_U^+$, so $\delta^W_U \subseteq C_V \cap C_W$.
		\end{proof}
		
		Identify $T_U$ isometrically with a subset of $\mathbb R$ such that, if $p_U^+,p_U^- \in T_U$ then $p_U^- < p_U^+$ and, if $p_U^+ \in \partial T_U$ (resp. $p_U^- \in \partial T_U$) then, extending the identification to the boundary, $p_U^+ = + \infty$ (resp. $p_U^- = - \infty$).
		
		We define $x_U\in T_U$ as follows, distinguishing two cases.  First, if $U\not\in\support(p^-)\cup\support(p^+)$, then $T_U$ is a finite interval. Claim \ref{claim:pairwise-intersecting-intervals} implies that $\cap_{V\in\mathfrak U_C}C_V\neq \emptyset$. In this case, choose $x_U \in \cap_{V \in \mathfrak U_C} C_V$ such that:
		\begin{itemize}
			\item if $x_V = p_V^+$ for all $V \in \mathfrak U_C$ then $x_U \coloneqq \sup C$;
			\item if $x_V = p_V^-$ for all $V \in \mathfrak U_C$ then $x_U \coloneqq \inf C$;
			\item otherwise, there exists $W\in\mathfrak U_C$ such that $x_W\not \in \{p_W^\pm\}$, so $C_W=\neb_R(\delta^W_U)$, and we choose $x_U\in\cap_{V\in\mathfrak W}C_V\subseteq \neb_R(\delta^W_U)$ arbitrarily.
		\end{itemize}
		Since $T_U$ is a finite interval,  $x_U \in T_U$ is well defined.
		
		The second case is where $U\in\support(p^+) \cup \support(p^-)$; without loss of generality, suppose $U\in\support(p^+)$. If there exists $W\in\mathfrak U_C$ such that $x_W\not\in \{p^\pm_W\}$, then, exactly as in the first case, $C_W=\neb_R(\delta^W_U)$, and we choose $x_U\in\cap_{V\in\mathfrak W}C_V\subseteq \neb_R(\delta^W_U)$ arbitrarily.  Hence we can assume that $x_V=p_V^+$ for all $V\in\mathfrak U_C$ (the case where $x_V = p_V^-$ for all $V \in \mathfrak U_C$ is similar). 
		
		Suppose $\sup C=p_U^+$, so, by the definition of a cluster, $\sup(\cup_{V\in\mathfrak U_C}\delta^V_U)=p_U^+$.  (So, $p_U^+\in\boundary T_U$ but $\max\widehat T_U=\hat x_U$.)  Hence there is a sequence $\{V_n\}_{n\in\naturals}$ of $\nest_{\mathfrak U}$--minimal elements in $\mathfrak U_C$ such that $V_n\transverse V_m$ and $\delta^{V_m}_{V_n}=p^+_{V_n}$ and $\delta^{V_n}_{V_m}=p^-_{V_m}$ whenever $n<m$, and $\lim_{n \to \infty} (\sup \delta^{V_n}_U) = p_U^+$.  By our assumption, $x_{V_n}=p_{V_n}^+$ for all $n$.  Now, for any $z\in X$, we have $\widehat \psi_{V_n}(z)= \widehat p_{V_n}^-$ for all but finitely many $n$.  By Lemma \ref{lem: consistent image 2}, $\widehat\Psi(z)\in Q$, but we have just shown that $\relevant_Q(\widehat\Psi(z),\hat x)$ is infinite, contradicting Lemma \ref{lem: Q is a metric space}. Thus $\sup C < \infty$ and we define $x_U \coloneqq \sup C$.
		
		Let $x=(x_U)_{U\in\mathfrak U}\in \mathcal Y$.  By construction, $\Phi(x) = \widehat x$, so it remains to show that $x$ is $10R$-consistent.\\
		
		Let $U_1, U_2 \in \mathfrak U$ be such that $U_1 \propnest U_2$. Suppose first that $\widehat x_{U_2} \neq \widehat \delta^{U_1}_{U_2}$, so $\widehat x_{U_1} \in \widehat \delta^{U_2}_{U_1}(\widehat x_{U_2}) = \{\widehat p_{U_1}^-\}$ or $\{\widehat p_{U_1}^+\}$. Suppose the latter holds; the other case is similar. We then have $d_{U_2}(x_{U_2}, \delta^{U_1}_{U_2}) > R > E'$ and $x_{U_2}$ is in the same connected component of $T_{U_2} - \mathcal N_{E'}(\delta^{U_1}_{U_2})$ as $p_{U_2}^+$, so Lemma~\ref{lem: pseudo HHS}.\eqref{it: BGI in Y} implies that $\delta^{U_2}_{U_1}(x_{U_2}) = p_{U_1}^+$. If $x_{U_1} = \widehat x_{U_1}$, then this implies immediately that $\delta^{U_2}_{U_1}(x_{U_2}) = x_{U_1}$. Otherwise, let $C \subseteq T_{U_1}$ be the cluster containing $x_{U_1}$. Then, for each domain $V \in \mathfrak U_C$, we have $\diam_{T_{U_2}}(\delta^V_{U_2} \cup \delta^{U_1}_{U_2}) < R$ by Lemma~\ref{lem: pseudo HHS}.\eqref{it: non-transverse nested domains in Y} so $\widehat \delta^V_{U_2} = \widehat \delta^{U_1}_{U_2} \neq \widehat x_{U_2}$. 
		By consistency of $\widehat x$, we have $\widehat x_V \in \widehat \delta^{U_2}_V(\widehat x_{U_2}) = \{p_V^+\}$ which, since $V \in \mathfrak U_C$, implies that $x_V = p_V^+$. Then $x_{U_1} = \rho_{U_1}^+ = \delta^{U_2}_{U_1}(x_{U_2})$.
		
		Next suppose that $\widehat x_{U_2} = \widehat \delta^{U_1}_{U_2}$ and let $C \subseteq T_{U_2}$ be the cluster containing $\delta^{U_1}_{U_2}$ and $x_{U_2}$. Suppose $d_{T_{U_2}}(x_{U_2}, \delta^{U_1}_{U_2}) > 2R$. Then Lemma \ref{lem: pseudo HHS}.\eqref{it: BGI in Y} implies that, either $x_V = p_V^+$ for all $V \in \mathfrak U_C$ such that $V \nest U_1$, or $x_V = p_V^-$ for all $V \in \mathfrak U_C$ such that $V \nest U_1$. Suppose the former case holds; the latter case is similar. If $\widehat x_{U_1} \notin \widehat T_{U_1}^e$, then $x_{U_1} = p_{U_1}^+$ by construction and, if $\widehat x_{U_1} \in \widehat T_{U_1}^e$, then $U_1 \in \mathfrak U_C$ so $x_{U_1} = p_{U_1}^+$. Moreover $x_{U_2} \in \cap_{V \in \mathfrak U_C} C_V - \mathcal N_{2R}(\delta^{U_1}_{U_2})$, which implies that $\delta^{U_2}_{U_1}(x_{U_2}) = p_{U_1}^+$.
		\\
		
		Now let $U_1, U_2 \in \mathfrak U$ be such that $U_1 \transverse U_2$. Up to relabelling, we can assume that $\delta^{U_1}_{U_2} = \{p_{U_2}^-\}$ and $\delta^{U_2}_{U_1} = \{p_{U_1}^+\}$, so we have $\widehat x_{U_1} = \widehat p_{U_1}^+$ or $\widehat x_{U_2} = \widehat p_{U_2}^-$. We first suppose that $\widehat x_{U_1} = \widehat p_{U_1}^+$ and $\widehat x_{U_2} \neq \widehat p_{U_2}^-$. If $\widehat x_{U_1} \in \widehat T_{U_1}^e$, then $x_{U_1} = p_{U_1}^+ \in \delta^{U_2}_{U_1}$. Otherwise, $p_{U_1} \in C$ for some cluster $C \subseteq T_{U_1}$ such that $\mathfrak U_C \neq \emptyset$. 
		Let $V \in \mathfrak U_C$. If $V \transverse U_2$ then $x_V = \widehat x_V = \widehat p_V^+ = p_V^+$. Otherwise $V \orth U_2$ and Lemma \ref{lem: pseudo HHS}.\eqref{it: consistency in Y: transverse + orthogonal} implies that $\diam(\{p_{U_1}^+\} \cup \delta^V_{U_1}) = \diam(\delta^{U_2}_{U_1} \cup \delta^V_{U_1}) \leq E' \leq R$. It follows that $x_{U_1} \in \mathcal N_R(p_{U_1}^+)$. If $\widehat x_{U_1} \neq \widehat p_{U_1}^+$ and $\widehat x_{U_2} = \widehat p_{U_2}^-$ then a symmetric argument shows that $d_{T_{U_2}}(x_{U_2}, p_{U_2}^-) \leq R$.
		
		Suppose that  $\widehat x_{U_1} = \widehat p_{U_1}^+$ and $\widehat x_{U_2} = \widehat p_{U_2}^-$. 
		Suppose that $x_{U_1} \neq p_{U_1}^+$ and $x_{U_2} \neq p_{U_2}^-$. Then there exist clusters $C_1 \subseteq T_{U_1}, C_2 \subseteq T_{U_2}$, containing $p_{U_1}^+, p_{U_2}^-$ respectively, such that $\mathfrak U_{C_1}, \mathfrak U_{C_2} \neq \emptyset$. If $\widehat x_V = \widehat p_V^+$ for all $V \in \mathfrak U_{C_1}$ then $x_{U_1} = p_{U_1}^+$ and, if $\widehat x_V = \widehat p_V^-$ for all $V \in \mathfrak U_{C_2}$ then $x_{U_2} = p_{U_2}^-$. So there are $V_1 \in \mathfrak U_{C_1}$ and $V_2 \in \mathfrak U_{C_2}$ with $\widehat x_{V_1} \neq \widehat p_{V_1}^+$ and $\widehat x_{V_2} \neq \widehat p_{V_2}^-$. 
		
		Suppose $V_1 \orth V_2$. If $V_2 \not\orth U_1$ then Lemma~\ref{lem: pseudo HHS}.\eqref{it: consistency in Y, transverse + nested} implies that $\diam(\delta^{V_2}_{U_1} \cup \{p_{U_1}^+\}) = \diam(\delta^{V_2}_{U_1} \cup \delta^{U_2}_{U_1}) \leq E'$ and $\diam(\delta^{V_1}_{U_1} \cup \delta^{V_2}_{U_1}) \leq E'$ so $\diam(\delta^{V_1}_{U_1} \cup \{p_{U_1}^+\}) < R$. 
        If $V_2 \orth U_1$ then Lemma \ref{lem: pseudo HHS}.\eqref{it: consistency in Y: transverse + orthogonal} implies that $\diam(\delta^{V_2}_{U_2} \cup \delta^{U_1}_{U_2}) \leq E'$.
        
        By construction, $d_{T_{U_2}}(x_{U_2}, \delta^{V_2}_{U_2}) \leq R$, so $$d_{T_{U_2}}(x_{U_2}, \delta^{U_1}_{U_2}) \leq E' + R < 10R.$$
		If $V_1 \not\orth V_2$ then, since $\widehat x$ is consistent, $V_1$ and $V_2$ are $\nest$-related, so $V_i \nest U_1$ and $V_i \nest U_2$ for some $i \in \{1,2\}$. Hence $d_{U_1}(\rho^{V_i}_{U_1}, \rho^{U_2}_{U_1}) \leq E$ so $\delta^{V_i}_{U_1} = p_{U_1}^+ = \delta^{U_2}_{U_1}$. Then $x_{U_1} \in \mathcal N_R(\delta^{V_1}_{U_1}) \subseteq \mathcal N_{2R}(\delta^{U_2}_{U_1})$ by construction and Lemma~\ref{lem: pseudo HHS}.
	\end{proof}

    \medskip
    
	\begin{proof}[Proof of Proposition~\ref{prop: coarsely surjective}]
		For later use, fix $z \in X$ and, for all $V \in \mathfrak S - \mathfrak U$, let $z_V \coloneqq f_V(\pi_V(z))$. 
		
		Let $\widehat x \in Q$; we need to find $y\in X$ with $d_Q(x,\widehat\Psi(y))\leq M_1$, where $M_1$ is uniform.  
		
		Using Lemma~\ref{lem: coarsely surjective part 2}, let $x \in \mathcal Z_{\alpha_1}$ be such that $\Phi(x) = \widehat x$. For each $U \in \mathfrak U$, there exists $y_U \in \mathcal CU$ such that $d_{T_U}(f_U(y_U), x_U) \leq L$. If $U \in \mathfrak S - \mathfrak U$ then let $y_U \coloneqq z_U$. The fact that $x$ is $\alpha_1$-consistent and $\diam(\mathcal CU) \leq LK + L$ for all $U \in \mathfrak{S - U}$ implies that $y = (y_U)_{U \in \mathfrak S}$ is $(L\alpha_1 + LK + 2L)$-consistent. Hence, by Theorem~\ref{thm:hhs_realisation} and the definition of $\kappa$, there exists $y \in X$ such that $d_U(y, y_U) \leq \kappa$ for all $U \in \mathfrak U$. Thus $d_{\widehat T_U}(\widehat \psi_U(y), \widehat x_U) \leq \dist_{T_U}(\psi_U(y), x_U) \leq L\kappa + L$. By Lemma~\ref{lem: Q is a metric space}.(1),  $$|\{V \in \relevant_Q(\widehat x, \widehat \Psi(y)) : \{\widehat x_V, \widehat \psi_V(y)\} \neq \{\widehat p_V^\pm\}\}| \leq \hhscomp'.$$
		Let $\mathcal V$ be the set of elements in $\relevant_Q(\widehat x, \widehat \Psi(y))$ such that $\{\widehat x_V, \widehat \psi_V(y)\} = \{\widehat p_V^\pm\}$.  Then
		\[d_Q(\hat x,\widehat{\Psi}(y))\leq (\hhscomp'+|\mathcal V|)(L\kappa+L),\]
		so to conclude, it is sufficient to prove:

        \setcounter{claim}{0}
		\begin{claim}\label{claim:bounded-V-cardinality}
			Let $\mathcal V' \subseteq \mathcal V$ be a set of pairwise transverse domains with maximal cardinality.  Then $|\mathcal V'|\leq P$, where $P$ depends only on the HHS parameters.  Hence there exists $P_1\in\naturals$, depending only on the HHS parameters, such that $|\mathcal V|\leq P_1$.
		\end{claim}
		
		\begin{proofofclaim}{\ref{claim:bounded-V-cardinality}}
			Since $r > L\kappa + L$, each $V \in \mathcal V$ is $\nest_{\mathfrak U}$-minimal, so $T_V=T_V^e$ and $q_V:T_V\to \widehat T_V$ is the identity. 
			
			Since both $\widehat x$ and $\widehat \Psi(y)$ are consistent, one of the following holds:
			\begin{enumerate}
				\item for all $V \in \mathcal V'$, we have $\widehat x_V = \widehat p_V^-$ and $\widehat \psi_V(y) = \widehat p_V^+$;
				\item for all $V \in \mathcal V'$ we have $\widehat x_V = \widehat p_V^+$ and $\widehat \psi_V(y) = \widehat p_V^-$.
			\end{enumerate}
			We assume that (1) holds; the other case is similar. \\
			
			\noindent\textbf{Test points and passing up.}  Since $\mathcal V'$ consists of transverse elements and $d_V(\pi_V(p^+),\pi_V(p^-))>K/L-L\geq 50E$ for all $V\in\mathcal V'$, we have a total order $\prec$ on $\mathcal V'$ given by $V\prec W$ if $\dist_V(\rho^W_V,\pi_V(p^+))\leq E$ (see \cite[Sec. 2]{BehrstockHierarchically19}).  Since $\hat x,\widehat{\Psi}(y)\in Q$, Lemma \ref{lem: Q is a metric space} implies $|\mathcal V'|<\infty$, and we may assume $\mathcal V'\neq\emptyset$.  Hence $\mathcal V'$  contains a unique $\prec$--minimal element $V_1$ and a unique $\prec$--maximal element $V_2$.  Let $\mathcal V''=\mathcal V'-\{V_1,V_2\}$, and, for $i\in \{1,2\}$, let $z_i\in P_{V_i}$.  (Recall that $P_{V_i}$ is the standard product region from Section \ref{subsec:product-regions}.)    
			
			Then for all $V\in\mathcal V''$, we have $\dist_V(\rho^{V_1}_V,z_1)\leq E$, and $\dist_V(\rho^{V_2}_V,z_2)\leq E$.  Now, since $V_1\prec V\prec V_2$, we also have $\dist_V(\rho^{V_1}_V,\pi_V(p^-))\leq E$ and $\dist_V(\rho^{V_2}_V,\pi_V(p^+))\leq E$, so $d_V(z_1,\pi_V(p^-))\leq 2E$ and $d_V(z_2,\pi_V(p^+))\leq 2E$.  Hence $d_V(z_1,z_2)\geq K/L-L-4E\geq 50E$.
			
			Let $P$ be the constant from \cite[Proposition~4.3]{DurhamCubulating23}, applied to the pair $(K/L - L-4E, L(L(\kappa + 2) + 2\alpha_1)+KL+L)$. This proposition then implies that, if $|\mathcal V''| \geq P$, then  there exists $W \in \mathfrak S$ and $\mathcal V''' \subseteq \mathcal V''$ such that 
			\begin{itemize}
				\item $d_W(z_1,z_2)\geq KL+L$, so $d_{T_W}(\psi_W(z_1),\psi_W(z_2))\geq K$ and hence $W\in\mathfrak U$.
				\item $V \propnest W$ for all $V \in \mathcal V'''$.
				\item $\mathcal V'''$ satisfies
				\[
				\diam_W \Big( \bigcup_{V \in \mathcal V'''} \rho^V_W \Big) >  L(L(\kappa + 2) + 2\alpha_1).
				\]
			\end{itemize}
			
			It follows that $\diam_{T_W}(\cup_{V \in \mathcal V'} \delta^V_W) > L(\kappa + 1) + 2\alpha_1$. Let $\mathcal V'''=\{V_1, \dots, V_k\}$ be numbered so that $i \leq j$ if and only if $\delta^{V_i}_{V_j} = p_{V_j}^-$. Since $L(L(\kappa + 2) + 2\alpha_1) > L(E' + L)$, we have $k \geq 2$ by Lemma \ref{lem: pseudo HHS}.\eqref{it: bounded projections in Y}. Then, since $x$ and $\Psi(y)$ are $\alpha_1$-consistent, we have $x_W \in \hull_{T_W}(p_W^- \cup \mathcal N_{\alpha_1}(\delta^{V_1}_W))$ and $\psi_W(y) \in \hull_{T_W}(\mathcal N_{\alpha_1}(\delta^{V_k}_W) \cup p_W^+)$. But then $d_{T_W}(x_W, \psi_W(y)) > L \kappa + L$, which is a contradiction. \\
			
			\noindent\textbf{Bounding $|\mathcal V|$.}  We have shown that $|\mathcal V'|\leq P+2$, which depends only on the HHS parameters, so, by Ramsey's theorem, we can conclude by taking $P_1=\mathrm{Ram}(\hhscomp,P+2)$, where $\mathrm{Ram}(\cdot,\cdot)$ denotes the Ramsey number, because no two elements of $\mathcal V$ are $\propnest$--related, pairwise orthogonal sets have cardinality at most $\hhscomp$, and pairwise transverse subsets have cardinality at most $P+2$.
		\end{proofofclaim}
				
		Claim \ref{claim:bounded-V-cardinality} and the inequality preceding it imply  
		\[
		d_Q(\widehat x, \widehat \Psi(y)) \leq (\hhscomp' + P_1) L (\kappa +1),
		\]
		which completes the proof.
	\end{proof}

	\subsection{Quasi-isometry}
	
	Given $B,C \geq 0$, let $\ignore{B}{C} \coloneqq 0$ if $B < C$ and $B$ otherwise. Given $x \in X$, denote $\widehat x \coloneqq \widehat \Psi(x)$, $x_U \coloneqq \psi_U(x) \in T_U$ and $\widehat x_U \coloneqq \widehat \psi_U(x) \in \widehat T_U$.
	
	\begin{prop}
		There exists a constant $L'$ depending only on the HHS parameters such that $\widehat \Psi$ is an $(L',L')$-quasi-isometry.
	\end{prop}
	\begin{proof}
		By Proposition~\ref{prop: coarsely surjective}, $\widehat \Psi$ is coarsely surjective, with constant depending only on the HHS parameters, so it remains to show that $\widehat \Psi$ is a quasi-isometric embedding. 
		
		Let $x,y \in X$. 
		For each $U \in \mathfrak U$, the map $f_U$ is an $(L,L)$-quasi-isometry and $q_U$ is 1-lipschitz, so 
		\[
		\frac{1}{L} \dist_{\widehat T_U}(\widehat x_U, \widehat y_U) - L \leq \frac{1}{L} \dist_{T_U}(x_U, y_U) - L \leq \dist_U(x,y) \leq L \dist_{T_U}(x_U, y_U) + L.
		\] 
		By Theorem \ref{thm:distance-formula}, there exist $S \geq L(K + 50E) + L$ and $\tau \geq 1$ depending only on $L$ and the HHS parameters such that
		\begin{equation}
			\label{eq1}
			\frac{1}{\tau} \sum_{U \in \mathfrak U} \ignore{d_{T_U}(x_U, y_U)}{S} - \tau \leq \dist_X(x,y) \leq \tau \sum_{U \in \mathfrak U} \ignore{d_{T_U}(x_U, y_U)}{S} + \tau.
		\end{equation}
		
		To bound the left-hand side of \eqref{eq1} from below in terms of $d_Q(\widehat x, \widehat y)$, we need to control the contribution to $d_Q(\widehat x,\widehat y)$ from the set $\mathcal V \coloneqq \{U \in \mathfrak U : \dist_{T_U}(x_U, y_U) < S$ and $\widehat x_U \neq \widehat y_U\}$.
		By Lemma~\ref{lem: Q is a metric space}.(2), there are at most $2\hhscomp'$ domains $U \in \mathcal V$ such that $\{\widehat x_U, \widehat y_U\} \not\subseteq \{\widehat p_U^-, \widehat p_U^+\}$, so these contribute at most $2\hhscomp'S$ to $d_Q(x,y)$. For all remaining $U \in \mathcal V$, we have $d_{U}(x,y) \geq K/L - L > 50E$ (using that $r > K$ for the case where $U$ is not $\nest_{\mathfrak U}$-minimal). It follows from \cite[Proposition~4.14]{DurhamCubulating23}\footnote{The proposition in \cite{DurhamCubulating23} is stated for rays, but the proof works for points in $X$.  Indeed, it uses Lemmas 4.4 and 4.10 in \emph{loc. cit.}, which are written for interior points as well as rays, and the rest of proof of Proposition 4.14 does not depend on any of the sets involved being unbounded.} that there exist constants $P_1, P_2 \in \mathbb N$, depending only on $LS + L$ and the HHS parameters, such that:
		\begin{itemize}
			\item for all but $P_1$ elements $V \in \mathcal V$, there exists $W \in \mathfrak U$ such that $d_W(x,y) \geq LS + L$ and $V \nest W$;
			\item for any $W \in \relevant_{LS + L}(x,y)$, if $\mathcal W \coloneqq \{V \in \mathcal V : V \propnest W\}$, then $|\mathcal W| \leq P_2 \dist_W(x,y) + P_2$.
		\end{itemize}
		Therefore
		\[
		\dist_Q(\widehat x, \widehat y) \leq 2\hhscomp'S + P_1S + (4SP_2L+1) \sum_{U \in \mathfrak U} \ignore{d_{T_U}(x_U,y_U)}{S}.
		\]
		Combined with \eqref{eq1}, we have shown that $d_Q(\widehat x, \widehat y) \leq \tau_1 \dist_X(x,y) + \tau_1$, where $\tau_1 \geq 1$ depends only on the HHS parameters. \\
		
		It remains to bound the right-hand side of \eqref{eq1} from above in terms of $d_Q(\widehat x, \widehat y)$.
		Let $U \in \mathfrak U$ and suppose that $d_{T_U}(x_U, y_U) > \dist_{\widehat T_U}(\widehat x_U, \widehat y_U)$. Let $\{C_1, \dots, C_k\}$ be the clusters in $T_U$ whose intersections with the geodesic $[x_U, y_U]_{T_U} \subseteq T_U$ from $x_U$ to $y_U$ have strictly positive diameter. For each $i$, let $C_i' \coloneqq (C_i \cap [x_U, y_U]_{T_U}) - (\mathcal N_{E'}(x_U) \cup \mathcal N_{E'}(y_U))$ and let $\mathcal V_i \coloneqq \{V \in \mathfrak U : V \propnest U$ and $\delta^V_U \subseteq C_i\}$. The set $\{\delta^V_U : V \in \mathcal V_i\}$ is $(R + r)$--dense in $C_i$ and it follows that there is a subset $\mathcal V_i' \subseteq \mathcal V$ such that: 
		\begin{itemize}
			\item each $V \in \mathcal V_i'$ is $\nest_{\mathfrak U}$-minimal;
			\item $\delta^V_U \subseteq C_i'$ for each $V \in \mathcal V_i'$;
			\item $\{\delta^V_U : V \in \mathcal V_i'\}$ is $2(R + r)$-dense in $C_i'$.
		\end{itemize}
		We then have $|\mathcal V'_i| \geq \diam(C_i') / 4(R + r) - 1$, using that $\diam(\delta^V_U) \leq E' < R$ for each $V \propnest U$. For each $V \in \mathcal V_i'$, the fact that $\delta^V_U \subseteq C_i'$ implies that $x_V = p_V^-$ and $y_V = p_V^+$, or vice versa, so, by Remark~\ref{rem: collapsed trees}, $d_{\widehat T_V}(\widehat x_V, \widehat y_V) = \dist_{T_V}(x_V, y_V) \geq K$. Thus we have:
		\begin{align*}
			d_{T_U}(x_U, y_U) & = d_{\widehat T_U}(\widehat x_U, \widehat y_U) + \sum_{i=1}^k \diam(C_i \cap [x_U, y_U]_{T_U}) \\
			& \leq d_{\widehat T_U}(\widehat x_U, \widehat y_U) + 4(R + r)\sum_{i=1}^k (|\mathcal V_i'| + 1) + 2kE' \\
			& \leq (1+2E')d_{\widehat T_U}(\widehat x_U, \widehat y_U) + 4(R + r)\sum_{i=1}^k K |\mathcal V_i'|  + 4(R + r)\\
			& \leq (1+2E')d_{\widehat T_U}(\widehat x_U, \widehat y_U) + 4(R + r) \sum_{i=1}^k \sum_{V \in \mathcal V_i} d_{\widehat T_V}(\widehat x_V, \widehat y_V) + 4(R + r).
		\end{align*}
		Now, \cite[Lemma~2.15]{DurhamMinskySisto:stable} provides a constant $N$, depending only on the HHS parameters, such that each $V \in \mathfrak U$ is properly nested in at most $N$ elements of $\mathfrak U$. Therefore:
		\[
		\sum_{U \in \mathfrak U} d_{T_U}(x_U, y_U) \leq (1 + 2E' + 4N(R + r)) d_Q(\widehat x, \widehat y) + 4(R + r).
		\]
		This completes the proof.
	\end{proof}
	
	\subsection{Quasi-median}
	
	\begin{prop}
		The map $\widehat \Psi$ is $\hhscomp'(L(C_0 + C_1) + 2L)$-quasimedian.
	\end{prop}
	\begin{proof}
		Let $x,y,z \in X$. We keep the same convention as in the previous section. Let $m \in \mu(x,y,z)$, let $\ddot m = (\ddot m_U)_{U \in \mathfrak U} \in Q$ be the median of $\widehat x, \widehat y, \widehat z$, and let $a,b \in \{x,y,z\}$ be distinct points. For any $U \in \mathfrak U$ and any geodesic $\gamma$ in $\mathcal CU$ joining $\pi_U(a)$ and $\pi_U(b)$, we have $d_U(\pi_U(m), \gamma) \leq C_0 + C_1$. It follows that $d_{T_U}(m_U, [a_U, b_U]_{T_U}) \leq L(C_0 + C_1) + 2L$, so $d_{\widehat T_U}(\widehat m_U, [\widehat a_U, \widehat b_U]_{\widehat T_U}) \leq L(C_0 + C_1) + 2L$. By definition, $\ddot m_U$ lies on $[\widehat a_U, \widehat b_U]_{\widehat T_U}$ for all distinct $a,b \in \{x,y,z\}$, so this implies that $d_{\widehat T_U}(\widehat m_U, \ddot m_U) \leq L(C_0 + C_1) + 2L$ for all $U \in \mathfrak U$. Then, since $r,K > L(C_0 + C_1) + 2L$, we have $\{\widehat m_U, \ddot m_U\} \neq \{\widehat p_U^\pm\}$ for all $U \in \relevant_Q(\widehat m, \ddot m)$. By Lemma \ref{lem: Q is a metric space}.(1) we then have $|\relevant_Q(\widehat m, \ddot m)| \leq \hhscomp'$, so $d_Q(\widehat m, \ddot m) \leq \hhscomp'(L(C_0 + C_1) + 2L)$.
	\end{proof}
	
	\subsection{Cocompact quasilines}\label{subsec:cocompact}
	We conclude the section with:
	
	\begin{proof}[Proof of Proposition \ref{prop:short-mountains}]
		Fix a point $z_0 \in X$.
		Let $\support(p^\pm)=\{U_1,\ldots,U_r\}$, where $r\leq \hhscomp$ since $\support(p^\pm)$ is a set of pairwise orthogonal domains.  Since $X=H_{M,z_0}(p^-,p^+)$, the points $\pi_{U_i}(p^-),\pi_{U_i}(p^+)$ are distinct points of $\boundary CU_i$ for each $i$, and $\pi_{U_i}(X)$ is $M$--close to some $(1,20E)$--quasigeodesic joining $\pi_{U_i}(p^{\pm})$.  Hence $\mathcal CU_i$ is a uniform quasiline, for $1\leq i\leq r$.  Moreover, Proposition \ref{prop:quasi-linear-domains} provides a finite-index subgroup $G'\leq G$ fixing $p^-,p^+$ and (passing if needed to a further finite-index subgroup) $G'\cdot U_i=U_i$ and $G'$ therefore acts on $\mathcal CU_i$ by isometries, fixing the two-point set $\boundary\mathcal CU_i$ pointwise.  Since $\pi_{U_i}(G\cdot z_0)$ is unbounded by hypothesis, so is $\pi_{U_i}(G'\cdot z_0) = G'\cdot\pi_{U_i}(z_0)$; therefore $G'$ contains elements acting on $\mathcal CU_i$ loxodromically, so the action of $G'$ on $\mathcal CU_i$ is cobounded.  Proposition \ref{prop:quasi-linear-domains} provides a $G'$--action on a line $T_{U_i}$ and an equivariant quasi-isometry $f_i:\mathcal CU_i\to T_{U_i}$, for each $i$.  
		
		For each $i$, let $\mathfrak U_i=\{V\in\mathfrak U:V\nest U_i\}$.  Then $V\orth V'$ whenever $V\in\mathfrak U_i$ and $V'\in\mathfrak U_j$ with $i\neq j$.  Let $\mathfrak S_i=\mathfrak S-\bigcup_{j\neq i}\mathfrak U_j$, so that we have a canonical cone-off $c_i:(X,\mathfrak S)\to (\newcone X_i,\mathfrak S_i)$.  Applying Theorem \ref{thm: median hulls} to each $G'$--action on $(\newcone X_i,\mathfrak S_i)$, we obtain a median space $Q_i$ and a $G'$--equivariant quasi-isometry $\Psi_i:\newcone X_i\to Q_i$ such that the diagram:
		\begin{center}
			$
			\begin{diagram}
				\node{X}\arrow{e,t}{c_i}\arrow{s,l}{\Psi}\node{\newcone X_i}\arrow{s,r}{\Psi_i}\\
				\node{Q}\arrow{e,b}{\bar c_i}\node{Q_i}
			\end{diagram}
			$
		\end{center}
		commutes, where $\bar c_i$ is a $G'$--equivariant $1$--lipschitz median-preserving map.  Indeed, $Q_i$ is obtained from $Q\subseteq \widehat{\mathcal Y}$ by discarding the $V$--coordinates from consistent tuples for all $V\not\in\mathfrak U_i$.  In other words, $\bar c_i$ is the restriction of the natural projection $\prod_{U\in\mathfrak U}\widehat T_U\to \prod_{U\in\mathfrak U_i}\widehat T_U$.  From this, one easily checks that $\prod_i\bar c_i:Q\to\prod_iQ_i$ is an isometry.
		
		Moreover, since $\bar c_i$ is the restriction to $Q$ of the natural projection $\prod_{V\in\mathfrak U}\widehat T_V\to \prod_{U\in\mathfrak U_i}\widehat T_U$, and the decomposition $\mathfrak U=\bigsqcup_i\mathfrak U_i$ is preserved by $G$, there is an isometric action of $G$ on $\prod_iQ_i$ that extends the diagonal action of $G'$ and makes $\prod_i\bar c_i$ a $G$--equivariant map.
		
		Since Theorem \ref{thm: median hulls} implies that the $Q_i$ are proper spaces, $G'$--cocompactness follows once we prove that $G'$ acts on $Q_i$ coboundedly.  This, plus the fact that $Q_i$ is a quasiline, will follow once we prove that $\pi_{U_i}:\newcone X_i\to\mathcal CU_i$ is a quasi-isometry, since we already saw that the $G'$--action on $\mathcal CU_i$ is cobounded and the latter space is a quasiline.  By the distance formula for $(\newcone X_i,\mathfrak S_i)$, the fact that $\pi_{U_i}$ is a quasi-isometry will follow once we have proved the ``moreover'' clause of the proposition.
		
		So, to conclude, it suffices to exhibit $B$ such that $\diam(\mathcal CV)<B$ unless $V=U_i$ for some $i$.  If $V\transverse U_i$ or $U_i\propnest V$ for some $i$, then consistency uniformly bounds $\diam(\mathcal CV)$, and if $V\orth U_i$ for all $i$, then Lemma \ref{lem: orthogonal supports are bounded} gives the desired bound.  Hence it remains to consider the case of $V\propnest U_i$.  For any such $V$, we have $\diam(\mathcal CV)<\infty$ by Proposition \ref{prop:quasi-linear-domains}, so it suffices to show that $\mathfrak U_i$ contains finitely many $G'$--orbits.
		
		Let $\Omega\subseteq \mathcal CU_i$ be a bounded set such that $G'\cdot \Omega=\mathcal CU_i$ and $\mathcal CU_i-\Omega$ has two unbounded components $\Omega^-,\Omega^+$.  Choose $z^-,z^+\in X$ such that $\pi_{U_i}(z^\pm)\subseteq \Omega^\pm$ and $\dist_{U_i}(z^\pm,\Omega)>100E$.  Recall that $\diam(\mathcal CV)\geq K$ for all $V\in\mathfrak U_i$.  Moreover, by consistency and bounded geodesic image, if $V\in\mathfrak U_i$ and $\rho^V_{U_i}\cap \Omega\neq\emptyset$, then (up to relabelling) $\dist_V(z^\pm,\pi_V(p^\pm))\leq E$, so $\dist_V(z^-,z^+)>K-2E>E$.  Hence \cite[Lem. 13.5]{Casals-RuizReal24} implies that $|\{V\in\mathfrak U_i:\rho^V_{U_i}\cap \Omega\neq\emptyset\}|<\infty$.  But if $V\in\mathfrak U_i$, then $G'\cdot V$ contains some $gV$ in the above set, so we are done. 
	\end{proof}
	
	\newsection{HHS actions with bounded orbits}\label{sec:bounded-orbits}
	The purpose of this section is to prove:
	
	\begin{prop}\label{prop:bounded-orbits}
		Let $(X,\mathfrak S)$ be an HHS.  Then there exists $k_0$, depending only on the HHS parameters, such that for all $k_1\geq k_0(2k_0+1)$, there exist $k_2,\dots,k_7$, depending only on $k_1$ and the HHS parameters, and $T'$ depending only on $k_1,\ldots,k_4$ and the HHS parameters, such that the following holds.
		
		Let $G\leq \Isom(X)$ act on $(X,\mathfrak S)$ by HHS automorphisms, with bounded orbits in $X$. Let $Z = Z(G,X,k_1) \coloneqq \{x\in X:\dist(x,gx)\leq k_1\ \forall g\in G\}$. Then
		\begin{enumerate}
			\item \label{item:uniformly-bounded-orbits}   $Z\neq \emptyset$;
			
			\item \label{item:injective-coarse-fixed-point-set} there exists a non-empty injective space $Y$, equipped with the trivial $G$ action, and a $G$--equivariant $(k_2,k_2)$--quasi-isometry $q:Z\to Y$, where $Z$ has the subspace metric inherited from $X$;
			
			\item \label{item:coarse-fixed-point-EHQC} $Z$ is $(k_3,k_4)$--EHQC in $(X,\mathfrak S)$ and a $k_5$-coarse median subalgebra;

            \item \label{it: fixed point set coarse lipschitz retract} $Z$ is a $k_6$--coarse lipschitz retract of $X$;
			
			\item \label{it:zs-are-close}
			there is an unbounded, increasing function $\omega_0:[0,\infty)\to[0,\infty)$, depending only on the HHS parameters, such that $\diam_X(G\cdot x)\geq \omega_0(\dist(x,Z))$ for all $x\in X$;
			
			\item \label{it:finite index hqc fixed point set} if $G_0 \unlhd G$ has the property that $gV\notorth V$ for all $V\in\mathfrak S$ such that $\pi_V(Z)$ has diameter at least $T'$, then $H_{\theta}(Z) \subseteq Z(G_0, X,k_7)$. This holds in particular if $G_0$ is contained in the intersection of all subgroups of $G$ with index $\leq \hhscomp!$.
		\end{enumerate}
        Finally, suppose that there exists $B$ such that $\diam(\mathcal CU)\leq B$ or $\diam(\mathcal CU)=\infty$ for all $U\in\mathfrak S$. Then one can replace $\hhscomp$ in item \eqref{it:finite index hqc fixed point set} with the maximal $r$ such that $\mathfrak S$ contains a pairwise-orthogonal set $\{U_1,\ldots,U_r\}$ with each $\mathcal CU_i$ unbounded.
	\end{prop}
	
	\begin{proof}
    In view of Lemma \ref{lem:assume-graph}, Theorem \ref{thm:coarse-injective} and Corollary \ref{cor: coarsely dense in injective hull}, we can and shall assume that $X$ has been chosen in its $G$--equivariant quasi-isometry class so that there is a constant $k_0$, depending only on the HHS parameters, and a $G$--equivariant, $k_0$--coarsely surjective, $(k_0,k_0)$--quasi-isometry $\iota:X\to \injhull(X)$, where $\injhull(X)$ is the injective hull of $X$.

    \medskip

		\noindent\textbf{Construction of $Y$.} Since $G$ acts on $\injhull(X)$ with bounded orbits, the set $Y\subseteq \injhull(X)$ of points fixed by $G$ is nonempty and injective by \cite[Prop. 1.2]{LangInjective13}.  Hence $Z_0:=\iota^{-1}(\neb_{k_0}(Y))$ is a nonempty, $G$--invariant subset of $X$.  The same computation as in the proof of Corollary \ref{cor: uniformly bounded orbits} shows that $\dist(x,gx)\leq k_0(2k_0+1)$ for all $x\in Z_0$.  

        \setcounter{claim}{0}
		\begin{claim}\label{claim:z_0-t}
			There is an unbounded, increasing function $t:[0,\infty)\to[0,\infty)$, depending only on the HHS parameters, such that $\diam_X(G\cdot x)\geq t(\dist(x,Z_0))$ for all $x\in X$.
		\end{claim}
		
		\begin{proofofclaim}{\ref{claim:z_0-t}}
			Let $x\in \injhull(X)$ and let $R=\diam_{\injhull(X)}(G\cdot x)$.  Consider $G\cdot x$ with the metric inherited from $\injhull(X)$, and the canonical isometric embedding $\iota:G\cdot x\to \injhull(G\cdot x)$ (see \cite[Sec. 3]{LangInjective13}).  Note that $\diam(\injhull(G\cdot x))\leq R$.  On the other hand, by \cite[Prop. 3.5]{LangInjective13}, the inclusion $G\cdot x\to \injhull(X)$ extends to an embedding $h:\injhull(G\cdot x)\to \injhull(\injhull(X))=\injhull(X)$, and $h$ is $G$--equivariant.  In other words, $\injhull(X)$ contains a $G$--invariant injective subspace $\injhull(G\cdot x)$ that has diameter $R$ and contains $G\cdot x$.  Proposition 1.2 of \cite{LangInjective13}, applied to this subspace, shows that $\dist_{\injhull(X)}(Y,G\cdot x)\leq R$, so $\dist_{\injhull(X)}(Y,x)\leq 2R$.  In particular, if $z\in X$, then $\dist_{\injhull(X)}(\iota(z),Y)\leq 2\diam_{\injhull(X)}(G\cdot\iota(z))$, which implies the claim since $\iota$ is an equivariant, $k_0$--coarsely surjective, $(k_0,k_0)$--quasi-isometry.
		\end{proofofclaim}
		
		Let $k_1\geq k_0(2k_0+1)$ be given, and let $Z = Z(G,X,k_1)$ be as in the statement.  
		
		\begin{claim}\label{claim:z_0-z-close}
			The set $Z$ is nonempty and $G$--invariant, and the Hausdorff distance in $X$ between $Z$ and $Z_0$ is bounded above in terms of $k_0$ and $k_1$.
		\end{claim}
		
		\begin{proofofclaim}{\ref{claim:z_0-z-close}}
			That $Z$ is $G$--invariant is immediate from the definition.  Since $z\in Z_0$ implies $\diam(G\cdot z)\leq k_0(2k_0+1)$, we have $Z_0\subseteq Z$; in particular, $Z\neq \emptyset$.  Let $t$ be as in Claim \ref{claim:z_0-t}.  Then $Z\subseteq \neb_{t^{-1}(k_1)}(Z_0)$, yielding the required Hausdorff distance bound. 
		\end{proofofclaim}
		
		Item \eqref{it:zs-are-close} follows immediately from Claims \ref{claim:z_0-t} and \ref{claim:z_0-z-close}. Since $Y$ is injective --- and therefore an absolute 1-lipschitz retract (see e.g. \cite[Prop. 2.2]{LangInjective13}) --- and an isometrically embedded subspace of $\injhull(X)$, there is a 1-lipschitz retraction $\injhull(X) \rightarrow Y$. Using Claim \ref{claim:z_0-z-close} and the fact that $\iota$ is $(k_0,k_0)$--quasi-isometry, we can therefore compose uniformly coarsely lipschitz maps $X \rightarrow \injhull(X) \rightarrow Y \rightarrow Z_0 \rightarrow Z$ to produce a coarsely lipschitz map $\rho:X \rightarrow Z$. Up to uniformly perturbing $\rho$ (here, uniformity allows dependence on $k_0,k_1$, and the HHS parameters, as in the statement), we can assume it is the identity on $Z$, proving Item \eqref{it: fixed point set coarse lipschitz retract}.
		
		\medskip
		
		\noindent\textbf{Construction of $q$.}
		By Claim \ref{claim:z_0-z-close}, $\iota:Z\to \injhull(X)$ is a $(k_0,k_0)$--quasi-isometric embedding whose image is $G$--invariant and at Hausdorff distance at most $k_0(t^{-1}(k_1)+2)$ from $Y$. Recall that $Y$, with the subspace metric inherited from $\injhull(X)$, is injective.  Let $\{z_i\}_i$ be a subset of $Z$ containing one point in each $G$--orbit.  For each $i$, let $q(z_i)$ be some point in $Y$ that is $k_0(t^{-1}(k_1)+2)$--close to $\iota(z_i)$.  Given $z\in Z$, let $z_i$ be the element in our set of orbit representatives belonging to $G\cdot z$.  Let $q(z)=q(z_i)$.  Note that for any $g\in G$ such that $gz_i=z$, we have $q(gz_i)=q(z_i)=gq(z_i)$ since $G$ acts on $Y$ trivially.  Hence $q:Z\to Y$ is $G$--equivariant.  Moreover, for any $z=gz_i$, since $\iota$ is also equivariant, we have $$\dist_{\injhull(X)}(q(z),\iota(z))=\dist_{\injhull(X)}(q(z_i),g\iota(z_i))= \dist_{\injhull(X)}(q(z_i),\iota(z_i))\leq k_0(t^{-1}(k_1)+2),$$ so $q$ is at uniformly bounded distance from the restriction of $\iota$ and is therefore a uniform quasi-isometry.  
		\medskip
		
		\noindent\textbf{Proof of \eqref{item:coarse-fixed-point-EHQC}.}
		We first check that $Z$ is a coarse median subalgebra (Definition \ref{defn:coarse-median-subalgebra}) with constant depending only on the HHS parameters and $k_1$.  First recall that by Condition (C1) of the definition of a coarse median (Definition \ref{defn:coarse-median}), there exists $\alpha$, depending only on the coarse median parameters (and hence, by Corollary \ref{cor:coarse-median}, only on the HHS parameters) such that the coarse median $\mu:X^3\to X$ is $(\alpha,\alpha)$--coarsely lipschitz.  Now let $z_1, z_2, z_3 \in Z$, let $m \coloneqq \mu(z_1,z_2,z_3)$, and let $g\in G$.  Then $\dist(z_i,gz_i)\leq k_1$ for $i\in\{1,2,3\}$, by the definition of $Z$.  Hence $\dist(m,gm)\leq 3\alpha k_1+\alpha$, which depends only on $k_1$ and the HHS parameters.  Therefore Claims \ref{claim:z_0-t} and \ref{claim:z_0-z-close} imply that $\dist(m, Z)$ is bounded in terms of $k_1$ and the HHS parameters.
				
		Let us now check that $Z$ is uniformly EHQC. To this end, let $p^+,p^- \in Z$. We will use Theorem~\ref{thm: median hulls} to find to the required hierarchy path.
		
		Let $W \coloneqq G \cdot H_\theta(p^+,p^-)$ and let $M \coloneqq \theta + Ek_1 + 3E$. Then $H_\theta(p^+,p^-) \subseteq W \subseteq H_M(p^+, p^-)$.
		It follows from Lemma~\ref{lem: hulls are quasiconvex} that $W$ is $\kappa'$-hierarchically quasiconvex, where $\kappa'(0) \coloneqq h_M(0)$ and $\kappa'(r) \coloneqq h_\theta(r)$ if $r > 0$. Let $\mathfrak S_W$ be the HHS structure on $W$ from Lemma~\ref{lem: HHS structure on HQC subspaces} and recall from that lemma that the action of $G$ on $(W, \mathfrak S_W)$ is by HHS automorphisms. The hierarchical hull $H^W_M(p^+,p^-)$ of $p^+,p^-$ in $W$ with respect to the constant $M$ is just the intersection $W \cap H^X_M(p^+, p^-)$, so $W = H^W_{M, z_0}(p^+, p^-)$ for all $z_0 \in W$. Thus, by Theorem~\ref{thm: median hulls}, there is a complete connected median space $(Q, d_Q)$, of rank at most $\hhscomp$, equipped with a $G$--action and a $G$-equivariant $L$-quasi-isometry $\Psi:W \rightarrow Q$, where $L \geq 1$ only depends on $M$ and HHS parameters for $(W, \mathfrak S_W)$, and hence only on the HHS parameters for $(X, \mathfrak S)$. 
		
		By \cite[Thm.~1.1]{BowditchProperties16}, $Q$ has a metric $\sigma$ such that $(Q,\sigma)$ is a CAT(0) space, any isometry of $(Q,d_Q)$ is an isometry of $(Q,\sigma)$, and $d_Q / \sqrt{\hhscomp} \leq \sigma \leq d_Q$. Let $\gamma_Q \subseteq Q$ be the unique $\sigma$--geodesic connecting $\Psi(p^+)$ and $\Psi(p^-)$. Then  $\gamma_W \coloneqq \Psi^{-1} (\gamma_Q)$ is a $(C,C)$-quasigeodesic, where the constant $C$ depends only on $L$ and $\hhscomp$. The $\sigma$--geodesic $\gamma_Q$ is also a (reparametrised) $d_Q$-geodesic (see the proof of \cite[Thm.~8.3]{BowditchProperties16}). Since $\Psi$ is quasimedian, it follows from Proposition \ref{prop:hierarchy-path-char} that $\gamma_W$ is a $k_3$--hierarchy path in $(W, \mathfrak S_W)$ and therefore in $(X, \mathfrak S)$, where $k_3$ depends only on the HHS parameters of $X$. Given $x \in \gamma$, convexity of the metric $\sigma$ implies that $\sigma(x, g \gamma) \leq \max\{\sigma(g_1 \Psi(p^+), g_2 \Psi(p^+)), \sigma(g_1 \Psi(p^-), g_2 \Psi(p^-))\}$ for all $g \in G$. Therefore $\diam_\sigma(G \cdot x) \leq Lk_1 + L$ for all $x \in \gamma$ and $\diam_X(G \cdot y) \leq L^2(k_1 + 1) + L$ for all $y \in \gamma_W$. By Claim \ref{claim:z_0-t} and Claim \ref{claim:z_0-z-close}, we then have $\gamma_W \subseteq \mathcal N_{k_4}(Z)$ for some $k_4 \geq 0$ depending only on the HHS parameters.
        \medskip
		
		\noindent\textbf{Proof of \eqref{it:finite index hqc fixed point set}.}
        Let $T \coloneqq Ek_1 + 4E$, let $T' \coloneqq k_3(2T + k_3) + k_4$ and let $\mathfrak U \coloneqq \bigcup_{a,b \in Z} \relevant_{T'}(a,b)$.  Since $G$ preserves $Z$, the set $\mathfrak U$ is $G$--invariant.  

        Fix $V\in\mathfrak U$ and $g\in G$.  Note that $gV$ and $V$ cannot be $\propnest$--related.  We next argue that $gV$ and $V$ cannot be transverse.  Indeed, suppose $gV \transverse V$. It follows from the definition of $\mathfrak U$ and the fact that $Z$ is $(k_3,k_4)$-EHQC that there exists $a \in Z$ such that $\dist_V(a,\rho^{gV}_V) > T$ and $\dist_V(a,\rho^{g^{-1}V}_V) = \dist_{gV}(ga, \rho^V_{gV}) > Ek_1 + 4E$. However the consistency axiom \cite[Def. 1.1.(4)]{BehrstockHierarchically19} implies that $\dist_{gV}(a,\rho^V_{gV}) \leq E$, so $\dist_{gV}(a,ga) \geq Ek_1 + 3E$, so $\dist(a,ga) \geq k_1 + 2E$, which contradicts the assumption that $a \in Z$. Therefore either $gV = V$ or $gV \orth V$.

        Now assume that $G_0\unlhd G$ has the property that $gV\notorth V$ for all $V\in\mathfrak U$ and $g\in G_0$.  The preceding discussion implies that $G_0$ acts on $\mathfrak U$ trivially.  Let $x \in H_{\theta}(Z)$ and $g \in G_0$. A standard hyperbolic geometry argument shows that $\dist_U(x,gx) \leq k_7'$ for all $U \in \mathfrak U$, where $k_7'$ is a constant which depends only on $E$ and $k_1$. If $U \in \mathfrak S - \mathfrak U$, then $\hull_U(Z)$ is uniformly bounded in terms of $E$ and $T'$, so $\pi_U(H_{\theta}(Z))$ is uniformly bounded in terms of $E$ and $T'$.
		It follows by the distance formula (Theorem \ref{thm:distance-formula}) that $\dist(x,gx) \leq k_7$, where $k_7$ depends only on $k_1$ and the HHS parameters.  This proves the first part of item \eqref{it:finite index hqc fixed point set}.    

        To conclude, suppose that $G_0$ is contained in the intersection of all subgroups of $G$ with index $\leq \hhscomp!$.  We need to show that $G_0\cdot V=\{V\}$ for each $V\in\mathfrak U$.  Now, we have shown that $G \cdot V$ is a pairwise orthogonal set, and therefore $|G \cdot V| \leq \hhscomp$.  Hence $\stabilizer_G(V)$ has index at most $\hhscomp!$ in $G$.  Therefore, $G_0\leq \stabilizer_G(V)$ for all $V\in\mathfrak U$, as required. \\

        Finally, suppose that there exists $B<\infty$ such that for all $U\in\mathfrak S$, either $\mathcal CU$ is unbounded, or $\diam(\mathcal CU)\leq B$.  In the preceding argument, we can assume that $T'\geq B$, by, for instance, choosing $k_1$ sufficiently large.  Thus $\mathfrak U$ consists only of elements $U$ with $\mathcal CU$ unbounded, and the above argument worked as long as $G_0$ is contained in each subgroup of index at most $r!$, where $r$ bounds the size of pairwise orthogonal sets in $\mathfrak U$. 
       \end{proof}
		
	\newsection{Some invariant subspaces}\label{subsec:initial-quasiflat}
	The next definition abstracts a property of the action by HHS automorphisms of an HHG on itself (see Definition \ref{defn:HHG}) by \cite[Thm.~3.1]{DurhamCorrection20}.
	
	\begin{defn}[Hierarchically semisimple]\label{defn:hierarchically-semisimple}
		An action of a group $H$ on an HHS $(X,\mathfrak S)$ by HHS automorphisms is \emph{hierarchically semisimple} if, for all $g\in H$ and all $U\in\mathfrak S$ such that $gU=U$ and $\langle g\rangle$ has unbounded orbits in $\mathcal CU$, the isometry $g$ of $\mathcal CU$ is loxodromic.
	\end{defn}
	
	Let $(X, \mathfrak S)$ be an HHS and let $A\leq \Isom(X)$ be an infinite finitely generated virtually abelian group acting hierarchically semisimply on $(X,\mathfrak S)$ by HHS automorphisms, with the $A$--orbits in $X$ unbounded. The goal of this section is to analyse some canonical $A$-invariant HQC and EHQC subspaces, using the results of the previous sections.

	By Lemma \ref{lem:assume-graph} and Lemma \ref{lem:graph-free}, we can assume that $X$ is a graph and the $A$--action on $X$ is free. For convenience, up to a uniform enlargement of the HHS parameters, we can and shall assume that $(X,\mathfrak S)$ is \emph{normalised} in the sense of \cite[Rem. 1.3]{BehrstockHierarchically19}, which is to say that $\pi_U$ is $E$--coarsely surjective for all $U\in\mathfrak S$. Recall that $\hhscomp$ is the complexity of $\mathfrak S$.

	\begin{lem} \label{lem: fixed boundary points}
		There exists a finite index abelian subgroup $A' \leq A$ and a pair of points $p^+,p^- \in \partial X$ such that:
		\begin{enumerate}
			\item \label{it: invariant boundary points} $ap^+ = p^+$ and $ap^- = p^-$ for all $a \in A'$;
			\item \label{it: identical support} $\support(p^+) = \support(p^-)$;
			\item \label{it: distinct projections} for all $U \in \support(p^\pm)$ we have $\pi_U(p^+) \neq \pi_U(p^-)$;
			\item \label{it: unbounded projections} for any $x \in X$, the projection $\pi_U(A \cdot x)$ is unbounded if and only if $U \in \support(p^\pm)$.
		\end{enumerate}
	\end{lem}
	\begin{proof}
		Since $A$ is assumed to be finitely generated and to act with unbounded orbits, \cite[Thm. 5.1]{PetytUnbounded23} produces $1 \leq n \leq \hhscomp$ and domains $U_1,\ldots,U_n\in\mathfrak S$ such that the following hold for any choice of basepoint $x_0\in X$:
		\begin{itemize}
			\item $U_i\orth U_j$ for $1\leq i<j\leq n$ (which is why $n\leq \hhscomp$);
			\item for all $W\in\mathfrak S$, if the set $\pi_W(A\cdot x_0)$ is unbounded, then $W\nest U_i$ for some $i\leq n$;
			\item $\pi_{U_i}(A \cdot x_0)$ is unbounded for all $i\in\{1,\ldots,n\}$;
			\item $A$ preserves the set $\{U_i\}_{i=1}^n$.
		\end{itemize}
		
		Let $\dddot A\leq A$ be the kernel of the $A$--action on $\{U_i\}_{i=1}^n$; note that $[A:\dddot A]\leq \hhscomp!$.  Let $\ddot A\leq \dddot A$ be a finite index abelian subgroup.  
		
		For each $i\leq n$, the subgroup $\ddot A$ acts on $\mathcal CU_i$ with unbounded orbits.  Let $a_1,\ldots,a_k\in \ddot A$ generate $\ddot A$.  If each $\langle a_j \rangle$ acts on $\mathcal CU_i$ with bounded orbits, then since $\ddot A$ is abelian, $\ddot A\cdot\pi_{U_i}(x_0)$, and hence $\pi_{U_i}(A\cdot x_0)$, is bounded, contradicting the choice of $U_i$.  Combined with the hierarchical semisimplicity assumption, this implies that some $b_i \in \ddot A$  acts on $\mathcal CU_i$ loxodromically, for each $i \in \{1, \dots, n\}$.  The classification of actions on hyperbolic spaces (see, for instance, \cite[Prop. 3.1]{CCMT:amenable}) therefore allows three possibilities for the $\ddot A$--action on $\mathcal CU_i$: it can be \emph{lineal, focal,} or \emph{general-type}.  Since $\ddot A$ is abelian, it contains no rank-$2$ free semigroup, which rules out general-type and focal actions.
		
		Hence $\boundary\mathcal CU_i$ contains exactly two points, denoted $p_i^\pm$, such that $\{p_i^+, p_i^-\}$ is $\ddot A$-invariant; these are the unique fixed points for each loxodromic element of $\ddot A$. Let $p^+, p^- \in \partial X$ be the boundary points whose supports are $\support(p^+) = \support(p^-) = \{U_1, \dots, U_n\}$, whose constituent boundary points are $p_i^+, p_i^- \in \partial \mathcal CU_i$, respectively, for each $i \in \{1, \dots, n\}$, and with coefficients $\{a_U = 1/n : U \in \support(p^\pm)\}$. Let $A' \leq \ddot A$ be a further finite index subgroup which fixes each $p_i^\pm$. Then Items~\eqref{it: invariant boundary points}--\eqref{it: unbounded projections} all hold.
	\end{proof}

	\begin{defn} \label{defn:P(A)}
		Let $P = P(A) \subseteq X$ be the set of $x$ such that $\dist_V(x,\rho^U_V)\leq 10E + \theta$ whenever $U\in\support(p^\pm)$ and $V\in\mathfrak S$ satisfies $U\propnest V$ or $V\transverse U$. 
	\end{defn}
	
	\begin{remark}
		Note that $P$ is $A$--invariant, nonempty, and $\kappa'$--median quasiconvex, where $\kappa'$ just depends on the HHS parameters. Thus, by Proposition \ref{prop:RST} and Lemma \ref{lem: HHS structure on HQC subspaces}, $P$ is an HHS with parameters depending only on the HHS parameters of $X$. Up to enlarging $E$ uniformly, we can assume that $E$ is the HHS constant for $P$. We also normalise (see \cite[Rem. 1.3]{BehrstockHierarchically19}) so that $\diam(\mathcal CU) \leq 30E+2\theta$ for all $U \in \mathfrak S$ such that $V\propnest U$ or $V\transverse U$ for some $V \in \support(p^+)$. From now on, we work mostly with $P$ and its HHS structure.  It may be helpful to observe that in the case where $\support(p^+)$ has a single element $U$, the subspace $P$ is defined exactly like the standard product region $P_U$, except using a different constant; $P$ and $P_U$ are at finite Hausdorff distance.  
	\end{remark}
	
	\begin{defn} \label{defn:U,P and X(A)}
		Let $\mathfrak U \coloneqq \{U \in \mathfrak S : U \nest V$ for some $V \in \support(p^+)\}$ and let $k \geq 0$.
		\begin{itemize}
			\item For each $U \in \mathfrak S - \support^\perp(p^\pm)$, let $\gamma_U \coloneqq \hull_U(\{p^+,p^-\})$ (i.e. the union of all $(1,20E)$--quasi-geodesics joining $\pi_U(p^-)$ to $\pi_U(p^+)$).
			\item Let $(\newcone P, \mathfrak S - \mathfrak U)$ be the canonical $\mathfrak U$--cone-off space of $P$ from Definition \ref{defn:modern-cone-off} and Corollary \ref{cor:uniform-cone-off} and let $c:P \rightarrow \newcone P$ be the (set-theoretic) identity.
			\item Let $\overline P_k = \overline P_k(A) \coloneqq \{x \in P : \dist_V(x, \gamma_V) \leq k$ for all $V \in \mathfrak S - \support^\perp(p^\pm)\}$.
			\item Let $X_k(A) \coloneqq \{x \in \overline P_k : c(x) \in Z(A,\newcone P,k)\}$, where $Z(A,\newcone P, k) = \{x \in \newcone P: \dist(x,ax) \leq k \; \forall a \in A\}$.
			\item Let $k_0$ be the constant obtained by applying Proposition \ref{prop:bounded-orbits} to the HHS $(\newcone P, \mathfrak S - \mathfrak U)$. Let $k_1 \coloneqq k_0(2k_0 + 1)$ and let $k_2,k_3,k_4$ be the corresponding constants from Proposition \ref{prop:bounded-orbits}. Let $\widecheck Z \coloneqq Z(A,\newcone P,k_1)$.
            \qedhere
		\end{itemize}
	\end{defn}
	
	\begin{prop} \label{prop: almost invariant hull}
		There exist constants $k \geq k_1$, $M \geq M' \geq \theta$ and $L_1,L_2,L_3.L_4 \geq 1$, depending only on the HHS parameters, such that the following hold.
		\begin{enumerate}
			\item \label{it:P(A) nonempty and HQC} $\overline P = \overline P(A) \coloneqq \overline P_k$ is nonempty, $A$--invariant, and uniformly median quasi-convex.
			
			\item \label{it:X(A) nonempty and a median subalgebra} $X(A) \coloneqq X_k(A)$ is a nonempty, $A$--invariant uniformly coarse median subalgebra.
			
			\item \label{it:choice-of-z}
			The image $c(X(A))$ is $L_1$--Hausdorff close to $\widecheck Z$. In addition, there is a $(L_1,L_1)$--quasi-isometry $\Phi: c(X(A)) \to Y$, where $Y$ is an injective space and, equipping $Y$ with the trivial $A$--action, $\Phi$ is $A$--equivariant.
			
			\item \label{it: almost an invariant convex hull} 
			$H_{M',z}(p^-,p^+) \neq \emptyset$ and $A \cdot H_{M',z}(p^-,p^+) \subseteq H_{M,z}(p^-, p^+)$ for all $z \in X(A)$. 
			
			\item \label{it: properties of H} Fix $z_0 \in X(A)$ and let $H \coloneqq A \cdot H_{M',z_0}(p^-,p^+)$. 
			Then 
			\begin{itemize}
				\item $H$ is $\kappa$--median convex in $P$, where $\kappa$ depends only on the HHS parameters;
				\item there is an $A$--equivariant, $L_2$--quasi-median gate map $\gate_H: P \rightarrow H$;
				\item there is a  connected, proper, finite rank median space $Q$, equipped with an $A$--action and an $A$--equivariant, $L_2$--quasi-median, $(L_2,L_2)$--quasi-isometry $\Psi: H \rightarrow Q$. 
			\end{itemize}
			\item \label{it:barP(A) is a quasi-product}The product map $c|_{\overline P} \times \gate_H|_{\overline P} : \overline P \rightarrow \newcone P \times H$ is an $A$--equivariant, $L_3$--quasi-median, $(L_3,L_3)$--quasi-isometry.
			\item \label{it:X(A) is a quasi-product} Let $\xi \coloneqq \Phi \circ c|_{X(A)}$ and $\eta \coloneqq \Psi \circ \gate_H|_{X(A)}$. Then $\xi \times \eta: X(A) \rightarrow Y \times Q$ is an $A$--equivariant $(L_4,L_4)$--quasi-isometry. 
            \item \label{it: X(A) is EHQC and a coarse retract} $X(A)$ is uniformly EHQC and a uniformly coarse lipschitz retract of $X$.
		\end{enumerate}
	\end{prop}
	
	\begin{proof}
		We start with a claim which will help to establish the $A$--invariance of the various subspaces we are interested in.
		
		\setcounter{claim}{0}
		\begin{claim} \label{claim: almost invariant hull (non-orthogonal)}
			If $U \in \mathfrak S - \support^\perp(p^\pm)$ and $a \in A$ then $a \gamma_U = \gamma_{aU}$.
		\end{claim}
		\begin{proofofclaim}{\ref{claim: almost invariant hull (non-orthogonal)}}
			Let $a \in A$ and $U \in \support(p^\pm)$. Recall that there is an element $b \in \ddot A$ which acts loxodromically on $\mathcal CU$. So $aba^{-1}$ acts loxodromically on $\mathcal CaU$ with quasi-axis $a\gamma_U$. Some nonzero power of $aba^{-1}$ belongs to $\ddot A$, and therefore has $\{\pi_{aU}(p^+),\pi_{aU}(p^-)\}$ as its unique fixed points. Thus $a\gamma_U = \gamma_{aU}$. It follows that $a \gamma_V = \gamma_{aV}$ for all $V \in \mathfrak S - \support^\perp(p^\pm)$.
		\end{proofofclaim}

		We now prepare to use the realisation theorem to find $k$. Let 
        \[
        P' \coloneqq \{x \in P : \dist_V(x, \gamma_V) \leq E \; \forall \; V \in \mathfrak S \text{ such that } V \transverse W \text{ or } W \nest V \text{ for some } W \in \support(p^\pm)\}.
        \]
        By the partial realisation axiom (\cite[Defn. 1.1.(8)]{BehrstockHierarchically19}), $P'\neq\emptyset$.
		
		Recall that $(\newcone P, \mathfrak{S - U})$ is an HHS with parameters $(\check E, {\check\theta_u})$ depending only on the HHS parameters of $(P, \mathfrak S)$, and $A$ acts on $\newcone P$ by isometries and on $(\newcone P, \mathfrak{S - U})$ by HHS automorphisms.  As a set, $\newcone P = P$, and the identity map $c: (P,\dist) \rightarrow (\newcone P, \check \dist)$ is $A$--equivariant and uniformly quasi-median.  Since the set $\pi_V(A\cdot y)$ is bounded for all $U \in \mathfrak {S - U}$ and $y \in P$, \cite[Thm. 5.1]{PetytUnbounded23} implies that $A$ acts on $\newcone P$ with bounded orbits. Thus Proposition \ref{prop:bounded-orbits} implies that $\widecheck Z \neq \emptyset$.
				
		\begin{claim} \label{claim: consistent z}
			There exists $\kappa > 0$, depending only on the HHS parameters such that the following holds.
			Let $x_0 \in P'$ and $y_0 \in \newcone P$. For each $U \in \support^\orth(p^\pm)$, let $z_U \in \pi_U(y_0)$ and, for each $U \in \mathfrak S - \support^\orth(p^\pm)$, let $z_U \in \gamma_U$ be the closest point projection of $\pi_U(x_0)$ onto $\gamma_U$. Then $(z_U)_{U \in \mathfrak S}$ is $\kappa$-consistent.
		\end{claim}
		\begin{proofofclaim}{\ref{claim: consistent z}}
			Let $U, V \in \mathfrak S$ and suppose $U \transverse V$ or $U \propnest V$. If $U,V \in \support^\perp(p^\pm)$ then the required inequality follows from the consistency axioms applied to $y_0 \in \newcone P$. So suppose without loss of generality that $V \notin \support^\perp(p^\pm)$. 
			
			Suppose that $V$ is not nested in any $W \in \support(p^\pm)$. Then the set $\mathcal W_V \coloneqq \{W \in \support(p^\pm) : W \propnest V$ or $W \transverse V\}$ is nonempty and $\pi_V(p^+) = \pi_V(p^-) = \cup_{W \in \mathcal W_V} \rho^W_V$. So, since the elements of $\mathcal W_V$ are pairwise orthogonal, \cite[Lem.~1.5]{DurhamBoundaries17} implies that $\diam(\gamma_V)$ is uniformly bounded in terms of $E$. Also, $d_V(x_0,z_V) \leq E$, by the definition of $z_V$.
			
			Hence, if $d_V(\rho^U_V, \gamma_V) \leq 10E$, then $\dist_V(\rho^U_V,z_V)$ is bounded uniformly in terms of $E$, and we are done.
			Suppose that $\dist_V(\rho^U_V,\gamma_V)>10E$, so in particular, $\dist_V(z_V,\rho^U_V) > 10E$ and $d_V(x_0, \rho^U_V) > 9E$.  If $W\in\mathcal W_V$, then $\dist_V(\rho^U_V,\rho^W_V)>10E$, so $U,W$ are not $\nest$--related (because of \cite[Defn. 1.1.(4)]{BehrstockHierarchically19}) and are not orthogonal (by \cite[Lem. 1.5]{DurhamBoundaries17}), so $U\transverse W$.  In particular, $W\in \mathcal W_U$, which is therefore nonempty, and $\pi_U(p^+) = \pi_U(p^-) = \cup_{W \in \mathcal W_U} \rho^W_U$.  Thus $\gamma_U$ has uniformly bounded diameter (in terms of $E$) and $d_V(x_0,z_U) \leq E$.
			
			If $U\transverse V$, then $\dist_U(x_0,\rho^V_U)\leq E$, by consistency of $x_0$.  Hence $\dist_U(\rho^V_U,z_U)\leq 2E$, as required.  The other possibility is that $U\propnest V$. 
            In this case $\diam_U(\rho^V_U(\pi_V(x_0))\cup\pi_U(x_0))\leq E$, by the consistency axiom for nested domains, applied to $x_0$ (see \cite[Defn. 1.1.(4)]{BehrstockHierarchically19}). The bounded geodesic image axiom \cite[Defn. 1.1.(7)]{BehrstockHierarchically19} further implies that $d_U(\rho^V_U(\pi_V(x_0)),\rho^V_U(z_V)) \leq E$, which implies that $d_U(z_U,\rho^V_U(z_V)) \leq 3E$.

			It remains to consider the case where $V \nest W$ for some $W \in \support(p^\pm)$.
			Then $U \notin \support^\perp(p^\pm) \cup \support(p^\pm)$. If $W' \propnest U$ or $W' \transverse U$ for some $W' \in \support(p^\pm)$ then $V \orth W'$ by \cite[Defn. 1.1.(3)]{BehrstockHierarchically19}, so
            $d_U(\rho^V_U, z_U) \leq d_U(\rho^V_U,\rho^{W'}_U) +\diam(\rho^{W'}_U)+ d_U(\rho^{W'}_U, z_U) \leq 10E$ (using \cite[Lem.~1.5]{DurhamBoundaries17} and the definition of $z_U$). So suppose that $U \propnest W'$ for some $W' \in \support(p^\pm)$. Since $U \not\perp V$, we have $W = W'$. 
            
            Suppose that $V = W$ and $d_V(z_V, \rho^U_V) > 10E + 2\Morse_E(1,20E)$. Since $x_0 \in P'$, we have $\dist_V(x_0,z_V) \leq E$. Assume without loss of generality that $\pi_V(x_0)$ and $z_V$ are in the same connected component of $\mathcal CV - \mathcal N_{10E + 2\Morse_E(1,20E)}(\rho^U_V)$ as $\pi_V(p^+)$. Then, using \cite[Defn. 1.1.(4,7)]{BehrstockHierarchically19}, we have 
            \begin{align*}
                &\dist_U(x_0, \pi_U(p^+)) \leq \dist_U(x_0,\rho^V_U(\pi_V(x_0)) + \dist_U(\rho^V_U(\pi_V(x_0)), \pi_U(p^+)) + E \leq 3E; \\
                &\dist_U(x_0,\rho^V_U(z_V)) \leq \dist_U(x_0,\rho^V_U(\pi_V(x_0))) + \dist_U(\rho^V_U(\pi_V(x_0)), \rho^V_U(z_V)) + E \leq 3E.
            \end{align*}
            We therefore have $\dist_U(x_0,z_U) \leq 3E$, so $\dist_U(z_U,\rho^V_U(z_V)) \leq 10E$.

			Finally, suppose $V \propnest W$. There exist uniform neighbourhoods $B_U, B_V$ of the coarse projections of $\rho^U_W, \rho^V_W$ onto the uniform quasiline $\gamma_W$ such that $B_U$ and $B_V$ each disconnect $\mathcal N_\delta(\gamma_W)$, where $\delta \geq 0$ is a uniform constant (depending only on $E$ and $\theta$) such that $\mathcal N_\delta(\gamma_W)$ is connected. Let $b^-, b^+ \in X$ be such that $\pi_W(b^-), \pi_W(b^+)$ are  points in the connected components of $\mathcal N_\delta(\gamma_W) - (B_U \cup B_V)$ containing $\pi_W(p^-)$ and $\pi_W(p^+)$ respectively. Using the definition of $z_U$, hyperbolicity of $\mathcal CV$ and the bounded geodesic image axiom, we find that $\pi_U(b^\pm) \in \mathcal N_E(\pi_U(p^\pm))$ and $\pi_V(b^\pm) \in \mathcal N_E(\pi_V(p^\pm))$. Applying \cite[Lem. 2.6]{BehrstockHierarchically19} to the triple $b^+, b^-, x$ shows that $z_U, z_V$ satisfy the $\kappa'$-consistency condition, where $\kappa'$ depends only on the HHS parameters.
		\end{proofofclaim}
		
		Let $r_0$ be the constant from the realisation theorem (Theorem~\ref{thm:hhs_realisation}) and let $k \coloneqq \kappa r_0 + k_1$.
		Then the realisation theorem together with Claim \ref{claim: consistent z} imply that $\overline P \coloneqq \overline P_k$ and $X(A) \coloneqq X_k(A)$ are nonempty, and $\overline P$ is uniformly median quasi-convex. Claim \ref{claim: almost invariant hull (non-orthogonal)} implies that $\overline P$ is $A$--invariant. Thus item \eqref{it:P(A) nonempty and HQC} holds. 
		
		Using Proposition \ref{prop:bounded-orbits}, there is a constant $k' \geq k$ such that :
		\begin{itemize}
			\item $d_{Haus}(c(X(A)), \widecheck Z)\leq k'$ and $\dist_{Haus}(c(X(A)),Z(A,\newcone P,k)) \leq k'$;
			\item there is a non-empty injective space $Y$ equipped with the trivial $A$--action and a $(k',k')$--quasi-isometry $q: Z(A,\newcone P,k) \rightarrow Y$;
			\item the sets $\widecheck Z$ and $Z(A,\newcone P,k)$ are $(k',k')$-EHQC and $k'$-coarse median subalgebras of $\newcone P$.
		\end{itemize}
		It follows that, for some $L_1 \geq k'$ depending only on $k'$ and $E$, the subset $c(X(A))$ is $(L_1,L_1)$-EHQC and a $L_1$-coarse median subalgebra of $\newcone P$, and the restriction $\Phi$ of $q$ to $c(X(A))$ is an $(L_1,L_1)$-quasi-isometry $c(X(A)) \rightarrow Y$. For all $U \in \mathfrak{S-U}$ the projection $\pi_U(X(A))$ is uniformly quasiconvex in $\mathcal CU$, so it follows that $X(A)$ is also a uniformly coarse median subalgebra of $P$. Thus items \eqref{it:X(A) nonempty and a median subalgebra} and \eqref{it:choice-of-z} hold.
		
		\medskip
		
		\noindent\textbf{Almost invariant hulls.}
		Let $M' \coloneqq k + \theta$ and let $z \in X(A)$.
		The hull $H_{M',z}(p^-,p^+)$ consists of all points $x \in X$ such that $d_U(x,z) \leq M'$ if $U \in \support^\perp(p^\pm)$ and $d_U(x,\gamma_U) \leq M'$ if $U \in \mathfrak S - \support^\perp(p^\pm)$. In particular, $z \in H_{M',z}(p^-,p^+)$, so $H_{M',z}(p^-,p^+) \neq \emptyset$.
		
		Let $x \in H_{M',z}(p^-,p^+)$ and $U \in \mathfrak S$. If $U \in \support^\perp(p^\pm)$ then $a^{-1}U \in \support^\perp(p^\pm)$ and 
		\[
		d_{U}(ax,z) = d_{a^{-1}U}(x, a^{-1}z) \leq d_{a^{-1}U}(x,z) + d_{a^{-1}U}(z, a^{-1}z) \leq M' + k.
		\]
		If $U \notin \support^\perp(p^\pm)$ then $d_U(ax, \gamma_U) = d_{a^{-1}U}(x, \gamma_{a^{-1}U}) \leq M'$. Taking $M \coloneqq M' + k$, this proves that $A\cdot H_{M',z}(p^-,p^+)\subseteq H_{M,z}(p^-,p^+)$. Thus \eqref{it: almost an invariant convex hull} holds.

		Fix $z_0 \in X(A)$ and let $H \coloneqq A \cdot H_{M',z_0}(p^-,p^+)$.  Since $H_{M',z_0}(p^-,p^+) \subseteq H \subseteq H_{M,z_0}(p^-,p^+)$, Lemma \ref{lem: hulls are quasiconvex} implies that $H$ is uniformly HQC, and Lemma \ref{prop:RST} implies that $H$ is uniformly median-quasiconvex.  Hence it admits a gate map and, since the $A$--action on $X$ is free, we can assume that the gate map $\gate_H: P\to H$ is $A$--equivariant by Lemma \ref{lem:equivariant-gate}. The gate map  $\gate_H$ is uniformly coarsely lipschitz and uniformly quasimedian, by Construction \ref{cons:gate}. Recall that there is an HHS structure $(H, \mathfrak S_H)$ on $H$ with parameters depending only on the HHS parameters of $(P, \mathfrak S)$ and the quasiconvexity function of $H$ (see Lemma \ref{lem: HHS structure on HQC subspaces} and Remark \ref{rem:HQC-HHS}), which in turn depends only on the HHS parameters of $(P, \mathfrak S)$. By Theorem \ref{thm: median hulls}, there is a complete, connected, finite rank, proper median space $Q$, equipped with an isometric action of $A$ and a $A$-equivariant, uniformly quasi-median, uniform quality quasi-isometry $\Psi: H \rightarrow Q$. This proves item \eqref{it: properties of H}.

		\medskip
		
		\noindent\textbf{Quasi-isometry to a product.}

		\begin{claim}\label{claim:product-H-terms}
			There exists a constant $T$, depending only on the HHS parameters of $(X,\mathfrak S)$, such that for all $T'\geq T$ there exists $L_H$, depending only on the HHS parameters of $(P,\mathfrak S)$ and $T'$, such that
			\[\sum_{U\in\mathfrak U}\ignore{\dist_U(x_1,x_2)}{T'}\leq L_H\dist_H(\gate_H(x_1),\gate_H(x_2))+L_H\]
			for all $x_1,x_2\in \overline P$. 
		\end{claim}
		
		\begin{proofofclaim}{\ref{claim:product-H-terms}}
			Since the hierarchical quasiconvexity parameters for $H$ are uniform, there is a uniform constant $B_0$ such that for all $U\in\mathfrak U$ and $x \in P$, the projection $\pi_U(\gate_H(x))$ is $B_0$--close to the image of $\pi_U(x)$ under the coarse closest-point projection to the uniformly quasiconvex set $\pi_U(H)\subseteq \mathcal CU$ (see Construction \ref{cons:gate}).  
			
			If $U\in\mathfrak S-\mathfrak U$, then $\pi_U(H)$ has diameter bounded by a uniform constant $T_0$.  Indeed, if $U\transverse U_i$ or $U_i\propnest U$ for some $U_i$, then $\diam(\mathcal CU) \leq 30E+2\theta$ (since we are working in $P$, whose HHS structure has been normalised).  If $U\orth U_i$ for all $i$, then $\pi_U(H)$ uniformly coarsely coincides with $\pi_U(z_0)$, so it is bounded.  Hence $\dist_U(\gate_H(x_1), \gate_H(x_2))\leq \diam(\pi_U(H))+2B_0\leq T_0+2B_0$ for all such $U$.
			
			On the other hand, if $U\in\mathfrak U$ (i.e. $U\nest V$ for some $V \in \support(p^\pm)$) and $x \in \overline P$, then $\dist_U(\pi_U(x), \pi_U(H))$ is uniformly bounded, so $\pi_U(x)$ is uniformly close to $\pi_U(\gate_H(x))$. Thus, up to uniformly enlarging $B_0$, we have $\dist_U(x,\gate_H(x))\leq B_0$ for such $U$ and $x$.  Hence 
			\[
				\dist_U(x_1,x_2)-2B_0\leq \dist_U(\gate_H(x_1),\gate_H(x_2))\leq \dist_U(x_1,x_2)+2B_0
			\]
			for $U\in\mathfrak U$.
			
			By definition, $\dist_H(\gate_H(x_1),\gate_H(x_2))=\dist_P(\gate_H(x_1),\gate_H(x_2))$.  By the strong distance formula (Theorem \ref{thm:distance-formula}), applied to $(H,\mathfrak S_H)$, there is a uniform $T\geq T_0+2B_0$ such that for all $T'\geq T$, we have $L_H'$ (depending uniformly on $T'$) such that
			\begin{align*}
				\sum_{U\in\mathfrak S}\ignore{\dist_U(\gate_H(x_1),\gate_H(x_2))+2B_0}{T'} 
				&= \sum_{U\in\mathfrak U}\ignore{\dist_U(\gate_H(x_1),\gate_H(x_2))+2B_0}{T'} \\
				&\leq L_H' \dist_H(\gate_H(x_1),\gate_H(x_2))+L_H',
			\end{align*}
			where we used the description of $\mathfrak S_H$ from Remark \ref{rem:HQC-HHS} and the above bound on the distance formula terms corresponding to $U\in\mathfrak S-\mathfrak U$.  The claim follows since $\dist_U(x_1,x_2)\leq \dist_U(\gate_H(x_1),\gate_H(x_2))+2B_0$ for $U\in\mathfrak U$.
		\end{proofofclaim}

		\begin{claim}\label{claim:product-cone-terms}
			There exists $T''$, depending only on the HHS parameters of $(P,\mathfrak S)$, such that for all $T'\geq T''$, there exists $\widecheck{L}$, depending only on the HHS parameters of $(P,\mathfrak S)$ and $T'$, such that 
			\[\sum_{U\in\mathfrak S-\mathfrak U}\ignore{\dist_U(x_1,x_2)}{T'}\leq \widecheck L\dist_{\newcone P}(c(x_1),c(x_2))+\widecheck{L}\]
			for all $x_1,x_2\in P$.
		\end{claim}
		
		\begin{proofofclaim}{\ref{claim:product-cone-terms}}
			The pair $(\newcone P,\mathfrak S-\mathfrak U)$ is an HHS with uniform parameters, by Corollary \ref{cor:uniform-cone-off}, with the projections being specified in the proof of said corollary. Theorem \ref{thm:distance-formula}, applied to this HHS structure, yields the claim.
		\end{proofofclaim}
		
		Both $\gate_H$ and $c$ are uniformly coarsely lipschitz, coarsely surjective and quasi-median maps.
		To get the other bound needed to show that $c|_{\overline P} \times \gate_H|_{\overline P}$ is a quasi-isometry, let $T,T''$ be the uniform constants from Claims \ref{claim:product-H-terms} and \ref{claim:product-cone-terms}, let $T'\geq\max\{T,T''\}$, let $L_H,\widecheck{L}$ be the constants provided by those claims, for the given $T'$, and let $L \coloneqq \max\{L_H,\widecheck{L}\}$.  We can moreover assume that $T'$ is sufficiently large to be a threshold in the distance formula for $(P,\mathfrak S)$ (with the affine function in Theorem \ref{thm:distance-formula} taken to be the identity).  Then Theorem \ref{thm:distance-formula} yields a uniform constant $L'$ such that, for all $x_1,x_2\in \overline P$,
		\begin{eqnarray*}
			\dist_X(x_1,x_2)&\leq& L'\sum_{U\in\mathfrak S}\ignore{\dist_U(x_1,x_2)}{T'}+L'\\
			&= &L'\left[\sum_{U\in\mathfrak S-\mathfrak U}\ignore{\dist_U(x_1,x_2)}{T'}+\sum_{U\in\mathfrak U}\ignore{\dist_U(x_1,x_2)}{T'}\right]+L'\\
			&\leq& L'\left[L\left(\dist_{\newcone P}(c(x_1),c(x_2))+\dist_H(\gate_H(x_1),\gate_H(x_2)\right)+2L\right]+L'.
		\end{eqnarray*}
		
		Thus $c|_{\overline P} \times \gate_H|_{\overline P}$ is a quasi-median quasi-isometry with constants depending only on the HHS parameters.
		Now let $\xi \coloneqq \Phi \circ c|_{X(A)}$ and $\eta \coloneqq \Psi \circ \gate_H|_{X(A)}$. The maps $\Phi$, $\Psi$ and $c|_{\overline P} \times \gate_H|_{\overline P}$ are all $A$--equivariant uniform quality quasi-isometries, so $\xi \times \eta$ is an $A$--equivariant uniform quality quasi-isometry. This proves items \eqref{it:barP(A) is a quasi-product} and \eqref{it:X(A) is a quasi-product}. \\

		Since $c(X(A))$ is uniformly EHQC in $\newcone P$ and $H$ is uniformly HQC in $P$, it also follows from this that $X(A)$ is uniformly EHQC. To see that $X(A)$ is a coarse lipschitz retract, define a map $X \rightarrow X(A)$ by first applying the gate map to $\overline{P}$, then applying the quasi-isometry to $\newcone P \times H$, followed by the map $\rho \times \id_H$, where $\rho$ is the coarsely lipschitz retraction $\newcone P \rightarrow Z(A,\newcone P,k)$ given by Proposition \ref{prop:bounded-orbits}.\eqref{it: fixed point set coarse lipschitz retract}. By perturbing the resulting map uniformly so that it is the identity on $X(A)$, we obtain a uniformly coarse lipschitz retraction $X \rightarrow X(A)$.
        We have thus proved item \eqref{it: X(A) is EHQC and a coarse retract}.
	\end{proof}
	
	If $H \leq G$ is a pair of groups, we denote the commensurator of $H$ in $G$ by $\commensurator{G}{H}$: i.e. $\commensurator{G}{H}$ is the subgroup of $G$ consisting of all $g \in G$ such that $gHg^{-1} \cap H$ has finite index in both $H$ and $gHg^{-1}$.  In the next lemma, we continue to work with an HHS $(X,\mathfrak S)$ and a finitely generated virtually abelian group $A\leq \Isom(X)$ acting hierarchically semisimply by HHS automorphisms.
	
	\begin{lem} \label{lem:commensurator stabilises U etc.}
		Suppose $A \leq B$, and the action of $A$ extends to an action of $B$ on $X$ by HHS automorphisms. Then the following hold:
		\begin{enumerate}
			\item \label{item:commensurated-invariant-boundary} Suppose $\commensurator{B}{A}=B$.  Then $\support(p^\pm)$ is $B$--invariant. Thus $P \subseteq X$ is $B$-invariant and there is a finite-index subgroup $B'\leq B$ which fixes each $U\in\support(p^\pm)$ and fixes $\pi_U(p^\pm)\in \boundary\mathcal CU$ for each $U\in\support(p^\pm)$.  
			
			\item \label{item:commensurated-cone-off} Suppose $\commensurator{B}{A}=B$.  Then $\mathfrak U$ is $B$--invariant.  Hence the $B$--action on $(X,\mathfrak S)$ induces an action of $B$ on $(\newcone P,\mathfrak S-\mathfrak U)$ by HHS automorphisms making  $c:P\to\newcone P$ a $B$--equivariant map.
			
			\item \label{item:normaliser stabilises X(A)} If $A$ is normal in $B$, then $B$ stabilises $X(A)$.
		\end{enumerate}
        \end{lem}
		\begin{proof}
			First assume that $\commensurator{B}{A}=B$. Let $V\in\mathfrak S$ and $x\in X$.  Then $\pi_V(A\cdot x)$ is unbounded if and only if $V\in\support(p^+)$, by Lemma \ref{lem: fixed boundary points}.\eqref{it: unbounded projections}.  The same holds when $A$ is replaced by any finite-index subgroup.  Hence, for any $b\in B$, we have that $\pi_V(bAb^{-1}\cdot x)$ is unbounded if and only if $V\in \support(p^+)$, and also if and only if $V\in b\support(p^+)$, so $\support(p^+)$ is $B$--invariant.  Thus there is a finite-index $B''\leq B$ that fixes each of the finitely many $U\in\support(p^+)$.  
			
			Let $b\in B''$, so $b$ acts as an isometry of $\mathcal CU$ and hence a homeomorphism of $\boundary\mathcal CU$.  Let $A'\leq A$ be a finite index subgroup fixing all elements of $\support(p^\pm)$ and fixing $p^-$ and $p^+$.  Then the $A'$--action on $\boundary\mathcal CU$ has exactly two fixed points, namely $\pi_U(p^\pm)$, and the same is true of any finite-index subgroup. The fixed-point set of $bA'b^{-1}$ in $\boundary\mathcal CU$ is $\{b\pi_U(p^+),b\pi_U(p^-)\}$.  Since $bA'b^{-1}\cap A'$ has finite index in $A'$, we therefore have that $b\pi_U(p^\pm)\in \{\pi_U(p^+),\pi_U(p^-)\}$.  So, up to passing to a further finite-index subgroup, $B''$ fixes $p^-$ and $p^+$, proving \eqref{item:commensurated-invariant-boundary}.  Since $\support(p^\pm)$ is $B$--invariant, and $B$ preserves the $\nest,\orth,\transverse$ relations in $\mathfrak S$, the set $\mathfrak U$ is $B$--invariant, which implies item \eqref{item:commensurated-cone-off}, using Corollary \ref{cor:uniform-cone-off}.  This argument also shows that the $B$--action on $\prod_{V\in\mathfrak U}\mathcal CV$ preserves the subset $\prod_{V\in\mathfrak U}\gamma_V$.
			
			Now assume $A\unlhd B$, so items \eqref{item:commensurated-invariant-boundary} and \eqref{item:commensurated-cone-off} apply and $B$ stabilises $\widecheck Z$.
			Together with the $B$--invariance of $\prod_{V\in\mathfrak U}\gamma_V$, this shows that $B$ stabilises $X(A)$. 
        \end{proof}

	\newsection{Hierarchical flat torus theorem}\label{sec:main-theorem}
	Now we can state and prove the main theorem:
	
	\begin{thm}\label{thm:HHS-semisimple-coarse-minset}
		Let $(X,\mathfrak S)$ be an HHS of complexity $\hhscomp$.  Then there exist constants $p,q$ and $r$, depending only on the HHS parameters, such that the following holds.  Let $A\leq\Isom(X)$ be a finitely-generated, infinite, virtually abelian group acting properly and hierarchically semisimply by HHS automorphisms on $(X,\mathfrak S)$.	Then:
		\begin{enumerate}
			\item \label{item:main-rank} Let $n$ be the maximal rank of free abelian subgroups of $A$.  Then $n\leq \hhscomp$.
			
			\item \label{item:main-one-cocompact-quasiflat-1} There exists a $(p,q)$--EHQC, $A$--invariant subspace $F\subseteq X$, a proper, cobounded  action $\alpha:A\to\Isom(\Euclidean^n)$, and an $A$--equivariant $(p,p)$--quasi-isometry $F\to \Euclidean^n$, where $F$ has the subspace metric from $X$. 
			
			\item \label{item:main-full-minset}  There is an injective uniformly coarse median metric space $(Y,\dist_Y)$ and a CAT(0) space $(C,\dist_C)$ with the following properties.  Let $A$ act on $Y\times C\times \Euclidean^n$, trivially on the first and second factors and via $\alpha$ on the third.  Then there is a $(p,q)$--EHQC, $A$--invariant $r$--coarse lipschitz retract $\mathcal M_A \subseteq X$ and an $A$--equivariant $(p,p)$--quasi-isometry 
			\[
				\Theta:\mathcal M_A \to Y\times C\times  \Euclidean^n,
			\]
			where $\mathcal M_A$ has the metric inherited from $X$.  Moreover, for each $(y,c)\in Y\times C$, the $A$--invariant subset $$F_{(y,c)}:=\Theta^{-1}(\neb_k((y,c))\times \Euclidean^n)$$ is $(p,q)$--EHQC and an $r$--coarse lipschitz retract.
			
			\item \label{it: M is coarsely X(A)} There exists $r_A$, depending on $A$, and a free abelian subgroup $A_0 \unlhd A$ such that $\dist_{Haus}(\mathcal M_{A_0},X(A_0))\leq r_A$.  Hence there exists $\lambda\geq 0$, depending on $A$ and the HHS parameters, such that $\mathcal M_{A_0}$ is a $\lambda$--coarse median subalgebra of $X$.
			
			\item \label{it:coarse min set}
			There is a function $\omega:[1,\infty) \to [0,\infty)$, depending on $A$ and the HHS parameters, such that for all $t \in [0,\infty)$, if $K\subseteq H_\theta(X(A))$ is $A$--invariant and $(t,t)$--quasi-isometric to $\mathbb E^n$, then $K\subseteq \neb_{\omega(t)}(F_{(y,c)})$ for some $(y,c) \in Y \times C$.
		\end{enumerate}
	\end{thm}
	
	\begin{proof}
		By Lemma \ref{lem:assume-graph}, Remark \ref{rem:HHS-are-graphs}, and Lemma \ref{lem:graph-free}, we can and shall assume that $X$ is a graph (and $\dist$ is the usual graph metric) and $A$ acts on $X$ freely, so Proposition \ref{prop: almost invariant hull} applies.

		Let $A_0 \unlhd A$ be the intersection of a normal finite index free abelian subgroup of $A$ with every subgroup of $A$ with index $\leq \hhscomp!$. Since $A$ is finitely generated, this subgroup has finite index in $A$.
		
		Let $p^+,p^- \in \partial X$ be the boundary points provided by Lemma \ref{lem: fixed boundary points}. 
		Recall the notation from Definitions \ref{defn:P(A)} and \ref{defn:U,P and X(A)}.
		Let $k \geq k_1$, $M \geq M' \geq \theta$ and $L_1, L_2, L_3, L_4 \geq 1$ be the uniform constants from Proposition \ref{prop: almost invariant hull}. Fix $z_0 \in X(A)$, let $H \coloneqq A\cdot H_{M',z_0}(p^-,p^+)$, and let $Y,Q$ be the spaces and $\Phi, \gate_H, \Psi, \xi, \eta$ be the maps produced by Proposition \ref{prop: almost invariant hull}.

		The action of $A$ on $X$, and therefore on $H \subseteq X$ is proper so, since $\Psi$ is a quasi-isometry, the $A$--action on $(Q,d_Q)$ is proper. The median space $Q$ is obtained by applying Theorem \ref{thm: median hulls}, so Proposition \ref{prop:short-mountains} applies: there exist median quasi-lines $Q_1, \dots, Q_r$ such that $A$ acts isometrically on $\prod_{i=1}^r Q_i$ and there is an $A$-equivariant isometry $Q \rightarrow \prod_{i=1}^r Q_i$. Since the rank of $Q$ is at most $\hhscomp$ by Lemma \ref{lem: Q is a metric space}.\eqref{it: median metric and rank of Q}, we have $r \leq \hhscomp$. Thus $A_0$ acts trivially on the set of factors of $\prod_{i=1}^r Q_i$. Proposition \ref{prop:short-mountains} furthermore implies that the action of $A_0$ on each $Q_i$ is cocompact. Since $A_0$ is free abelian, this also implies that $A_0$ fixes the endpoints of each $Q_i$. 
		Identify $Q$ with $\prod_{i=1}^r Q_i$. \\

		\noindent\textbf{CAT(0) minset in $Q$ and proof of \eqref{item:main-rank}.} 
		Each $Q_i$ is a finite rank, connected complete median space so \cite{BowditchProperties16} provides a complete CAT(0) metric $\sigma_i$ on $Q_i$, bilipschitz equivalent to $d_{Q_i}$ (with uniform constant), such that any isometry of $(Q_i, d_{Q_i})$ is an isometry of $(Q_i,\sigma_i)$. Let $\sigma$ be the $\ell^2$ product metric on $Q$, where each factor $Q_i$ is equipped with $\sigma_i$. Then $\sigma$ is a complete CAT(0) metric on $Q$, the action $A \curvearrowright Q$ is by $\sigma$-isometries, and the identity map $(Q,d_Q) \rightarrow (Q,\sigma)$ is bilipschitz with uniform constant (in fact, one can take $\sqrt{\hhscomp}$ as the bilipschitz constant). In particular, the $A$--action on $(Q,\sigma)$ is also proper.
		
		Since the $A$--action on $X$ is hierarchically semisimple, each $\langle a\rangle\leq A$ acting with unbounded orbits has the property that $a$ has positive stable translation length on $X$ (combine Definition \ref{defn:hierarchically-semisimple} with the fact that $\pi_U$ is coarsely lipschitz for each $U\in\mathfrak S$), from which it follows that the $A$--action on $(Q,\sigma)$ is semisimple.

		Hence \cite[II.7.1]{BridsonHaefliger:metric} implies that the minset 
		\[
		\Min_Q(A_0) \coloneqq \{x \in Q : \sigma(x,ax) = \inf_{y \in Q} \sigma(y,ay) \text{ for all } a \in A_0\}
		\]
		is nonempty, convex, $A$--invariant and isometric to a product $D \times \mathbb E^n$, where $n$ is the rank of $A_0$. The action of $A$ on $D \times \mathbb E^n$ preserves the product decomposition, with $A$ acting properly and cocompactly on $\mathbb E^n$, and elliptically on $D$, since the finite index subgroup $A_0$ acts trivially on $D$. We henceforth identify $\Min_Q(A_0)$ with $D \times \mathbb E^n$ in an $A_0$--equivariant way.

		For any $x \in \mathbb E^n$, the subspace $D \times \{x\} \subseteq Q$ is closed and convex and -- when equipped with the metric induced by $\sigma$ -- isometric to $D$. Therefore $D$ is CAT(0). Let $C \subseteq D$ be the fixed point set of the $A$--action on $D$. In particular, $C$ is a nonempty, closed, convex subset of $D$, so it is also CAT(0).
		
		Since $Q$ has rank at most $\hhscomp$ by Lemma \ref{lem: Q is a metric space}, any locally compact subset of $Q$ has topological dimension at most $\hhscomp$ by \cite[Lem. 4.1]{Bowditch:large-scale-MCG}. So, since $\mathbb E^n$ is bilipschitz--embedded in $(Q,d_Q)$, we get $n\leq \hhscomp$, proving \eqref{item:main-rank}.\\
		
		\noindent\textbf{Description of $\mathcal M_A$ and $\Theta$ and proof of \eqref{it: M is coarsely X(A)}.} Combining Proposition \ref{prop: almost invariant hull}.\eqref{it:X(A) is a quasi-product} with the fact that $(Q,d_Q)$ and $(Q,\sigma)$ are $\sqrt \hhscomp$--bilipschitz equivalent, the product map $\xi \times \eta: X(A)\to Y\times (Q,\sigma)$ is an $(R,R)$--quasi-isometry for some uniform $R$. We can assume $R$ is large enough that $\eta$ is $R$--coarsely surjective with respect to $\sigma$.
		Let 
		\[
			\mathcal M_A\coloneqq\{x\in X(A):\sigma(\eta(x), C \times \mathbb E)\leq R\}.
		\]
		Since the closest point projection in $(Q,\sigma)$ from $\neb_R(C\times \mathbb E)$ to the CAT(0) convex set $C\times \mathbb E$ is an $A$--equivariant uniform quasi-isometry (it moves each point a uniformly bounded distance), we can perturb $\xi \times \eta$ to an equivariant, $p$--coarsely surjective, $(p,p)$--quasi-isometry $\Theta: \mathcal M_A\to Y\times C\times\Euclidean^n$, where $p$ is uniform. 
		
		Let $r_0\geq 0$ be chosen (depending on $A_0$ and the HHS parameters) so that for all $i$, the space $Q_i$ is contained in the $r_0$--neighbourhood of any $A_0$--orbit in $Q_i$.  Hence $\sigma_i(\Min_{Q_i}(A_0),x)\leq r_0$ for all $x \in Q_i$.  Therefore, since, by the definition of $\sigma$, we have $\Min_Q(A_0) = \prod_{i=1}^r \Min_{Q_i}(A_0)$, there exists $r_1\geq 0$, depending only on $A_0$ and the HHS parameters, such that all of $Q$ is contained in the $r_1$--neighbourhood of $\Min_Q(A_0)$.  Let $x \in X(A)$.  Choose $x'\in \Min_Q(A_0)$ such that $\sigma(x',\eta(x))\leq r_1$ and let $y \coloneqq \xi(x)$.
		Then choose $m \in \mathcal M_{A_0}$ such that $\xi \times \eta(m)$ is $p$--close in $Y \times \Min_Q(A_0)$ to $(y, x')$.  Then 
		\begin{align*}
			d_X(m,x)
			&\leq R(d_{Y\times Q}(\xi \times \eta(m),\xi \times \eta(x)) + R) \\
			&\leq R(k + r_1 + R).
		\end{align*}
		This proves \eqref{it: M is coarsely X(A)}.

        \medskip

		\noindent\textbf{Proof of \eqref{item:main-one-cocompact-quasiflat-1}-\eqref{item:main-full-minset}.} 
		Let us show $\mathcal M_A$ is uniformly EHQC.
		Let $x,x'\in \mathcal M_A$ and let $z,z' \in C \times \mathbb E$ be such that $\Theta(x) = (c(x),z)$ and $\Theta(x') = (c(x'),z')$.  Let $\check\alpha$ be a uniform-quality hierarchy path in $\newcone P$ which joins $c(x),c(x')$ and lies in $c(X(A))$.  As in the proof of Proposition \ref{prop:bounded-orbits}, the $\sigma$--geodesic $\beta$ in $C\times \mathbb E$ joining $z$ to $z'$ is a geodesic with respect to $d_Q$. 		
		Let $\alpha_1$ be the path in $Y\times Q$ given by $\alpha_1(t)=(\check\alpha(t),z)$ for all $t$, and let $\beta_1$ be the path in $Y\times Q$ given by $\beta_1(t)=(c(x'),\beta(t))$ for all $t$.  Let $\gamma_1$ be the concatenation of $\alpha_1$ and $\beta_1$.  Note that $\alpha_1$ and $\beta_1$ lie in $Y\times C\times\mathbb E$.  Hence, for each $t$, there exists $\gamma(t)\in \mathcal M_A$ such that $d_{Y\times Q}(\Theta(\gamma(t)),\gamma_1(t)) \leq p$. The path $t \mapsto \gamma(t)$ is a uniform quality quasigeodesic in $X(A)$ and, after a uniform perturbation, joins $x$ to $x'$.  	
		The map $\eta$ is uniformly quasimedian by Proposition \ref{prop: almost invariant hull}, and $c$ is uniformly quasimedian by Corollary \ref{cor:uniform-cone-off}.  This implies that $\gamma$ is a uniform quality hierarchy path, since $\alpha$ is a uniform quality hierarchy path in $\newcone X$ and $\beta$ is a geodesic in the median metric $d_Q$.  By construction, $\gamma$ is a path in $\mathcal M_A$.

        To see that $\mathcal M_A$ is a coarse lipschitz retract define a map $X \rightarrow \mathcal M_A$ by composing the uniformly coarse lipshcitz retract $X \rightarrow X(A)$ given by Proposition \ref{prop: almost invariant hull}.\eqref{it:X(A) is a quasi-product} with the quasi-isometry $\xi \times \eta: X(A) \rightarrow Y \times (Q,\sigma)$, followed by the map $\id \times \pi_{C \times \mathbb E^n}: Y \times Q \rightarrow Y \times C \times \mathbb E^n$, where $\id$ is the identity on $Y$ and $\pi_{C \times \mathbb E^n}$ is the closest point projection from $Q$ onto $C \times \mathbb E^n$ (which is well-defined and 1-lipschitz because $C \times \mathbb E^n$ is convex and $(Q,\sigma)$ is CAT(0)). Finally, compose with a quasi-inverse of $\xi \times \eta$. The resulting map is uniformly coarsely lipschitz since it is a composition of uniformly coarsely lipschitz maps, and it can be uniformly perturbed so that its image is contained in $\mathcal M_A$ and it is the identity on $\mathcal M_A$.
		
		Now fix $c\in C$ and $y\in Y$.  Let $F_{(y,c)}$ be the set of $m\in\mathcal M_A$ such that $\Theta(m)$ is $p$--close to the subset $\{(y,c)\}\times \mathbb E^n$ of $Y\times Q$.  Since $\Theta$ is $A$--equivariant and $A$ acts trivially on $Y\times C$, the set $F_{(y,c)}$ is $A$--invariant, and the choice of $p$ ensures that $\Theta$ restricts to a uniformly coarsely surjective, uniform quasi-isometry from $F_{(y,c)}$ to $B\times \mathbb E^n$, where $B$ is a uniformly bounded set where $A$ acts trivially.  So composing with the natural projection $B\times \mathbb E^n\to\mathbb E^n$ gives a uniform quasi-isometry $\tau:F_{(y,c)}\to \mathbb E^n$.
		
		Let $\alpha:A\to\Isom(\mathbb E^n)$ be induced by projecting the action on $Y\times C\times \mathbb E^n$ to the third factor. Then $\tau$ is $A$--equivariant. 
        
		To complete the proof of \eqref{item:main-one-cocompact-quasiflat-1} and \eqref{item:main-full-minset}, we therefore just have to observe that $F_{(y,c)}$ is uniformly EHQC and a uniformly coarse lipschitz retract, by essentially the same argument as was used for $\mathcal M_A$.  Indeed, if $x,x'\in F_{(y,c)}$, then $c(x),c(x')$ are uniformly close, so the hierarchy path $\gamma$ produced above stays, by construction, in a uniformly bounded neighbourhood of $F_{(y,c)}$. Moreover, the construction of a coarse lipschitz map $X \rightarrow \mathcal M_A$ above can be reproduced using the closest point projection $Q \rightarrow \{c\} \times \mathbb E^n$ rather than $Q \rightarrow C \times \mathbb E^n$ (since the former is also convex in $(Q,\sigma)$) resulting, up to a uniform perturbation, in a uniform quality coarsely lipschitz retraction $X \rightarrow F_{(y,c)}$. \\

        \noindent\textbf{Proof of \eqref{it:coarse min set}.} Let $\overline H_\theta(X(A))$ be the $\theta$-hull of $X(A)$ in $\overline P$.  It follows from Proposition \ref{prop: almost invariant hull} that there exists $L_3' \geq 1$, depending only on $L_3$, $M$ and the HHS parameters, such that $\overline H_\theta(X(A))$ is $(L_3',L_3')$--quasi-isometric to $H_\theta(\widecheck Z) \times H$. The Hausdorff distance between $\overline H_\theta(X(A))$ and $H_\theta(X(A))$ is uniformly bounded, so there exists $L$ depending only on the HHS parameters such that $H_\theta(X(A))$ is $(L,L)$-quasi-isometric to $H_\theta(\widecheck Z) \times H$. 

        Given $t\geq 1$, let $\mathcal K_t$ be the set of $A$--invariant subspaces $K\subseteq H_\theta(X(A))$ such that there is a $(t,t)$--quasi-isometry $K\to \Euclidean^n$.  Let $$\omega(t)=\sup\{\sup_{x_0\in K}\dist(x_0,\mathcal M_A):K\in\mathcal K_t\}.$$
        Note that $\omega$ is a non-decreasing function of $t$ since $\mathcal K_t\subseteq \mathcal K_{t'}$ for $t\leq t'$, so we need to show that $\omega$ takes finite values and is unbounded as $t\to\infty$.
			
		Let $s_0 \in [1,\infty)$ and suppose $K \subseteq H_\theta(X(A))$ is $A$--invariant and there exists $x_0 \in K$ such that $\dist(x_0,\mathcal M_A) > s_0$. Then either
			\begin{enumerate}
				\item\label{it:far-in-H} $\dist_H(\gate_H(x_0), \gate_H(\mathcal M_A)) > s_1$ or
				\item \label{it:far-in-cone}$\dist_{\widecheck X}(c(x_0), \widecheck Z) > s_1$,
			\end{enumerate}
			where $s_1 \coloneqq \frac{s_0-L}{2L}$.
            
            By \eqref{it: M is coarsely X(A)}, there is a point $x' \in \mathcal M_{A_0}$ such that $\dist_X(\gate_H(x_0), x') \leq r_A$ and, by definition, there exists $z_0 = (d_0,e_0) \in \Min_Q(A_0)=D\times \Euclidean^n$ such that $\sigma(z_0,\Psi(x')) \leq R$.

            Suppose that case \eqref{it:far-in-H} holds.  Then there exists $\kappa\geq 1$, depending on $L_2,r_A,R,\hhscomp$ such that $\sigma(d_0,C)\geq s_1/\kappa-\kappa$.  Hence Lemma \ref{lem:big rotations} below implies that there exists $d_1 \in D$ such that $d_1 \in A \cdot d_0$ and $\dist_D(d_0, d_1) \geq s_1/\kappa-\kappa$. In this case, let $a_1 \in A$ be such that $a_1d_0 = d_1$ and let $x_1 \coloneqq a_1x_0$. 

            Now assume that case \eqref{it:far-in-cone} holds.  Let $\omega_0$ be the increasing function from Proposition \ref{prop:bounded-orbits}, applied to the action of $A$ on $\widecheck X$.    Then $\dist_{\widecheck X}(c(x_0),\widecheck{Z})>s_1$, so $\diam(A\cdot c(x_0))> \omega_0(s_1)$, by Proposition \ref{prop:bounded-orbits}.\eqref{it:zs-are-close}, and thus there is a point $y \in c(K)$ such that $\dist_{\check X}(c(x_0), y) \geq \omega_0(s_1)$. In this case, let $x_1 \in K$ be such that $c(x_1) = y$.
			
			In either case, let $\Delta_i \coloneqq A_0 \cdot x_i \subseteq K$ for $i\in\{0,1\}$, which is a subspace of $H_\theta(X(A))$.  Hence $\Delta_i$ is $(L,L)$-quasi-isometric to $c(\Delta_i) \times \gate_H(\Delta_i)$. Indeed, $c(\Delta_i)$ has uniformly bounded diameter by the choice of $x_i$, so the claimed quasi-isometry holds up to uniformly increasing $L$. By Proposition \ref{prop:bounded-orbits}.\eqref{it:finite index hqc fixed point set}, $\diam(c(\Delta_i)) = \diam(A_0 \cdot c(x_i))$ is uniformly bounded, and by construction $\gate_H(\Delta_i)$ is $(L',L')$--quasi-isometric to $\mathbb E^n$, with constant $L'$ depending only on the HHS parameters and $A$.  Hence there exists $L''\geq 1$, depending only on the HHS parameters and $A$, and $(L'',L'')$-quasi-isometries $\Xi_i:\Euclidean^n\to \Delta_i$.

            Suppose that $K\in \mathcal K_t$ for some $t\geq 1$, and suppose that $\zeta:K\to \Euclidean^n$ is a $(t,t)$--quasi-isometry.  Then $\zeta\circ \Xi_i:\Euclidean^n\to\Euclidean^n$ is a $(t(L''+1),t(L''+1))$--quasi-isometric embedding for $i\in\{0,1\}$.  Hence, by Lemma \ref{lem:Borsuk-Ulam-redux}, each of the two maps $\zeta\circ\Xi_i$ is $\BU\left(t(L''+1)\right)$--coarsely surjective.  Hence there exist $x'_i\in\Delta_i$ such that $|\zeta(x'_0)-\zeta(x'_1)|\leq \BU\left(t(L''+1)\right)$, and therefore $\dist_X(\Delta_0,\Delta_1)\leq L''(\BU\left(t(L''+1)\right)+L'')$.

            On the other hand, in either of the above cases, we have shown that $\dist_X(\Delta_0,\Delta_1)\geq \min\{\kappa's_0-\kappa',\kappa'\omega_0(s_0-L)/2L-\kappa'\}$, where $\kappa'$ depends only on the HHS parameters and $A$ (but not on $K$).  Hence
            $$\min\{\kappa's_0-\kappa',\kappa'\omega_0(s_0-L)/2L-\kappa'\}\leq L''(\BU\left(t(L''+1)\right)+L''),$$
            so since $\omega_0$ is an ubounded, increasing function, it follows that $s_0\leq\omega(t)$ where $\omega:[1,\infty)\to[1,\infty)$ is a function that can be chosen in terms of the HHS parameters and $A$ only.

            Thus far, we have shown $K\subseteq \neb_{\omega(t)}(\mathcal M_A)$ whenever $K\in\mathcal K_t$.  The above argument also showed that $c(K)$ has diameter in $\widecheck X$ bounded above by a function of $t$ that can be chosen in terms of the HHS parameters and $A$.  Similarly, the image of $K$ in $C$ is bounded uniformly in terms of $t$.  Therefore, $\Theta(K)$ has diameter bounded uniformly in terms of $t$, so, up to uniform perturbation of $\omega$, we get that $K\subseteq \neb_{\omega(t)}(F_{(y,c)})$ for some $y\in Y,c\in C$.
    \end{proof}

	We used the following standard lemmas in the above proof.
	
	\begin{lem} \label{lem:big rotations}
		Let $D$ be a CAT(0) space, let $G$ be a group acting elliptically on $D$ and let $\Fix(G) \subseteq D$ be the fixed point set of $G$. Then $\diam(G \cdot x) \geq \dist(x,\Fix(G))$ for all $x \in D$.
	\end{lem}
	
	\begin{proof}
		Fix $x\in D$.  Let $r$ be the infimal radius of balls containing the bounded set $G\cdot x$; note that $r\leq \diam(G\cdot x)$.  Then \cite[Prop. II.2.7]{BridsonHaefliger:metric} provides a unique circumcentre for $G\cdot x$, i.e. a point $c\in D$ such that $G\cdot x\subset \neb_r(c)$.  Since $g\cdot c$ has this property for all $g\in G$, uniqueness of $c$ implies that $c\in\Fix(G)$, so $\dist(x,\Fix(G))\leq \dist(x,c)\leq r\leq \diam(G\cdot x)$.
	\end{proof}

	\begin{lem}\label{lem:Borsuk-Ulam-redux}
		For each $n\in\naturals$, there is a function $\BU:[1,\infty)\to [0,\infty)$ such that the following holds. Let $f:\Euclidean^n\to\Euclidean^n$ be a $(C,C)$--coarsely lipschitz map.  Suppose that $\sup_{x\in\Euclidean^n}\diam(f^{-1}(\neb_r(x)))<\infty$ for all $r\geq 0$.  Then $f$ is $\BU(C)$--coarsely surjective.
	\end{lem}
	
	\begin{proof}
		By a standard triangulation argument, there is a continuous map $g:\Euclidean^n \to \Euclidean^n$ such that $\sup_{x\in\Euclidean^n}|f(x)-g(x)|\leq R$, and $g$ is $(K,K)$--coarsely lipschitz, for some $R,K$ depending only on $C$.  We will show that $g$ is surjective, from which it will follow that $f$ is $R$--coarsely surjective.
		
		Let $\bar E$ be the one-point compactification of $\mathbb E^n$.  Since $g$ is continuous and preimages of bounded sets are bounded, $g$ is proper map.  Hence $g$ extends to a continuous map $\bar g:\bar E\to\bar E$ with $\bar g(\infty)=\infty$.
		
		Define a homeomorphism $h:\bar E\to \mathbb S^n$ to be stereographic projection, where $\mathbb E^n$ is viewed as a coordinate plane in $\mathbb E^{n+1}$ and $\mathbb S^n$ is viewed as the sphere in $\mathbb E^{n+1}$ centred at the origin and with radius $R$, where $R=\sup_{x\in\Euclidean^n}\diam(g^{-1}(\neb_1(x)))<\infty$.  The ``north pole'' for the projection is in the direction normal to $\Euclidean^n$, and $\mathbb S^n$ is identified with $\bar E$ in such a way that the north pole is $\infty\in \bar E$.  
        
        This choice of $h$ has the following property: if $x,y\in \Euclidean^n$ and $g(x)=g(y)$, then $h(x)$ and $h(y)$ are not antipodal.
		
		If $g$ were not surjective, then $h\circ \bar g \circ h^{-1}:\mathbb S^n\to \mathbb S^n$ has range properly contained in $\mathbb S^n$, to the Borsuk-Ulam theorem provides antipodal points $h(x),h(y)\in\mathbb S^n$ such that $\bar g(x)=\bar g(y)$.  Since $\bar g^{-1}(\infty)$ has a single point, $x,y\in \Euclidean^n$ and $g(x)=g(y)$, a contradiction.
	\end{proof}

	A consequence of Theorem \ref{thm:HHS-semisimple-coarse-minset}.\eqref{item:main-full-minset} we will need is:
	
	\begin{cor}\label{cor:coarse-foliation}
		Let $(X,\mathfrak S)$ be an HHS on which the finitely generated virtually abelian group $A\leq\Isom(X)$ acts properly and hierarchically semisimply by HHS automorphisms.  Then there exist $k,r,s$, depending only on the HHS parameters, such that the following holds.  There is a proper left-invariant quasigeodesic metric $\dist_A$ on $A$ such that for all $x\in\mathcal M_A$, there is an $A$--equivariant $(s,s)$--quasi-isometric embedding $\phi_x:(A,\dist_A)\to \mathcal M_A$ such that $\phi_x(1)=x$.  Moreover, $\phi_x(A)$ is contained in the $(k,r)$--EHQC, $A$--invariant $(k,k)$--quasiflat $F$ from Theorem \ref{thm:HHS-semisimple-coarse-minset}.\eqref{item:main-full-minset}.
	\end{cor}
	
	\begin{proof}
		By Theorem \ref{thm:HHS-semisimple-coarse-minset}.\eqref{item:main-one-cocompact-quasiflat-1},\eqref{item:main-full-minset}, there is a fixed proper cocompact action $\alpha:A\to\Isom(\Euclidean^n)$ and a uniform constant $k$ such that for all $x\in\mathcal M_A$, there is an $A$--equivariant $(k,k)$--quasi-isometry $q:F_x\to \Euclidean^n$, where $F_x\subseteq \mathcal M_A$ is an $A$--invariant, $(k,r)$--EHQC subspace containing $x$.  Let $\Omega=Cay(A,A)$, which is a graph of diameter $1$ equipped with a free $A$--action.  Let $\mathbf E=\Euclidean^n\times\Omega$, so that the diagonal $A$--action on $\mathbf E$ is free, metrically proper, and cobounded.  Let $\dist_A$ be the metric on $A$ obtained by pulling back the subspace metric on the orbit $A\cdot (\vec 0,1_A)\subset \mathbf E$.  Note that $\dist_A$ is uniformly quasi-isometric to the pseudometric coming from the orbit map $A\to \Euclidean^n$ given by $a\mapsto \alpha(a)\cdot\vec 0$.
             
        We can assume that $q(x)=\vec 0$.  Define $\phi_x:(A,\dist_A)\to (F_x,\dist_X)$ by $\phi_x(a)=ax$.  Then $q\circ\phi_x(a)=\alpha(a)\cdot \vec 0$, which is to say that $q\circ\phi_x$ is the orbit map associated to $\alpha$, which is, by construction, a quasi-isometry with uniform constants.  Since $q$ is a quasi-isometry with constants independent of $x$ and $A$, it follows that $\phi_x$ is also. 
	\end{proof}

	\begin{thm}[store=central extensions]\label{thm:B-central-extension} Let $(X,\mathfrak S)$ be an HHS and let $B\leq \Isom(X)$ act by HHS automorphisms. Suppose $A \unlhd B$ is a normal finitely generated infinite abelian subgroup, and that the action of $A$ on $(X,\mathfrak S)$ is proper and hierarchically semisimple.  Then there are finite index subgroups $A' \leq A$ and $B' \leq B$ such that $A'$ is central in $B'$, and the element of $H^2(B'/A',A')$ corresponding to the central extension $A'\hookrightarrow B'\twoheadrightarrow B'/A'$ is represented by a bounded cocycle.
	\end{thm}

	\begin{proof}
        Let $p^+,p^- \in \partial X$ be the boundary points provided by Lemma \ref{lem: fixed boundary points} and let $S \coloneqq \prod_{U\in\support(p^+)}\hull_U(p^-,p^+)$.
        Recall Definitions \ref{defn:P(A)} and \ref{defn:U,P and X(A)}.
        Item \eqref{item:commensurated-invariant-boundary} of Lemma \ref{lem:commensurator stabilises U etc.} implies that $\support(p^\pm)$ and $P$ are $B$--invariant and there exists $B' \leq B$ which fixes each $U \in \support(p^\pm)$ and $\pi_U(p^\pm) \in \partial CU$, and item \eqref{item:normaliser stabilises X(A)} implies that $X(A)$ is $B$--invariant. In particular, there is a natural isometric action of $B$ on $\prod_{U\in\support(p^+)}\mathcal CU$, and this restricts to an action of $B'$ on $S$ which preserves the factors of the above product decomposition.

		\medskip
		
		\noindent\textbf{Normal implies virtually central.}
		Let $Y$ be the injective space, from Proposition \ref{prop: almost invariant hull}, such that there is an $A$--equivariant quasi-isometry $\Phi:c(X(A)) \rightarrow Y$, where $A$ acts on $Y$ trivially. Recall from the proof of that proposition and the proof of Proposition \ref{prop:bounded-orbits} that $Y$ is the fixed point set of the $A$--action on the injective hull $\injhull(\newcone P)$, and $\Phi$ is the restriction of the $\Isom (\newcone P)$--equivariant isometric embedding $\newcone P \hookrightarrow \injhull(\newcone P)$. Normality of $A$ in $B$ implies that $Y$ is $B$--invariant, so there is an action of $B$ on $Y$ such that $\Phi$ is $B$--invariant.

        Let $A' \leq A \cap B'$ be a finite index, free abelian, characteristic subgroup of $A$. Note that $A'$ is normal in $B'$.
		
		\setcounter{claim}{0}
		\begin{claim}\label{claim:B-product-action}
			The action of $A'$ on $S$ is proper, and if the $B'$--action on $X$ is proper, then the diagonal action of $B'$ on $Y\times S$ is proper.
		\end{claim}
		
		\begin{proofofclaim}{\ref{claim:B-product-action}}
			Proposition \ref{prop:short-mountains} provides a constant $\zeta$ such that $\pi_V(X(A))$ has diameter at most $\zeta$ for all $V\in\mathfrak U-\support(p^+)$.
			
			Consider the action of $B'$ on $X(A)$.  Since $c|_{\overline P} \times \gate_H|_{\overline P}$ is a quasi-isometry by Proposition \ref{prop: almost invariant hull}.\eqref{it:barP(A) is a quasi-product}, it follows from Theorem \ref{thm:distance-formula} and the existence of the constant $\zeta$ that $c|_{X(A)} \times \prod_{U\in\support(p^+)}\pi_U|_{X(A)}: X(A) \to Y\times S$ is a $B'$--equivariant quasi-isometry. Hence the $A'$--action on $Y \times S$ is proper and, if the $B'$--action on $X$ is proper, then so is the $B'$--action on $Y \times S$.  Since $A'\leq B'$ acts on $Y$ trivially, it follows that the $A'$--action on $S$ is proper.
		\end{proofofclaim}
		
		Let $N=|\support(p^+)|$, so $N\leq \hhscomp$ and $S$ is naturally quasi-isometric to $\Euclidean^N$.  For each $U\in\support(p^+)$, recall that $L_U:=\hull_U(p^-,p^+)$ is a quasi-line equipped with a $B'$--action fixing the endpoints, so by Lemma \ref{lem:quasi-line-action} and its proof, for each such $U$ there exists a quasi-action $q_U$ of $B'$ on $\reals$ such that (using Claim \ref{claim:B-product-action}):
		\begin{itemize}
			\item For all $b\in B'$, $q_U(b)$ has bounded orbits if and only if $b$ has bounded orbits in $L_U$.
			\item The restriction of $q_U$ to $A'$ is an isometric action.
			\item The diagonal quasi-action of $B'$ on $\reals^N$ given by $\prod_Uq_U$ restricts on $A'$ to a proper action.
			\item The resulting diagonal quasi-action of $B'$ on $Y\times \reals^N$ is quasi-conjugate to the original action of $B'$ on $Y\times S$.
		\end{itemize}
		
		Let $a\in A'$ and $b\in B'$.  Then $[b,a]\in A'$, since $A'$ is normal in $B'$, and therefore $[b,a]$ acts trivially on $Y$.  On the other hand, the displacement of $q_U([b,a])$ is bounded above independently of $a,b$, so by properness, there is a finite subset $\mathcal F\subseteq A'$ such that $[b,a]\in\mathcal F$ for all $b\in B',\ a\in A'$.  Since $A'$ is residually finite, there is a finite-index subgroup of $A'$ disjoint from $\mathcal F-\{1\}$.  Hence, replacing $A'$ by this finite index subgroup, $A'$ is central in $B'$.
		
		\medskip
		
		\noindent\textbf{Bounded cocycle.}  So far, we have a central extension $1\to A'\hookrightarrow B'\stackrel{\phi}{\longrightarrow} \bar B':=B'/A'\to 1$, and a quasi-action of $B'$ on the Euclidean space $\mathbb E^N$ whose restriction to $A'$ is a proper affine isometric action.  
		
		Let $\mathbb F\subseteq \mathbb E^N$ be an $A'$--invariant affine subspace on which $A'$ acts properly and cocompactly by translations, using that $A'$ is free abelian.  Without loss of generality, $0\in \mathbb F$.  Let $R\geq 0$ be chosen so that $\mathbb F\subseteq A'\cdot \neb_R^{\mathbb E^N}(0)$.
		
		Let $\delta'\geq 0$ be such that $\|bb'\cdot 0-b\cdot 0-b'\cdot 0\|_2\leq \delta'$ for all $b,b'\in B'$.  Define a map $q:B'\to A'$ as follows.  Let $\pi:\mathbb E^N\to\mathbb F$ be the orthogonal projection.  Fix $b\in B'$. If $b\in A'$, let $q(b)=b$.  Otherwise, let $q(b)$ be some element of $A'$ such that $\|q(b)\cdot 0-\pi(b\cdot 0)\|_2\leq R$.  
		
		Given $b,b'\in B'$, we have $\|q(bb')\cdot 0-q(b)\cdot 0-q(b')\cdot 0\|_2\leq \|\pi(bb')\cdot 0-\pi(b\cdot 0)-\pi(b'\cdot 0)\|_2+3R\leq \delta'+3R$.  Hence there exists $\delta\geq 0$ such that $|q(bb')-q(b)-q(b')|_{A'}\leq \delta$ for all $b,b'\in B'$, where $|-|_{A'}$ is a a word norm on ${A'}$ coming from some finite generating set.  In other words, $q$ is an ${A'}$--valued quasimorphism which is the identity on $A'$.
		
		Let $s_0:\bar B'\to B'$ be a set-theoretic section of the quotient map $\phi$, with $s_0(1)=1$.  Define a new section $s:\bar B'\to B'$ as follows.  Given $\bar b\in\bar B'$, let $a(\bar b)=q(s_0(\bar b))$, and let $s(\bar b)=s_0(\bar b)a(\bar b)^{-1}=a(\bar b)^{-1}s_0(\bar b)$.  Then $\phi(s(\bar b))=\phi(s_0(\bar b))=\bar b$, so $s$ is again a section of $\phi$.  But now $\|q(s(\bar b))\|_1\leq \delta$ for all $\bar b$.  Now, since $s$ is a section of $\phi$, the cohomology class corresponding to the central extension $B'$ of $B'/A'$ is represented by the cocycle $\bar {B'}^2\ni(\bar g,\bar h)\mapsto s(\bar g)s(\bar h)s(\bar g\bar h)^{-1}\in A'$.  Hence $|q(s(\bar g)s(\bar h)s(\bar g\bar h)^{-1})|_{A'}\leq |q(s(\bar g))+q(s(\bar h))-q(s(\bar g\bar h))|_{A'}+3\delta\leq 6\delta$ for all $\bar g,\bar h\in\bar B'$, so the cocycle is bounded.
	\end{proof}

	\newsection{Applications}\label{sec:applications}
    We now describe some applications to groups acting on HHSes.
	
	\begin{defn}\label{defn:highest-abelian}
		Let $\Gamma$ be a group and let $A\leq \Gamma$ be a virtually abelian subgroup.  Then $A$ is \emph{highest} if for all finite-index subgroups $A'\leq A$ and all virtually abelian subgroups $A''\leq G$ with $A'\leq A''$, we have $[A'':A']<\infty$.
	\end{defn}

    \begin{lem}\label{lem:isomorphism-types}
    Let $(X,\mathfrak S)$ be an HHS and suppose $G$ acts uniformly properly and hierarchically semisimply by isometric HHS automorphisms on $(X,\mathfrak S)$.  Then there are only finitely many abstract isomorphism types of finitely generated virtually abelian subgroups of $G$.    In particular, there are constants $B,N_0,R_0$, depending on $G$ but not on the specific choice of $(X,\mathfrak S)$, such that for all finitely generated virtually abelian $A\leq G$,
    \begin{itemize}
        \item  $|A|\leq B$ if $|A|<\infty$;
        \item  there is a free abelian subgroup $A_0\leq A$ with $[A:A_0]\leq N_0$;
        \item  $\rank(A)\leq R_0$. 
    \end{itemize}
    \end{lem}

    \begin{proof}
    Uniform properness provides a function $f:[0,\infty)\to[0,\infty)$ such that  $|\{g\in G:\dist_G(x,gx)\leq r\}|\leq f(r)$ for all $x\in X$ and all $r\geq 0$.  Corollary \ref{cor: uniformly bounded orbits} provides a constant $B'$, depending only on the HHS parameters, so that any finite subgroup $F\leq G$ has some orbit in $X$ of diameter at most $B'$.  Hence $|F|\leq f(B')$ for all finite subgroups $F\leq G$, which proves the first assertion.  Let $\{F_1,\ldots,F_p\}$ be a set of finite subgroups containing exactly one representative of each abstract isomorphism class of finite subgroups of $G$.

    If $\hhscomp_0$ is the minimal complexity of HHSs admitting $G$--actions satisfying the hypotheses of the lemma, then Theorem \ref{thm:HHS-semisimple-coarse-minset}.\eqref{item:main-rank} implies that $\rank(A)\leq \hhscomp_0$ for all finitely generated virtually abelian subgroups of $G$.  In fact, Proposition \ref{prop: almost invariant hull} together with properness (see also \cite[Prop. 2.17]{HRSS:3-manifold} shows that $\rank(A)$ is bounded above by the maximal $r$ such that $X$ contains an $r$--dimensional quasiflat. So, Lemma \ref{lem:virtually-abelian-crystallographic} implies that $A$ fits into an exact sequence 
    $$1\to F_s\to A\to C\to 1$$
    for some $s$, where $C$ is a crystallographic group of dimension at most $\hhscomp_0$.  Lemma \ref{lem:finitely-many-crystallographic} then implies that there are only finitely many possibilities for the isomorphism type of $A$.
    \end{proof}
	
	\begin{cor}[store=cor:ascending chains, note=Ascending chain condition] \label{cor: ascending chain}
		Let $G$ be a group and let $H_1 \leq H_2 \leq \dots \leq G$ be an ascending chain of finitely generated virtually abelian subgroups. If $G$ acts uniformly properly and hierarchically semisimply by isometric HHS automorphisms on an HHS $(X, \mathfrak S)$, then $H_n = H_{n+1}$ for all sufficiently large $n$.
		
		In particular, any virtually abelian subgroup $H\leq G$ is finitely generated and virtually contained in a highest virtually abelian subgroup.
	\end{cor}

    \begin{proof}
    Let $H_1\leq H_2\leq \ldots$ be an ascending chain of finitely generated virtually abelian subgroups.  Let $R_0,N_0$ be the constants from Lemma \ref{lem:isomorphism-types}.  After passing to a subsequence, there exists $m\in\{0,\ldots,R_0\}$ such that for each $n\geq 1$, there is a normal subgroup $A_n\leq H_n$ such that $A_n\cong \integers^m$ and $[H_n:A_n]\leq N_0$.  Lemma \ref{lem:isomorphism-types} shows that we are done if $m=0$, so assume $m\geq 1$.

    For all $n$, $[H_1:H_1\cap A_n]\leq [H_n:A_n]\leq N_0$ and $H_1\cap A_n$ is free abelian. Since $H_1$ has finitely many subgroups with index $\leq N_0$, we can assume, up to passing to a subsequence, that $A_1=H_1\cap A_n$ for all $n \in \mathbb N$.  It therefore suffices to show that $[A_n:A_1]$, is bounded independently of $n$, since this will imply $H_n=H_{n+1}$ for sufficiently large $n$.

    Let $X(A_1)$ be as in Proposition \ref{prop: almost invariant hull}. Let $\mathcal M_{A_1}\subseteq X(A_1)$ be the subspace obtained by applying Theorem \ref{thm:HHS-semisimple-coarse-minset}.\eqref{item:main-full-minset} to $A_1$.  The same theorem provides $q\geq 0$, depending only on $A_1$ and the HHS parameters, such that for all $y\in\mathcal M_{A_1}$, there is an $A_1$--invariant $(k,k)$--quasiflat $F(y)\subseteq \mathcal M_{A_1}$ and $F(y)\subseteq A_1\cdot \neb_q(y)$.

    For any $n\geq 1$, we have $A_1\leq A_n$ and since $A_n$ is abelian, $A_1$ is normal in $A_n$.  Hence $A_n$ stabilises $X(A_1)$, by Lemma \ref{lem:commensurator stabilises U etc.}, and therefore each $A_n$ stabilises the uniformly HQC subset $H_\theta(X(A_1))$.  So, by Lemma \ref{lem: HHS structure on HQC subspaces}, $H_\theta(X(A_1))$ is an HHS on which each $A_n$ acts properly and hierarchically semisimply by HHS automorphisms.  Now apply Theorem \ref{thm:HHS-semisimple-coarse-minset}.\eqref{item:main-one-cocompact-quasiflat-1} to this action, so that for each $n$, we get an $A_n$--invariant subspace $F_n$ of $H_\theta(X(A_1))$ that is $(k,k)$-quasi-isometric to $\mathbb E^m$.  Now, Theorem \ref{thm:HHS-semisimple-coarse-minset}.\eqref{it:coarse min set} implies that $F_n\subseteq \neb_{\omega(k)}(\mathcal M_{A_1})$, where $\omega$ is the function provided by Theorem \ref{thm:HHS-semisimple-coarse-minset}.\eqref{it:coarse min set}.

    Choose $y\in\mathcal M(A_1)$ such that $\dist(y,y')\leq \omega(k)$ for some $y'\in F_n$.  Define $\pi_n:F(y)\to F_n$ as follows.  Given $x\in F(y)$, choose $a\in A_1$ such that $x\in \neb_q(a\cdot y)$, and then let $\pi_n(x)=a\cdot y'$.  Note that $\dist(x,\pi_n(x))\leq q+\omega(k)$.  So, $\pi_n$ is $C$--coarsely lipschitz, where $C$ depends on $k,q,\omega$ but not on $n$.  Lemma \ref{lem:Borsuk-Ulam-redux} provides $C'$, depending only on $C$ and $k$, such that $\pi_n$ is $C'$--coarsely surjective.  
        
    Now let $z\in F_n$.  Choose $x\in F(y)$ such that $\dist(z,\pi_n(x))\leq C'$.  Then $\dist(z,x)\leq C'+q+\omega(k)$.  Hence $F_n\subseteq A_1\cdot\neb_{C'+q+\omega(k)}(y)$.  Thus $[A_n:A_1]\leq f(2(\omega(k)+q+C'))$, so $[H_n:A_1]\leq N_0f(2(\omega(k)+q+C'))$ which is bounded independently of $n$, as needed. 
    \end{proof}

\begin{cor}[store=cor:solvable] \label{cor:solvable}
	Let $(X, \mathfrak S)$ be an HHS and let $G$ be a group acting uniformly properly and hierarchically semisimply by HHS automorphisms on $X$. If $G$ is virtually solvable then $G$ is virtually abelian. 
\end{cor}

\begin{proof}
	Suppose $G$ is virtually solvable with derived length $k \geq 1$. We proceed by induction on $k$. If $k=1$ then $G$ is abelian, so suppose $k > 1$ and $G' \coloneqq [G,G]$ is virtually abelian. It follows from Corollary \ref{cor: ascending chain} that $G'$ is finitely generated, so there exists a finite index, finitely generated, free abelian subgroup $A \leq G'$. By Theorem \ref{thm:B-central-extension}, we can assume, up to taking a further finite index subgroup, that $A$ is central in a finite index subgroup $H \leq G$. Moreover the element of $H^2(H/A, A)$ corresponding to the central extension $1 \rightarrow A \rightarrow H \rightarrow H/A \rightarrow 1$ is bounded. This implies that $H$ is quasi-isometric to $A \times H/A$, by \cite{GerstenBounded92}. Let $\pi: H \rightarrow H/A$ be the quotient map and let $H' \coloneqq [H,H]$. Since $H' \leq G'$, the image $\pi(H')$ is finite. Moreover $\pi(H')$ is normal in $H/A$ and $(H/A) /  \pi(H') = H/H'$ is abelian, so \cite[II.7.9]{BridsonHaefliger:metric} implies that $H/A$ is virtually abelian. Therefore $A \times H/A$ is virtually abelian and, since virtually abelian groups are quasi-isometrically rigid \cite{PansuCroissance83} (see \cite[Sec. 4.3]{DrutuQuasi-isometry09}), this implies that $H$ is virtually abelian.
\end{proof}

As we remarked earlier, the hierarchical semisimplicity assumption holds for HHGs, so the theorems of Section \ref{sec:main-theorem} apply to HHGs, yielding the following ``coarse flat torus'' result:

\begin{cor}\label{cor:abelian-subgroups-HQC}
    Let $(G,\mathfrak S)$ be an HHG with complexity $\hhscomp$.
	There exists a constant $k$, depending only on the HHS parameters, such that for any virtually abelian subgroup $A\leq G$, there is an $A$--invariant $(k,k)$--EHQC subspace $F\subseteq G$, a proper, cocompact, isometric $A$--action on $\Euclidean^n$ for some $n\in\{0,\ldots\,\hhscomp\}$, and an $A$--equivariant $(k,k)$--quasi-isometry $F\to\Euclidean^n$. 
    
    In particular, virtually abelian subgroups of $G$ are finitely generated and belong to one of finitely many abstract isomorphism types.
\end{cor}

\begin{proof}
	Hierarchical semisimplicity comes from \cite[Thm. 3.1]{DurhamCorrection20}, at which point Theorem \ref{thm:HHS-semisimple-coarse-minset} applies, since $A$ is finitely generated by Corollary \ref{cor: ascending chain}.  This gives the subspaces $F$ from the statement, while the statement about isomorphism types comes from Lemma \ref{lem:isomorphism-types}. 
\end{proof}

\subsection{Centralisers, normalisers and commensurators}

We can now use the previous results to study the algebraic and geometric properties of the centralisers, normalisers and commensurators of virtually abelian subgroups of HHGS.

\begin{cor}\label{cor:normaliser}
	Let $(G,\mathfrak S)$ be an HHG.  Let $A\leq G$ be a virtually abelian subgroup.  Then $N_G(A)$ lies at finite Hausdorff distance from the set $\mathcal M_A$ from Theorem \ref{thm:HHS-semisimple-coarse-minset}.\eqref{item:main-full-minset}, and in particular is EQHC.  Moreover, $[N_G(A_0):C_G(A_0)]<\infty$, where $A_0$ is the finite-index subgroup of $A$ from Theorem \ref{thm:HHS-semisimple-coarse-minset}.\eqref{it: M is coarsely X(A)}.
    
    Finally, there is a free abelian subgroup $A_1\leq A$ such that $C_G(A_1)$ is HQC in any HHG structure on $G$, and $[A:A_1]\leq D$, where $D\in\naturals$ depends only on the group $G$ (and not on $A$ or the choice of HHG structure).
\end{cor}

\begin{proof}[Proofs of Corollary \ref{cor:normaliser} and Corollary \ref{cor:intro-centraliser}]
    By Corollary \ref{cor:abelian-subgroups-HQC}, the virtually abelian subgroup $A$ is finitely generated. Let $n \coloneqq \rank(A)$. The set $X(A)$ from Theorem \ref{thm:HHS-semisimple-coarse-minset} is $N_G(A)$--invariant by Lemma \ref{lem:commensurator stabilises U etc.}.\eqref{item:normaliser stabilises X(A)}.  Recall from Theorem \ref{thm:HHS-semisimple-coarse-minset} the EHQC set $\mathcal M_A$, which is $A$--invariant.  Let $\{x_i\}_{i\in I}$ contain exactly one point in each $A$--orbit in $\mathcal M_A$.  Let $\mathcal S$ be a fixed finite generating set of $A$. 
		
	Let $(A,\dist_A)$ be the metric group from Corollary \ref{cor:coarse-foliation}.  The same corollary then provides $k,r\geq 1$, depending only on the HHS parameters, such that for all $x\in\mathcal M_A$, there is an $A$--equivariant $(r,r)$--quasi-isometric embedding $\phi_x:(A,\dist_A)\to \mathcal M_A$ given by $\phi_x(a)=ax$, and $\phi_x(A)\subseteq F_x$, an $A$--invariant $(k,k)$--EHQC $(k,k)$--quasiflat in $\mathcal M_A$.

    \setcounter{claim}{0}
    \begin{claim}\label{claim:M-A-close-to-normaliser}
     There exists $R<\infty$ such that $\mathcal M_A\subseteq \neb_{R}(N_G(A))$.  If $A$ is abelian, the same holds with $N_G(A)$ replaced by $C_G(A)$.    
    \end{claim}

    \begin{proofofclaim}{\ref{claim:M-A-close-to-normaliser}}
    Let $\xi=r\max_{s\in \mathcal S}\dist_A(1,s)+r$.  Then $\phi_{x_i}(s)=sx_i\in \neb_\xi(x_i)$ for all $i\in I$ and $s\in \mathcal S$.  For each $i\in I$, we have $x_i^{-1}\neb_\xi(x_i)=\neb_\xi(1)$, which is finite.  The inner automorphism of $G$ given by conjugation by $x_i^{-1}$ restricts to an injective map $\mathcal S\to \neb_\xi(1)$.  Since there are finitely many possibilities for this map, there is a finite set $\{i_1,\ldots,i_p\}\subseteq I$ such that for all $i\in I$, there exists $j\leq p$ for which $x_{i_j}x_i^{-1}sx_ix_{i_j}^{-1}=s$ for all $s\in \mathcal S$.  Hence $x_i=bx_{i_j}$ for some $b\in C_G(A)$.  Now, any $x\in\mathcal M_A$ has the form $ax_i$ for some $a\in A$ and $i\in I$, so $x=abx_{i_j}\in N_G(A)\cdot x_{i_j}$.  Letting $R=\max_{j\leq p}\dist(1,x_{i_j})$, we therefore have $\dist(x,N_G(A))\leq \dist(abx_{i_j},ab)\leq R$, as required.  If $A$ is abelian, then $ab\in C_G(A)$, so $x=abx_{i_j}\in C_G(A)\cdot x_{i_j}$, and we get $\dist(x,C_G(A))\leq R$, as claimed.     
    \end{proofofclaim}

    On the other hand:

    \begin{claim}\label{claim:normaliser-close-to-M-A}
    There exists $R'<\infty$ such that $N_G(A)\subseteq \neb_{R'}(\mathcal M_A)$.
    \end{claim}

    \begin{proofofclaim}{\ref{claim:normaliser-close-to-M-A}}
    Let $x\in\mathcal M_A$ and let $g\in N_G(A)$.  Then $g\cdot\phi_x(A)$ is an $(r,r)$--quasiflat contained in the $(k,k)$--EHQC, $(k,k)$--quasiflat $gF_x$, which is invariant under $gAg^{-1}=A$.  Moreover, $F_x\subseteq X(A)$ since, by definition, $\mathcal M_A\subseteq X(A)$.  Since $X(A)$ is $N_G(A)$--invariant, $gF_x\subseteq X(A)$.  Hence, by Theorem \ref{thm:HHS-semisimple-coarse-minset}.\eqref{it:coarse min set}, $gF_x$, and in particular $gx$, is contained in  $\neb_{\omega_1(k)}(\mathcal M_A)$. We have just shown that $N_G(A)\cdot x\subseteq \neb_{\omega_1(k)}(\mathcal M_A)$.  Hence $N_G(A)\subseteq \neb_{R'}(\mathcal M_A)$, where $R'=\omega_1(k)+\dist(1,\mathcal M_A)$.
    \end{proofofclaim}

    The preceding two claims show that $N_G(A)$ lies at finite Hausdorff distance from $\mathcal M_A$.  Since $\mathcal M_A$ is EHQC, it follows that $N_G(A)$ is also EHQC, as claimed.

    \medskip

    Let $A_0\leq A$ be the finite-index free abelian subgroup obtained from Theorem \ref{thm:HHS-semisimple-coarse-minset}.  Claim \ref{claim:M-A-close-to-normaliser} implies that $\mathcal M_{A_0}\subseteq \neb_R(C_G(A_0))$ for some $R<\infty$.   Theorem \ref{thm:HHS-semisimple-coarse-minset}.\eqref{it: M is coarsely X(A)} says that $\mathcal M_{A_0}$ coarsely coincides with $X(A_0)$, so $N_G(A_0)$ acts on $X(A_0)$ coboundedly, and the induced action of $C_G(A_0)$ is also cobounded.  This proves that $[N_G(A_0):C_G(A_0)]<\infty$.

    Let $\hhscomp_0$ denote the maximal $r$ such that there is a subset $\{U_1,\ldots,U_r\}\subseteq \mathfrak S$ for which $U_i\orth U_j$ for $i\neq j$, and $\mathcal CU_i$ is unbounded for $i\neq j$.  (So, $\hhscomp_0$ is the \emph{rank} of $(G,\mathfrak S)$ in the sense of \cite[Defn. 1.10]{BehrstockQuasiflats21}.)

    Let $A_1\leq A$ be any free abelian subgroup that is contained in every subgroup of $A$ of index at most $\hhscomp_0!$.  Then by cofiniteness of the $G$--action on $\mathfrak S$, and Proposition \ref{prop:bounded-orbits}.(\ref{it:zs-are-close},\ref{it:finite index hqc fixed point set}), and Proposition \ref{prop: almost invariant hull}, there is an $A_1$--invariant hierarchically quasiconvex subset $\Theta$ of $X$ that lies in a neighbourhood of $X(A_1)$.  Since $C_G(A_1)$ acts on $X(A_1)$ coboundedly, as shown above, it follows that $C_G(A_1)$ acts on $\Theta$ coboundedly, and is therefore hierarchically quasiconvex with respect to the HHG structure $(G,\mathfrak S)$. 

    By \cite[Thm. 1.15]{BehrstockQuasiflats21} and coboundedness of the $G$--action, $\hhscomp_0$ coincides with the maximal $r$ so that there is a quasi-isometric embedding $\reals^r\to G$, and therefore depends only on (the QI type of) $G$.  Since Lemma \ref{lem:isomorphism-types} bounds the index to which one must pass to find a free abelian subgroup of $A$, we can therefore find the required $A_1$ with $[A:A_1]$ bounded in terms of $G$ only.  In particular, since the choice of $A_1$ in $A$ was made independently of the HHG structure, $C_G(A_1)$ is HQC with respect to any HHG structure.  This completes the proof of Corollary \ref{cor:normaliser}.

    Corollary \ref{cor:intro-centraliser} follows by a very similar argument.  Indeed, if no element of $\mathfrak S$ is orthogonal to its $G$--translates, and $A$ is a free abelian subgroup, then the above application of Proposition \ref{prop:bounded-orbits} and Proposition \ref{prop: almost invariant hull} allows us to take $A_1=A$ and conclude that $C_G(A)$ is hierarchically quasiconvex.
\end{proof}

Next we analyse commensurators; the following result is analogous to the statement for CAT(0) and cubical groups in \cite{HuangPrytula:commensurators}.  In fact, \cite[Prop. 9.1]{HuangPrytula:commensurators} shows that this statement is sharp, in the sense that the finitely generated subgroup $K$ cannot be replaced, in general HHGs, with the entire commensurator.

\begin{cor}[store=cor:commensurators]\label{cor:commensurators}
	Let $(G,\mathfrak S)$ be an HHG and let $A\leq G$ be a virtually abelian subgroup.  Let $\commensurator{G}{A}$ be the commensurator of $A$ in $G$.  Let $K\leq \commensurator{G}{A}$ be a finitely generated subgroup.  Then $K\leq N_G(A')$, where $A'\leq A$ is a finite-index subgroup.
\end{cor}

\begin{proof}
	Let $C=\commensurator{G}{A}$.  Let $K\leq C$ be a finitely generated subgroup.  Without loss of generality, $A\leq K$.
	
	Lemma \ref{lem:commensurator stabilises U etc.}.\eqref{item:commensurated-invariant-boundary} states that $C\cdot \support(p^\pm)=\support(p^\pm)$ (where $p^\pm$ come from applying Lemma \ref{lem: fixed boundary points} to $A$) and there is a finite-index subgroup $C'\leq C$ such that $C'$ fixes each $U\in\support(p^\pm)$ and each $\pi_U(p^\pm)$, and hence $C'$ acts by isometries on $S=\prod_{U\in\support(p^\pm)}\hull_U(p^-,p^+)$, preserving the factors.  Lemma \ref{lem:commensurator stabilises U etc.}.\eqref{item:commensurated-cone-off} gives a $C$--action on $(\newcone P,\mathfrak S-\mathfrak U)$ making $c:P\to\newcone P$ a $C$--equivariant map (recall Definitions \ref{defn:P(A)} and \ref{defn:U,P and X(A)}).  By replacing $\newcone P$ with its injective hull, we can assume that $\newcone P$ is injective.\footnote{When replacing $\newcone P$ by the injective hull, we keep the same HHS structure, except we modify the projections $\pi_U:\newcone P\to \mathcal CU$ using a fixed quasi-isometry $\injhull(\newcone P)\to \newcone P$, as in \cite[Prop. 1.10]{BehrstockHierarchically19}.  If desired, although it is not strictly necessary in this proof, we can further modify these maps as in \cite[Sec. 2.2]{DurhamCorrection20} so that the action is still by HHS automorphisms.}
	
	Let $K'=C'\cap K$ and note that $[K:K']<\infty$, so $K'$ is generated by elements $k_1,\ldots,k_s\in K'$.  Let $A'\leq A\cap K'$ be a finite index free abelian subgroup, and let $A''=\bigcap_{i=1}^sA'\cap k_iA'k_i^{-1}$.  Let $A'_i=k_iA'k_i^{-1}$.  Then each $A_i'$ has a nonempty fixed point set $Y(A_i')$ in the injective space $\newcone P$, and $A''$ has a nonempty fixed point set $Y(A'')$.  Note that $Y(A_i')\subseteq Y(A'')$ for all $i$.
	
	The map $\delta:=c\times \prod_{U\in\support(p^\pm)}\pi_U:X(A)\to \newcone P\times S$ is $C'$--equivariant, where the codomain is equipped with the diagonal action.  Moreover, exactly as in the proof of Claim \ref{claim:B-product-action} of the proof of Theorem \ref{thm:B-central-extension}, there exists $r$, depending on $A$ and the HHS parameters but independent of $K$ and the choice of $A'$ and the $k_i$, such that $\delta$ is an $(r,r)$--quasi-isometric embedding.
	
	Now let $i\leq s$ and let $a\in A''$.  Then $k_iak_i^{-1}\in A_i'$ since $a\in A''\leq A'$.  On the other hand, $a \in A''\leq A_i'$.  So, $[a,k_i]\in A_i'$, and hence $[a,k_i]$ fixes $Y(A_i')$ pointwise.  Meanwhile, $[a,k_i]$ displaces all points in $S$ by a uniformly bounded amount, because $a$ and $k_i$ both stabilise each quasiline $\hull_U(p^-,p^+)$ and fix the endpoints.  
	
	This shows that there exists $r'<\infty$ depending only on $r$ and the HHS constants, and $x_0\in G$ depending on $K$ and the choice of $k_1,\ldots,k_s$ but not on $A'$, such that for all $a\in A''$ and all $k_i$, we have $\dist_G([a,k_i]\cdot x_0,x_0)\leq r'$, and $[a,k_i]\in A_i'$.  Now, since $A$ is residually finite, we could have chosen $A'$ so that every nontrivial element of $A'$ moves every point in $G$ a distance more than $r'$.  Since $A_i'$ is conjugate in $G$ to $A'$, the same is true of $A_i'$.  With this choice of $A'$, it then follows that $[a,k_i]=1$ for all $i\leq s$ and all $a\in A''$.  Hence $A''\cap K'$ is central in $K'$.  But $A''':=A''\cap K'$ has finite index in $A$ since we are assuming that $A\leq K$.
	
	So, $A'''$ is a finite index subgroup of $A$ that is central in $K'$.  Let $\{\ell_1,\ldots,\ell_m\}\subseteq K$ be a left transversal for $K'$ and let $A_1=\bigcap_{i=1}^m\ell_iA'''\ell_i^{-1}$.  Now, $\ell_i\in C$, so $[A:A_1]<\infty$, and $A_1$ is central in $K'$, so from the definition, it follows that $A_1$ is normal in $K$, so we are done.
\end{proof}

\subsection{Ruling out HHG structures using highest virtually abelian subgroups}\label{subsec:coxeter}

Highest virtually abelian subgroups can be used to rule out the existence of HHG structures on a group using the following corollary.

\begin{cor}[store=cor:highest-abelian]\label{cor:highest-abelian}
	Let $(G,\mathfrak S)$ be an HHG. Every highest virtually abelian subgroup $A\leq G$ is hierarchically quasiconvex.  Thus every abelian subgroup is virtually contained in a hierarchically quasiconvex abelian subgroup.
\end{cor}

\begin{proof}
Suppose that $A$ is highest.  Corollary \ref{cor:normaliser} provides a finite-index subgroup $A_1\leq A$ such that $C_G(A_1)$ is hierarchically quasiconvex. If $C_G(A_1)$ has a nonabelian free subgroup, then there exists $g\in C_G(A_1)$ such that $\langle g,A_1\rangle\cong \integers\times A_1$, contradicting that $A$ is highest.  Hence, by \cite[Thm. 4.1]{DurhamCorrection20}, every finitely generated subgroup of $C_G(A_1)$ is virtually abelian, so Corollary \ref{cor: ascending chain} implies $[C_G(A_1):A_1]<\infty$, so $A_1$ (and thus $A$) is HQC.
\end{proof}

Using this, we can describe the possible crystallographic groups that can appear as highest virtually abelian subgroups of HHGs, by generalising \cite[Thm. 4.4]{PetytUnbounded23}.  Recall from Definition \ref{defn:hyperoctahedral} the notion of a hyperoctahedral virtually abelian group.

\begin{cor}\label{cor:abelian-subgroups-octahedral}
	Let $(G,\mathfrak S)$ be an HHG.  Let $A\leq G$ be a highest virtually abelian subgroup.  Then $A$ is hyperoctahedral.
\end{cor}

\begin{proof}
	By Corollary \ref{cor:highest-abelian}, $A$ is hierarchically quasiconvex in $G$.  Hence $A$ acts properly and coboundedly on an HHS, by Lemma \ref{lem: HHS structure on HQC subspaces}, and is therefore coarsely injective, by Corollary \ref{cor: coarsely dense in injective hull}.  This, together with \cite[Thm. A]{hodaCrystallographicHellyGroups2023} implies that the point group is conjugate in $GL_n(\reals)$ into the isometry group of $(\Euclidean^n,\|\cdot\|_\infty)$, which is to say $O_n(\integers)$.
\end{proof}

\begin{remark}\label{rem:crystallographic-obstruction}
	One can use highest virtually abelian subgroups $A$ of a group $G$ as an obstruction to cocompact cubulation of $G$, as follows: if $G$ is cocompactly cubulated, then $A$ is also cocompactly cubulated, by \cite[Thm. 3.6]{WoodhouseWise:cubical}, so $A$ is hyperoctahedral by \cite{Hagen:crystallographic} or \cite{hodaCrystallographicHellyGroups2023}.  Corollary \ref{cor:abelian-subgroups-octahedral} can be used in an analogous way to rule out the existence of HHG structures in various examples.  We now illustrate this in the case of Coxeter groups.
\end{remark}

Let $(W,S)$ be a Coxeter system.  Following \cite[Sec. 6]{Krammer:Coxeter}, a \emph{standard abelian subgroup} $A\leq W$ has the form $\prod_{i=1}^kH_i$, where each $H_i$ is defined as follows: there are irreducible, non-spherical subsets $I_1,\ldots,I_k$ of $S$ such that $\langle W_{I_i},W_{I_j}\rangle_W\cong W_{I_i}\times W_{I_j}$ for $i\neq j$, and, if $I_i$ is affine, then $H_i\leq_{f.i.} W_{I_i}$ is the translation subgroup, and otherwise $H_i\cong\integers$.  A \emph{standard virtually abelian subgroup} has the form $\prod_i\widehat H_i$, where $\widehat H_i=W_{I_i}$ for $I_i$ irreducible affine, and $\widehat H_i=H_i\cong\integers$ otherwise. 

\begin{cor}\label{cor:coxeter-group}
	Let $(W,S)$ be a Coxeter system.  If $W$ is a hierarchically hyperbolic group, then every maximal standard virtually abelian subgroup is hyperoctahedral.
\end{cor}

\begin{proof}
	By \cite[Thm. 6.8.3]{Krammer:Coxeter}, every abelian subgroup of $W$ is virtually contained in a conjugate of a standard abelian subgroup.  Hence maximal standard virtually abelian subgroups are highest virtually abelian subgroups, and the claim follows from Corollary \ref{cor:abelian-subgroups-octahedral}. 
\end{proof}

For example, if $W$ contains the $(3,3,3)$ triangle group as a maximal standard virtually abelian subgroup, then $W$ is not an HHG.

\subsection{Quasiflat closing}\label{subsec:quasiflat-closing}
We now turn to the matter of the existence of higher rank free abelian subgroups in non-hyperbolic HHGs.  

\begin{lem} \label{lem: cobounded nest-minimal domains}
    Let $(G, \mathfrak S)$ be a normalised HHG and let $U \in \mathfrak S$ be such that $\mathcal CU$ is unbounded and $\mathcal CV$ is bounded for any $V \propnest U$. Then $\stabilizer_G(U)$ acts cocompactly on $P_U$ and coboundedly on $\mathcal CU$.
\end{lem}

The above lemma is of independent interest for the theory of HHGs.  The lemma does not hold for arbitrary $U\in\mathfrak S$, as can be illustrated with a suitably-chosen HHG structure on $F_2$.  

\begin{proof}[Proof of Lemma \ref{lem: cobounded nest-minimal domains}]
    Let $C \geq 1$ be such that every bounded domain of $\mathfrak S$ has diameter at most $C$.

    \setcounter{claim}{0}
    \begin{claim} \label{claim: locally finite product regrions}
        If $I \subseteq G \cdot U$ is such that $\{P_i\}_{i \in I}$ all intersect the same ball in $G$, then $|I| < \infty$.
    \end{claim}
    \begin{proofofclaim}{\ref{claim: locally finite product regrions}}
    Suppose there is an infinite subset $I \subseteq G \cdot U$ such that $\{P_i\}_{i \in I}$ all intersect the same ball $B$ in $G$.    

    Since $\mathfrak S$ has a unique $\nest$--maximal element and $\nest$--chains are bounded, there exists a $W\in\mathfrak S$ that is $\nest$--minimal with the property that $aU\nest W$ for infinitely many $aU\in I$.  By passing to an infinite subset, we can assume that this holds for all $aU\in I$.  Fix a coarse parallel copy $F_W\subseteq P_W$ (recall the definition of $F_W$ from Section \ref{subsec:product-regions}).

    Since $G$ is the $0$--skeleton of a locally finite graph, up to replacing $I$ with an infinite subset, there exists $x\in \bigcap_{i \in I} P_i$.  Considering the gate of $x$ in $F_W$ shows that there is a ball of uniform radius in $F_W$ that intersects each $P_i$, so we can assume that $x\in F_W\cap \bigcap_{i\in I}P_i$. Recall from \cite[Prop. 5.11]{BehrstockHierarchically19} that $(F_W,\mathfrak S_W)$ inherits an HHS structure with uniform parameters, where $\mathfrak S_W=\{V\in\mathfrak S:V\nest W\}$.  Since the product regions $P_{aU}$ defined intrinsically in this HHS uniformly coarsely coincide with $P_{aU}\cap F_W$, we will work in $(F_W,\mathfrak S_W)$ and regard each $P_i$ as a standard product region for this HHS.
    
    For each $i \in I$ let $x_i \in P_i$ be such that $\theta_u(2(C+2E)) \leq d(x,x_i) \leq \theta_u(2(C+2E)) + D$, where $\theta_u$ is the uniqueness function from \cite[Defn. 1.1.(9)]{BehrstockHierarchically19} and $D$ is a uniform constant such that all standard product regions are $D$--coarsely connected. 
    Since $F_W$ is proper, there is an infinite subset $I' \subseteq I$ such that $x_i = x_j$ for all $i,j \in I'$. Let $y \coloneqq x_i$ for any $i \in I'$. By the uniqueness axiom, there exists $V \in \mathfrak S_W$ such that $d_V(x,y) \geq 2(C+2E)$ which, by the definition of $C$ and the fact that $x,y\in P_{aU}$ for all $aU\in I'$, implies that $V \orth  aU$ or $V = aU$ for all $aU \in I'$. Let $I_1 \coloneqq I' - \{V\}$.  By \cite[Defn. 1.1.(3)]{BehrstockHierarchically19}, there exists $W'\propnest W$ such that $aU\nest W'$ for all $aU\in I_1$.  Since $I_1$ is infinite, this contradicts the $\nest$--minimality of $W$.
    \end{proofofclaim}

    Claim \ref{claim: locally finite product regrions} implies, by the pigeonhole principle for cocompact actions (see e.g. \cite[Prop. 7.2]{HruskaConnectedness21} or \cite[Lem. 2.3]{HagenSusse}), that the action of $\stabilizer_G(P_U)$ on $P_U$ is cocompact. Moreover, since $P_U = P_{gU}$ for any $g \in \stabilizer_G(U)$, the claim also implies that the orbit $\stabilizer_G(P_U) \cdot U$ is finite, so the index of $\stabilizer_G(U)$ in $\stabilizer_G(P_U)$ is finite, and the action of $\stabilizer_G(U)$ on $P_U$ is also cocompact. Since $\pi_U$ is coarsely surjective, and coarsely factors through the gate map $\gate_{P_U}$, this implies that the action of $\stabilizer_G(U)$ on $\mathcal CU$ is cobounded.
\end{proof}

\begin{thm}[store=flat closing] \label{thm: flat closing}
    Let $(G, \mathfrak S)$ be an HHG. Then $G$ is hyperbolic if and only if $G$ contains no $\mathbb Z^2$ subgroups.
\end{thm}
\begin{proof}
    First normalise $(G,\mathfrak S)$.
    Then there exists $C \geq 1$ such that every bounded domain of $\mathfrak S$ has diameter at most $C$. Given $g \in G$, let $\bigset(g) \coloneqq \{U \in \mathfrak S : \pi_U(\la g \ra)$ has unbounded diameter$\}$, any two of whose elements are orthogonal. Recall that the action of $G$ on $(G,\mathfrak S)$ is hierarchically semisimple.

    Suppose $G$ is not hyperbolic. We will show that there exists $f \in G$ such that $|\bigset(f)| \geq 2$. Then, by Corollary \ref{cor:highest-abelian}, there is a highest abelian subgroup $A \leq G$ which contains a finite index subgroup of $\la f \ra$, and $A$ is HQC and therefore must have rank $\geq 2$.
    
    By \cite[Cor. 2.15]{BehrstockQuasiflats21} and non-hyperbolicity of $G$, there exist $U,V \in \mathfrak S$ such that $U \orth V$ and $\mathcal CU, \mathcal CV$ are unbounded. Using \cite[Def. 1.1.(3),(5)]{BehrstockHierarchically19}, we can assume that, for any $U' \propnest U$ and $V' \propnest V$, the domains $\mathcal CU', \mathcal CV'$ are bounded. Therefore Lemma \ref{lem: cobounded nest-minimal domains} implies that there exist $g \in \stabilizer_G(U)$ and $h \in \stabilizer_G(V)$ such that $U \in \bigset(g)$ and $V \in \bigset(h)$. 
    
    If either $|\bigset(g)| > 1$ or $|\bigset(h)| > 1$ then we are done, so suppose $\bigset(g) = \{U\}$ and $\bigset(h) = \{V\}$. 
    
    Fix a point $x \in G$. Let $L \coloneqq H_\theta(\la g \ra \cdot x)$ and $M \coloneqq H_\theta(\la h \ra \cdot x)$. Note that $L$ is $g$--invariant and $M$ is $h$--invariant. Let $H \coloneqq H_\theta(L \cup M)$. 

    Lemma 6.6 in \cite{DurhamBoundaries17} provides a constant $B$ such that $\diam(\pi_W(L)) \leq B$ for all $W \neq U$ and $\diam(\pi_W(M)) \leq B$ for all $W \neq V$. If $W \in \mathfrak S - \{U,V\}$, then $\pi_W(H) \subseteq \mathcal N_B(\pi_W(x))$, so $\diam(\pi_W(H)) \leq 3B$.
    Therefore, if $x,y \in H$, then $\relevant_{4B}(x,y) \subseteq \{U,V\}$ and, if $n \in \mathbb Z$ and $x,y \in g^nH$, then $\relevant_{4B}(x,y) \subseteq \{U,g^nV\}$. If $x,y \in H \cap g^nH$ and $g^nV \neq V$, then  $\relevant_{4B}(x,y) \subseteq \{U\}$. Since $L \subseteq H \cap g^nH$, this, together with the uniqueness axiom and the definition of $H$, implies that $H \cap g^nH \subseteq \mathcal N_D(L)$, for some constant $D \geq 1$ which is independent of $n$. 

    Fix a point $x_0 \in L$. There exists $k \geq 2$ such that, for each $n \in \mathbb N$, there exists $x_n \in g^nH$ such that $d(x_0,x_n) \leq kD$ and $d(x_n,L) > D$. The space $G$ is locally finite, so there exists $n \neq m$ such that $x_n = x_m$. But then $g^{-n}x_n \in H \cap g^{m-n}H$ and $d(L, g^{-n}x_n) = d(L,x_n) > D$, so we have $g^{m-n}V = V$. By a symmetric argument, there is a non-zero power of $h$ which fixes $U$. It follows by \cite[Prop. 6.68]{Genevois:hyperbolicities} that there exists $f \in \stabilizer_G(U) \cap \stabilizer_G(V)$ such that $\{U,V\} \subseteq \bigset(f)$.
\end{proof}

Recall that the \emph{rank} of a hierarchically hyperbolic group $(G,\mathfrak S)$ is the maximal $k$ such that $\mathfrak S$ contains a set of $k$ pairwise orthogonal elements $U_1,\ldots,U_k$ such that $\mathcal CU_i$ is unbounded for all $i\leq k$.  Theorem \ref{thm: flat closing} generalises as follows:

\begin{thm}[store=higher-rank]\label{thm:higher-rank}
    Let $(G, \mathfrak S)$ be an HHG with rank $\nu$. Then there exists $\mathbb Z^\nu \leq G$.
\end{thm}

We first need a strengthening of Lemma \ref{lem: cobounded nest-minimal domains}.

\begin{lem} \label{lem: nest-minimal intersections of product regions}
    Let $(G, \mathfrak S)$ be a normalised HHG and let $U_1, \dots, U_k \subseteq \mathfrak S$ be pairwise orthogonal domains such that, for any $U \propnest U_i$, the space $\mathcal CU$ is bounded. Let $Q \coloneqq \bigcap_{i=1}^k P_{U_i}$. Then the restriction of $\pi_{U_i}$ to $Q$ is $10E$--coarsely surjective for each $i$, and $\bigcap_{i=1}^k \stabilizer_G(U_i)$ acts cocompactly on $Q$. Therefore $\bigcap_{i=1}^k \stabilizer_G(U_i)$ acts coboundedly on $\mathcal CU_i$ for each $i$.
\end{lem}
\begin{proof}
    The fact that the restriction of each $\pi_{U_i}$ to $Q$ is $10E$--coarsely surjective is a consequence of the partial realisation axiom \cite[Def. 1.1.(8)]{BehrstockHierarchically19}. To prove the cocompactness statement, we will induct on $k$.
    
    If $k=1$ then $Q = P_{U_1}$, so cocompactness holds by Lemma \ref{lem: cobounded nest-minimal domains}. Suppose $k > 1$ and suppose the claim holds for sets of up to $k-1$ domains. Then $H \coloneqq \bigcap_{i=1}^{k-1} \stabilizer_G(U_i)$ acts cocompactly on $Q' \coloneqq \bigcap_{i=1}^{k-1} P_{U_i}$. 

    Let $T \subseteq H$ be a left-transversal for $\stabilizer_H(U_k)$. Let $I \subseteq T$ be such that $\{hQ: h \in I\}$ all intersect the same ball in $Q'$. For each $h \in H$, we have $hQ = h(Q' \cap P_{U_k}) = Q' \cap P_{hU_k}$, so this implies that $\{P_{hU_k} : h \in I\}$ all intersect the same ball in $G$. By Claim \ref{claim: locally finite product regrions} from the proof of Lemma \ref{lem: cobounded nest-minimal domains}, we therefore have $|I| < \infty$. 
    The pigeonhole principle for cocompact actions then implies that $\stabilizer_H(Q)$ acts cocompactly on $Q$. Moreover, since $hQ = Q$ for any $h \in \stabilizer_H(U_k)$, this also implies that $H \cdot U_k$ is finite, so $\stabilizer_H(U_k)$ has finite index in $H$. Thus $\bigcap_{i=1}^k \stabilizer_G(U_i) = \stabilizer_H(U_k)$ acts cocompactly on $Q$.
\end{proof}

\begin{proof}[Proof of Theorem \ref{thm:higher-rank}]
    First normalise $(G,\mathfrak S)$.
    Then there exists $C \geq 1$ such that every bounded domain of $\mathfrak S$ has diameter at most $C$. Given $g \in G$, let $\bigset(g) \coloneqq \{U \in \mathfrak S : \pi_U(\la g \ra)$ has unbounded diameter$\}$. Recall that the action of $G$ on itself is hierarchically semisimple.

    By assumption, there exists a set $\mathfrak V \subseteq \mathfrak S$ of $\nu$ pairwise orthogonal elements such that $\mathcal CV$ is unbounded for all $V \in \mathfrak V$. 
    Using \cite[Def. 1.1.(3)]{BehrstockHierarchically19}, we can assume without loss of generality that, if $V \in \mathfrak V$ and $V' \propnest V$, then $\mathcal CV'$ is bounded. 
    It follows from Lemma \ref{lem: nest-minimal intersections of product regions} that, for each $W \in \mathfrak V$, there exists $g \in \bigcap_{V \in \mathfrak V} \stabilizer_G(V)$ such that $W \in \bigset(g)$. Therefore \cite[Prop. 6.68]{Genevois:hyperbolicities} implies that there exists $g \in G$ such that $\mathfrak V = \bigset(g)$. Let $A \leq G$ be a highest abelian subgroup which contains a finite index subgroup of $\la g \ra$. Then Corollary \ref{cor:highest-abelian} implies that $A$ is HQC and therefore has rank $\nu$.
\end{proof}

\appendix\newsection{Chasing constants}\label{app:constants}
Throguhout, we used some results from the literature on hierarchical hyperbolicity that take as input an HHS $(X,\mathfrak S)$ and yield some conclusion involving various quantities depending only on the HHS parameters for $(X,\mathfrak S)$.  This uniformity of constants is not always explicitly asserted in the versions of those statements in the literature, but can be extracted from their proofs.  In this appendix, we make the dependence of the constants explicit in these statements.

\subsection{Strong distance formula}\label{appsubsec:strong-DF}
The definition of an HHS is designed to enable one to prove several results that then serve as the main tools; most notable is the \emph{distance formula} from \cite{BehrstockHierarchically19}, of which special cases appear in \cite{MasurMinsky:II,SistoProjections13,BehrstockHierarchically17}.  In the statement, we use \emph{threshold} notation: given $A,T\in\reals$,  let $\ignore{A}{T} \coloneqq A$ if $A\geq T$ and $0$ if $A<T$.

\begin{thm}[Strong distance formula]\label{thm:distance-formula}
	For all $\lambda\geq 1$, the following holds.  Let $(X,\mathfrak S)$ be a hierarchically hyperbolic space and let $h:[0,\infty)\to[0,\infty)$ be such that $\lambda^{-1}r-\lambda \leq h(r)\leq \lambda r+\lambda$ for all $r\geq 0$.  Then there exists $T_0\geq 2\lambda$, depending only on $\lambda$ and the HHS parameters, such that for all $T\geq T_0$, there exists $\kappa\geq 0$, depending only on $T,\lambda,$ and the HHS parameters, such that 
	$$\kappa^{-1}\dist_X(x,y)-\kappa\leq \sum_{U\in\mathfrak S}\ignore{h(\dist_U(x,y))}{T} \leq \kappa\dist(x,y)+\kappa$$
	for all $x,y\in X$.
\end{thm}

\begin{proof}
	The case $\lambda=1$ is the original distance formula for HHSes, which is \cite[Thm. 4.5]{BehrstockHierarchically19}, whose proof makes clear that the threshold $T_0$ for $\lambda=1$ depends only on the HHS parameters and, for any $T\geq T_0$, one obtains $\kappa$ depending only on $T$ and the HHS parameters.  The general case follows from the case $\lambda=1$ by the exact same argument as is used in the proof of \cite[Thm. 2.9]{BehrstockCombinatorial24}; in fact, the latter theorem itself implies the present theorem as soon as we observe that the proof works for any $T$ chosen sufficiently large in terms of $\lambda$ and the ``threshold'' $T_0$ from the $\lambda=1$ case.
\end{proof}

The interested reader can refer to \cite[Rem. 12.10]{Casals-RuizReal24}, which is about an alternate proof due to Durham that also yields uniform constants.

\subsection{Cubical approximation theorem}\label{appsubsec:cubical-approx}
Next, we recall the cubical approximation theorem from \cite{BehrstockQuasiflats21}, of which there are several variants with different proofs, for instance in \cite{DurhamMinskySisto:stable,DurhamCubulating23,DurhamGeometry22,Casals-RuizReal24}.  The following version is the same as the statements from \cite{BehrstockQuasiflats21} and \cite{Casals-RuizReal24}, except it is explicit about what the constants depend on.

\begin{prop}[Cubical approximation]\label{prop:cubical_approximation}
	Let $(X,\mathfrak S)$ be a hierarchically hyperbolic space of complexity $\hhscomp$.  Then there exists $M_0$, depending on the HHS parameters only, such that for all $k\in\naturals$ there exists $C$, depending only on the HHS parameters and $k$, such that the following holds.  
	
	Let $M_1\coloneqq100M_0Ck$. Then for all $M\geq M_1$, there is a constant $Q\geq1$, depending on $M,k$, and the HHS parameters only, satisfying the following.  
	
	Let $A=\{x_1,\ldots,x_k\}\subset X$ and let $\mathfrak U$ be the set of $U\in\mathfrak S$ such that $\dist_U(x_i,x_j)\geq M$ for some $i,j$.  Then there exists a CAT(0) cube complex $\Box_A$ and a map $f:\Box_A\to X$ such that:
	\begin{itemize}
		\item \label{table:C} $f$ is a $(Q,Q)$--quasi-isometric embedding and $d_{Haus}(f(\Box_A),H_\theta(A))\leq Q$.
		\item There exist $\hat x_{i_1},\ldots,\hat x_{i_s}\in\Box_A$ such that $\dist_{X}(f(\hat x_{i_j}),x_j)\leq C$ for all $j$ and $\Box_A$ is equal to the convex hull in $\Box_A$ of $\{\hat x_{i_1},\ldots,\hat x_{i_s}\}$.  In particular, $\Box_A$ has finitely many cubes.
		\item The map $f$ is $Q$--quasimedian.
		\item Each hyperplane $h$ of ${\Box_A}$ is labelled by an element of $\mathfrak U\subset\mathfrak S$, and hyperplanes $h,h'$ cross if and only if their labels $U,U'$ satisfy $U\orth U'$.  In particular, $\dimension{\Box_A}\leq\hhscomp$.
	\end{itemize}
	Finally, if $\gamma$ is a combinatorial geodesic in ${\Box_A}$, then $f\circ\gamma$ is a $(D,D)$--hierarchy path in $X$, where $D$ depends only on $M,k$, and the HHS parameters.
\end{prop}

\begin{proof}
	The only difference between this proposition and \cite[Prop. 18.1]{Casals-RuizReal24} is that the current proposition says explicitly what the constants $M_0$ and $Q$ depend on.  As noted in \cite[Sec. 18]{Casals-RuizReal24}, everything in the statement (except for the claim about the dependency of constants) is given by \cite[Thm. 2.1]{BehrstockQuasiflats21}, except for the final assertion about hierarchy paths.  Therefore, we will first explain why the proof of \cite[Thm. 2.1]{BehrstockQuasiflats21} gives effective constants, i.e., why $M_0$ can be chosen in terms of the HHS parameters, and $Q$ can be chosen in terms of $k$ and the HHS parameters.  Then we will verify the claim about hierarchy paths. \\
	
	\noindent\textbf{Constants in the proof from \cite{BehrstockQuasiflats21}.}  
    Now we trace through the proof of \cite[Thm. 2.1]{BehrstockQuasiflats21}, starting with the constructions in \cite[Sec. 2]{BehrstockQuasiflats21} before the theorem:
	\begin{enumerate}
		\item The finite trees $T_U$ are $(1,C)$--quasi-isometrically embedded in $\mathcal CU$, where $C\geq 1$ is a constant depending on $E$ and $k$.
		
		\item The constant $M$ in \cite{BehrstockQuasiflats21} is what we have here called $M_0C$.  In \cite{BehrstockQuasiflats21}, the set $\mathfrak U$ is chosen to consist of those $U\in\mathfrak S$ for which $\pi_U(A)$ has diameter at least $100M_0Ck$, in our notation.  So, once we have chosen $M_0$, the constant $M_1$ from our current statement is determined.  The constant $M_0$ is chosen by imposing various (finitely many) constraints as the proof proceeds, all in terms of the HHS parameters only; we will mention these as they arise.  The first is just that $M_0>E$, and hence $M>E$.
		
		\item The cardinality of $\mathfrak U$ is bounded in terms of $\diam_X(A)$ and the HHS parameters only (including $\hhscomp$), using the large link axiom, although what's needed is just that $\mathfrak U$ is finite.  (See the remark on \cite[p. 947]{BehrstockQuasiflats21}, or \cite[Lem. 13.5]{Casals-RuizReal24}.)
		
		\item The only constraints on constants arising in the construction of the walls in $H_\theta(A)$ (and hence the dual cube complex $\Box_A$, which is called $\mathcal Y$ in \cite{BehrstockQuasiflats21}) is that $M_0\geq E$.
		
		\item In Section 2.2 of \cite{BehrstockQuasiflats21}, there are three lemmas, Lemmas 2.3--2.5, whose hypotheses require that (in our notation) $M_0\geq 10E$, which is the most restrictive hypothesis in any of these lemmas.
		
		\item Lemma 2.6 introduces a constant $\tau=\tau(M,k)$.  The proof of this lemma makes it explicit that it is sufficient to choose $\tau=50Mk(k-2)$.
		
		\item Lemma 2.7 introduces a constant $\eta$; the statement says that this depends only on $M,k$, and the HHS $X$, but the dependency on $X$ is only via the HHS parameters; specifically, the proof shows that $\eta=100(Mk+C+2E)$ works.
		
		\item In the proof of Theorem 2.1 itself, the choice of $\eta$ determines another constant $\xi$, which depends only on $\eta$ and the HHS parameters; specifically, $\xi=\eta r_0$, where $r_0$ is the realisation theorem constant (Theorem \ref{thm:hhs_realisation}) discussed above.
		
		\item Lemma 2.8 does not influence the constants, and the constant $N$ from Lemma 2.9 depends only on an HHS parameter, the complexity, via Ramsey's theorem.
		
		\item Lemma 2.10 provides a constant $T$ depending on $M,k,\xi,\tau$ and the HHS $(X,\mathfrak S)$.  The proof of Lemma 2.10 again makes it clear that the dependence on $(X,\mathfrak S)$ is only via dependence on the HHS parameters (not on the specific choice of space $X$ etc.).  This uses \cite[Lem. 1.6]{BehrstockQuasiflats21}, the ``passing up large projections'' lemma, in which the output $N_0$ depends on the input and on the HHS parameters (in particular, the constants in the large link axiom and the complexity) only.
		
		\item Lemma 2.12 does not involve constants.  The proof of Lemma 2.11 makes clear that the constant implicit in the word ``coarsely'' in the statement depends on $E$ and $\theta$ and hence only on the HHS parameters.
		
		\item Lemma 2.13 provides a quasimedian constant.  One must read the proof to check that this depends only on the HHS parameters and $k$, but the key facts are that the constants $\tau$ and $\eta$ depend only on those things, and that the distance bounds involved in the construction of the coarse median map $\mu$ (see Section \ref{sec:coarse-median} above) also depend only on the HHS parameters.
		
		\item Finally, we examine the proof of Theorem 2.1 itself.  The constants listed just above appear in that proof, and are bounded in terms of the HHS parameters and $M,k$; hence the same is true of any quantity computed from those constants via some procedure that only depends on the HHS parameters. In particular, the constant $C_1'$ in the proof of Theorem 2.1 is obtained by evaluating the uniqueness function, and the hierarchical quasiconvexity function $h$, on uniform inputs, and is thus uniform.
		
		The constant $C_1''$ is a distance formula quasi-isometry constant, with uniform threshold (depending on $E$ and $M$), which is therefore uniform (see Theorem \ref{thm:distance-formula}).  The constant $C_1'''$ depends in a uniform way on the output of Lemma 2.6, i.e. on $\tau$, and is therefore uniform.  Finally, the constant $C_1''''$ is the uniform constant from Lemma 2.13.  
	\end{enumerate}
	This establishes that $M_0$ depends only on the HHS parameters, $C$ depends only on the HHS parameters and $k$, and $Q$ depends on the HHS parameters, $k$, and $M$.
	
	It just remains to prove the statement about hierarchy paths.  As noted in \cite[Rem. 18.2]{Casals-RuizReal24}, this follows from the fact that $f$ is a $(Q,Q)$--quasimedian quasi-isometric embedding together with Proposition \ref{prop:hierarchy-path-char}.  That $D$ depends only on $M,k$, and the HHS parameters follows from the latter proposition and the fact, established above, that $Q$ is uniform.
\end{proof}

Proposition \ref{prop:cubical_approximation} immediately yields the following (which is also proven in \cite[Sec. 7]{BehrstockHierarchically19} without using cubical approximation):

\begin{cor}\label{cor:coarse-median}
	Let $(X,\mathfrak S)$ be an HHS.  Let $\mu:X^3\to 2^X$ be the coarse map from Construction \ref{cons:coarse-median}.  Then $\mu$ is a coarse median in the sense of Definition \ref{defn:coarse-median}, and the coarse median parameters depend only on the HHS parameters.
\end{cor}

\subsection{Non-canonical cone-off}\label{appsubsec:cone-off}
The following proposition is almost exactly the same as Proposition~19.1 in \cite{Casals-RuizReal24} except for the clause about HHS parameters.

\begin{prop}[Equivariant cone-off, with parameters]\label{prop:cone-off-QI}
	Let $(X,\mathfrak S)$ be an HHS, let $A\leq \Isom(X)$ be a group of HHS automorphisms, and let $\mathfrak U\subseteq\mathfrak S$ be an $A$--invariant downward-closed set.  Let $(\oldcone X,\widehat{\dist})$ be the non-canonical cone-off from Definition \ref{defn:old-cone-off}. Then the pair $(\oldcone X,\mathfrak S-\mathfrak U)$ is a hierarchically hyperbolic space, whose HHS parameters are bounded above in terms of the HHS parameters of $(X,\mathfrak S)$ only.  Moreover, the $A$--action on $(X,\mathfrak S)$ induces an action of $A$ on $(\oldcone X,\mathfrak S-\mathfrak U)$ by HHS automorphisms for which the $A$--action on $(X,\widehat{\dist})$ is an action by $(C,C)$--quasi-isometries, where $C$ depends only on the HHS parameters of $(X,\mathfrak S)$.
\end{prop}

\begin{proof}
	This is \cite[Prop. 19.1]{Casals-RuizReal24} (in turn a variant of \cite[Prop. 2.4]{BehrstockAsymptotic17}) except we need to justify the statements about the HHS parameters of $(\oldcone X,\mathfrak S-\mathfrak U)$ and $C$.  As noted in the proof of \cite[Prop. 19.1]{Casals-RuizReal24}, the $A$--action is by quasi-isometries, by the distance formula for the HHS $(\oldcone X,\mathfrak S-\mathfrak U)$.  Hence $C$ depends only on the HHS parameters of $(\oldcone X,\mathfrak S-\mathfrak U)$, so it suffices to prove the latter just depend on the HHS parameters for $(X,\mathfrak S)$.
	
	This follows by examining the proof that $(\oldcone X,\mathfrak S-\mathfrak U)$ is in HHS.  In \cite[Sec. 2]{BehrstockAsymptotic17}, the proof of Proposition 2.4 shows that all of the ``new'' HHS parameters except for the uniqueness function are bounded in terms of the ``old'' ones by a straightforward observation.  For example, the underlying set of $\oldcone X$ is just $X$, and for $U\in\mathfrak S-\mathfrak U$, the projection $X\to \mathcal CU$ is the same coarse map $\pi_U$ as in the original HHS structure.  So the consistency, bounded geodesic image, large links, and partial realisation axioms, which do not mention the metric on $X$, hold unchanged for the new HHS structure.  The same is true of the axioms that just concern the relations $\nest,\orth,\transverse$ on $\mathfrak S$ (such as finite complexity).  The first place where the difference between $\dist$ and $\widehat{\dist}$ matters is in the proof that $\pi_U:\oldcone X\to\mathcal CU$ is coarsely lipschitz, but the constant is computed in \cite{BehrstockAsymptotic17} in terms of the HHS parameters of $(X,\mathfrak S)$.  
	
	The remaining, and most complicated, thing to check is that the uniqueness function for $(\oldcone X,\mathfrak S-\mathfrak U)$ is bounded above in terms of the HHS parameters for $(X,\mathfrak S)$.  Here we follow the proof in \cite[Proposition~19.1]{Casals-RuizReal24}, which uses Proposition \ref{prop:cubical_approximation} and yields explicit estimates.	Tracing through that proof, the inputs are:
	\begin{itemize}
		\item The uniqueness function $\theta_u$ for $(X,\mathfrak S)$.
		\item The constant $\theta$, which depends only on $(X,\mathfrak S)$, given by \cite[Prop. 6.15]{BehrstockHierarchically19}.
		\item The HHS constant $E$.
		\item Constants $M_0,M_1,M$ that depend only on the HHS parameters of $(X,\mathfrak S)$ and are provided by Proposition \ref{prop:cubical_approximation}, with $k=2$.  
		\item A constant $N=N(M)$ which depends only on $M$ and the HHS parameters, via \cite[Proposition~13.1]{Casals-RuizReal24}, essentially a consequence of the large link and finite complexity axioms for $(X,\mathfrak S)$.
	\end{itemize}
	In the remainder of the proof, we are given an arbitrary $\kappa\geq 0$ and points $x,y\in\oldcone X$ with $\dist_U(x,y)\leq \kappa$ for all $U\in\mathfrak S-\mathfrak U$.  The goal is to bound $\widehat{\dist}(x,y)$ in terms of $\kappa$ and the HHS parameters of $(X,\mathfrak S)$.  First, Claim 12 in the proof of \cite[Proposition~19.1]{Casals-RuizReal24} produces a constant $\nu$ that depends only on $\kappa$ and the HHS parameters of $(X,\mathfrak S)$.  Next, the constant $N$ in Claim 13 is the one mentioned above.  In the final part of the proof, induction on the complexity of the original HHS yields a constant $\nu^1$, explicitly stated to depend only on the original HHS parameters and $\kappa$, and the estimate $\widehat{\dist}(x,y)\leq N(\nu+\nu^1)$.  This shows that the new uniqueness function is bounded in terms of the old HHS parameters, as required.
\end{proof}

\subsection{Coarse injectivity}\label{subsec:coarsely-injective-appendix}
We recall the following theorem from \cite[Theorem~A]{HaettelCoarse23}:

\begin{thm}\label{thm:coarse-injective}
	Let $(X, \mathfrak{S})$ be an HHS. There is a metric $\sigma$ on $X$ and constants $M_1, M_2$ such that $(X,\sigma)$ is $M_1$--coarsely injective and $M_2$--quasi-isometric to $(X,d)$, and any subgroup $H\leq \Isom(X,\dist)$ acting by HHS automorphisms on $(X,\mathfrak S)$ acts by isometries on $(X,\sigma)$. Moreover the constants $M_1, M_2$ depend only on the HHS parameters.
\end{thm}

\begin{proof}
	Using Lemma \ref{lem:assume-graph}, we can assume that $X$ is a graph with the usual combinatorial metric.  The theorem is now exactly the content of \cite[Theorem~A]{HaettelCoarse23}; we just need to go through their proof and check that the constants depend only on the HHS parameters.
	
	The definition of $\sigma$ requires an arbitrary choice of positive constant; we choose 1. 
	\begin{enumerate}
		\item By Proposition \ref{prop:cubical_approximation}, there exists $\kappa$, depending only on the HHS parameters, such that: for any pair of points $x,y \in X$, there exists a finite CAT(0) cube complex $Q$ with dimension at most $\hhscomp$ and a $\kappa$--quasimedian $(\kappa, \kappa)$--quasi-isometry $\lambda:Q \rightarrow H_\theta(x,y)$, with $\theta$ as in Lemma \ref{lem: hulls are quasiconvex}. There exists $\kappa'\geq 1$, depending only on $\kappa$ and $h_\theta(0)$, such that $\lambda$ has a quasi-inverse which is a $(\kappa',\kappa')$--quasi-isometry and is $\kappa'$--quasi-median.
		\item By \cite[Lemmas~2.18 and 2.19]{NibloFour19}, there exists $H_5 \geq 0$, which depends only the coarse median parameters of $X$, and hence, by Corollary \ref{cor:coarse-median}, only depends on the HHS parameters, such that:
		\begin{align*}
			&d(\mu(a,b,\mu(x,y,z)), \mu(\mu(a,b,x), \mu(a,b,y), z)) \leq  H_5, \\
			&d(\mu(a,b,\mu(x,y,z)), \mu(a,b,x), \mu(a,b,y), \mu(a,b,z)) \leq H_5
		\end{align*}
		for all $a,b,x,y,z \in X$.
		\item Let $L \coloneqq \max\{1 + h(0), 2 + H_5\}$. Then, according to the proof of \cite[Proposition~2.16]{HaettelCoarse23}, we can take $M_2 = \hhscomp {\kappa'}^2 L$.
		\item By \cite[Lemma~2.22]{HaettelCoarse23}, there exists $\varepsilon \geq 0$ such that, for any $x,y,z \in X$ such that $x = \mu(x',y,z)$ for some $x' \in X$, we have $d(x,\mu(x,y,z)) \leq \varepsilon$. 
		One can choose $\varepsilon \coloneqq r_0' + r_0 + h(0)(r_0' + 1)$, where $r_0, r_0'$ are constants provided by \cite[Lemma~8.1(2)]{BowditchQuasiflats19} which only depend on the coarse median parameters of $X$, which in turn only depend on the HHS parameters, again by Corollary \ref{cor:coarse-median}.
		
		\item Let 
		\[
		C_\sigma' \coloneqq 64 (1 + \kappa) \hhscomp + 4 M_2 (1 + \kappa) + 4 \kappa + 4 + 2(4h(0) + 4).
		\]
		Then $(X,\sigma)$ is $C_\sigma'$--weakly roughly geodesic \cite[Prop.~2.21]{HaettelCoarse23}.
		\item For $M_3 \coloneqq \varepsilon + 2 + 3C_\sigma'$, balls in $(X,\sigma)$ are $M_3$--median quasiconvex \cite[Lem.~2.23]{HaettelCoarse23}.
		\item There exists $k:[0, \infty) \rightarrow [0,\infty)$, depending only on $M_3$ and the HHS parameters, such that balls in $(X,\sigma)$ are $k$--hierarchically quasiconvex \cite[Prop.~5.11]{RussellConvexity23}.
		\item For all $r \geq 0$, there exists $R(r) \geq 0$, depending only on $r$ and $E$, such that the following holds. If $\mathcal{Q}$ is a collection of bounded $k$--hierarchically quasiconvex subspaces of $X$ which are $r$--close (with respect to $d$), then there is a point $x \in X$ with $d(x,Q) \leq R(r)$ for each $Q \in \mathcal{Q}$ \cite[Theorem~3.5]{HaettelCoarse23}. 
		\item Following the proof of \cite[Corollary~3.6]{HaettelCoarse23}, let us check that $(X,\sigma)$ is coarsely injective with constants only depending on the HHS parameters. Let $\{B_\sigma(x_i,r_i) : i \in I\}$ be a family of balls in $(X, \sigma)$ such that $\sigma(x_i, x_j) \leq r_i + r_j$ for all $i,j \in I$. Since $(X,\sigma)$ is weakly roughly geodesic with constant $C_\sigma'$, the balls $\{ B_\sigma(x_i, r_i + 2C_\sigma') : i \in I\}$ intersect pairwise. For each $i \in I$, let $B_i$ be the image of $B_\sigma(x_i, r_i + 2C_\sigma')$ under the identity $M_2$--quasi-isometry $(X,\sigma) \rightarrow (X,d)$. Each $B_i$ is $M_3$--coarsely median convex and therefore $k$--hierarchically quasiconvex. Thus there is a point $x \in X$ such that $d(x,B_i) \leq R(0)$ for each $i \in I$. So $(X,\sigma)$ is coarsely injective with constant $M_1 \coloneqq M_2 (R(0) + C_\sigma') + M_2$. \qedhere
	\end{enumerate}
\end{proof}

\bibliographystyle{alpha}
\hypertarget{References}{\bibliography{Biblio}}

\end{document}